%% file: bifur_capture_numerical_FoC_Revision_clean.tex
\RequirePackage{fix-cm}
\RequirePackage{snapshot} 
\documentclass[smallextended]{svjour3}       
\smartqed  
\usepackage{graphicx}
\usepackage{amsmath}
\usepackage{amssymb} 

\usepackage{xcolor}
\usepackage{graphicx}

\usepackage[colorlinks=true, linkcolor=blue]{hyperref}

\hypersetup{
pdftitle={Preservation of bifurcations of Hamiltonian boundary value problems under discretisation},
pdfauthor={Christian Offen},
pdfsubject={Bifurcation Theory, symplectic geometry, geometric integration boundary value problem},
pdfkeywords={bifurcation, geometric integration, boundary value problems, symplectic geometry, Lagrangian, Hamiltonian systems, integrable systems, periodic orbits, invariant tori, pitchfork, saddle-node, singularity theory, catastrophe theory},
bookmarksdepth=10,
}

\newcommand{\Hess}{\mathrm{Hess}\, }

\renewcommand{\d}{\mathrm{d}}
\newcommand{\p}{\partial }

\newcommand{\D}{\mathrm{D}}

\newcommand{\Id}{\mathrm{Id}}
\newcommand{\e}{\epsilon}

\newcommand{\Z}{\mathbb{Z}}

\newcommand{\R}{\mathbb{R}}

\spdefaulttheorem{prop}{Proposition}{\upshape}{\itshape}
\spdefaulttheorem{observation}{Observation}{\itshape}{\upshape}

 \journalname{Foundations of Computational Mathematics}
\begin{document}

\title{Preservation of bifurcations of Hamiltonian boundary value problems under discretisation
}


\titlerunning{Preservation of bifurcations of Hamiltonian boundary value problems}        

\author{Robert I McLachlan         \and
        Christian Offen%
\protect\footnote{corresponding author\\  Communicated by Hans Munthe-Kaas} 
}

\authorrunning{Robert I McLachlan \and Christian Offen} 

\institute{R. McLachlan \at
              School of Fundamental Sciences\\Massey University\\Palmerston North, New Zealand\\
             ORCiD: 0000-0003-0392-4957\\
              \email{r.mclachlan@massey.ac.nz}           
           \and
           C. Offen \at
              School of Fundamental Sciences\\Massey University\\Palmerston North, New Zealand\\
             	ORCiD: 0000-0002-5940-8057\\
			\email{c.offen@massey.ac.nz}
}


\maketitle

\begin{abstract}
We show that symplectic integrators preserve bifurcations of Hamiltonian boundary value problems and that nonsymplectic integrators do not. We provide a universal description of the breaking of umbilic bifurcations by nonysmplectic integrators. We discover extra structure induced from certain types of boundary value problems, including classical Dirichlet problems, that is useful to locate bifurcations. Geodesics connecting two points are an example of a Hamiltonian boundary value problem, and we introduce the jet-RATTLE method, a symplectic integrator that easily computes geodesics and their bifurcations. Finally, we study the periodic pitchfork bifurcation, a codimension-1 bifurcation arising in integrable Hamiltonian systems. It is not preserved by either symplectic on nonsymplectic integrators, but in some circumstances symplecticity greatly reduces the error.
\end{abstract}

\input{intro_bratu}

 \input{introduction}

\input{broken_bifur}
 \input{dirichlet_extra_structure}

\input{pitchfork_capture}

\input{summary_outlook}

\begin{acknowledgements}
We thank Peter Donelan, Bernd Krauskopf, and Hinke Osinga for useful discussions. This research was supported by the Marsden Fund of the Royal Society Te Ap\={a}rangi.
\end{acknowledgements}

\bibliographystyle{spmpsci}      
\bibliography{resources}   

\appendix
\setcounter{section}{0}

\input{RATTLE_jet_method}

\input{RK_break_numerical_experiment}

%
%

\end{document}

%% file: intro_bratu.tex
\section{Motivation and introduction}

Symplectic integrators can be excellent for Hamiltonian initial value problems. Reasons for this include
their preservation of invariant sets like tori, good energy behaviour, nonexistence of attractors,
and good behaviour of statistical properties. These all refer to {\em long-time} behaviour. They are
directly connected to the dynamical behaviour of symplectic maps $\varphi\colon M\to M$
on the phase space under iteration. Boundary value problems, in contrast, are posed for fixed (and often quite short)
times. Symplecticity manifests as a symplectic map $\varphi\colon M\to M’$ which is not iterated.
Is there any point, therefore, for a symplectic integrator to be used on a Hamiltonian boundary value problem? 
In this paper we show that symplectic integrators preserve bifurcations of Hamiltonian boundary value problems and that nonsymplectic integrators do not.

\subsection{The Bratu problem - an example of Hamiltonian boundary value problem}


A reaction-diffusion model for combustion processes is given by the PDE
\begin{equation}\label{BratuPDE}u_t = u_{xx}+Ce^u, \qquad u(t,0)=0=u(t,1),\end{equation}
with parameter $C > 0$. 
Finding steady-state solutions $x \mapsto u(x)$ of \eqref{BratuPDE}, i.e.\ solving the Dirichlet problem
\begin{equation}\label{eq:BratuODE}
u_{xx}+Ce^u=0, \quad u(0)=0=u(1)
\end{equation}
and analysing their bifurcation behaviour is known as the Bratu problem \cite{Mohsen201426}. For small, positive $C$ there are two solutions which undergo a fold bifurcation as $C$ increases (figure \ref{fig:BratuSol}).
\begin{figure}
\begin{center}
\includegraphics[width=0.4\textwidth]{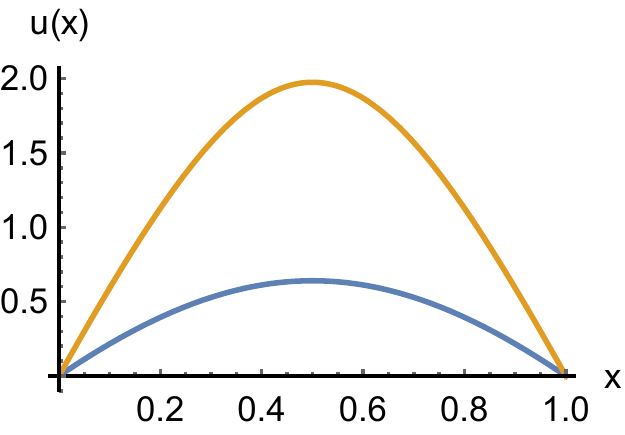}
\quad
\includegraphics[width=0.4\textwidth]{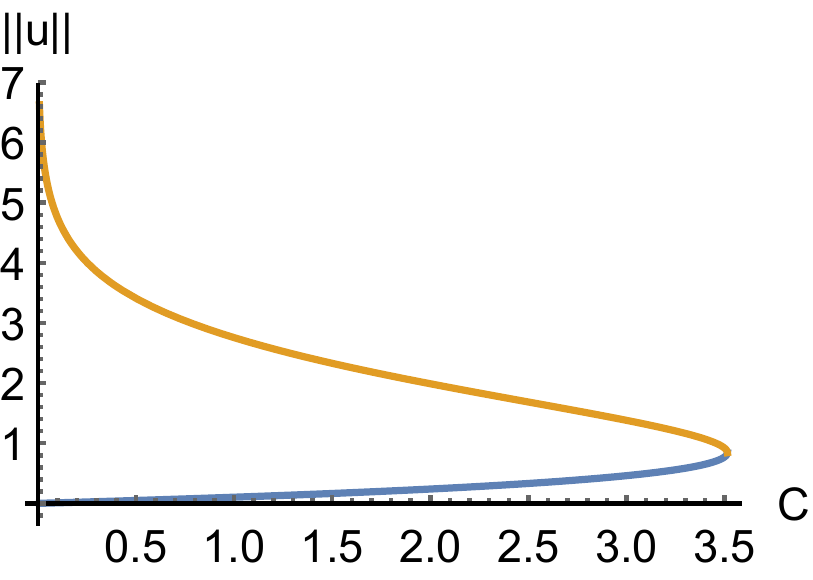}
\end{center}
\caption{Left plot: solutions to the Bratu problem \eqref{eq:BratuODE} for $C=3$. Right plot: bifurcation diagram of the Bratu problem \eqref{eq:BratuODE} showing a fold bifurcation at $C\approx 3.51$. }\label{fig:BratuSol}\label{fig:BratuBifur}
\end{figure}
A Hamiltonian formulation of \eqref{eq:BratuODE} is given as the first order system
\begin{align}\label{eq:HamEq}
\dot q &= \; \; \, \nabla_p H(q,p) = p\\ \nonumber
\dot p &= -\nabla_q H(q,p) = - C e^q
\end{align}
together with the boundary condition $q(0)=0=q(1)$ for the Hamiltonian
\begin{equation}\label{eq:HamBratu}
H(q,p)=\frac 12 p^2 + C e^q
\end{equation}
defined on $T^\ast\R \cong \R^2$.
The boundary value problem is visualised in figure \ref{fig:BratuPhase}.
A Hamiltonian motion, i.e.\ a solution curve to \eqref{eq:HamEq}, solves the boundary value problem $q(0)=0=q(1)$ if it starts on the line $\Lambda=\{(0,p)\,|\, p \in \R\}$ and returns to $\Lambda$ after time 1. 
For $0<C <C^\ast$ two such solutions are illustrated as black curves starting at $\times$ and ending at $o$.

\begin{figure}
\begin{center}
\includegraphics[width=0.4\textwidth]{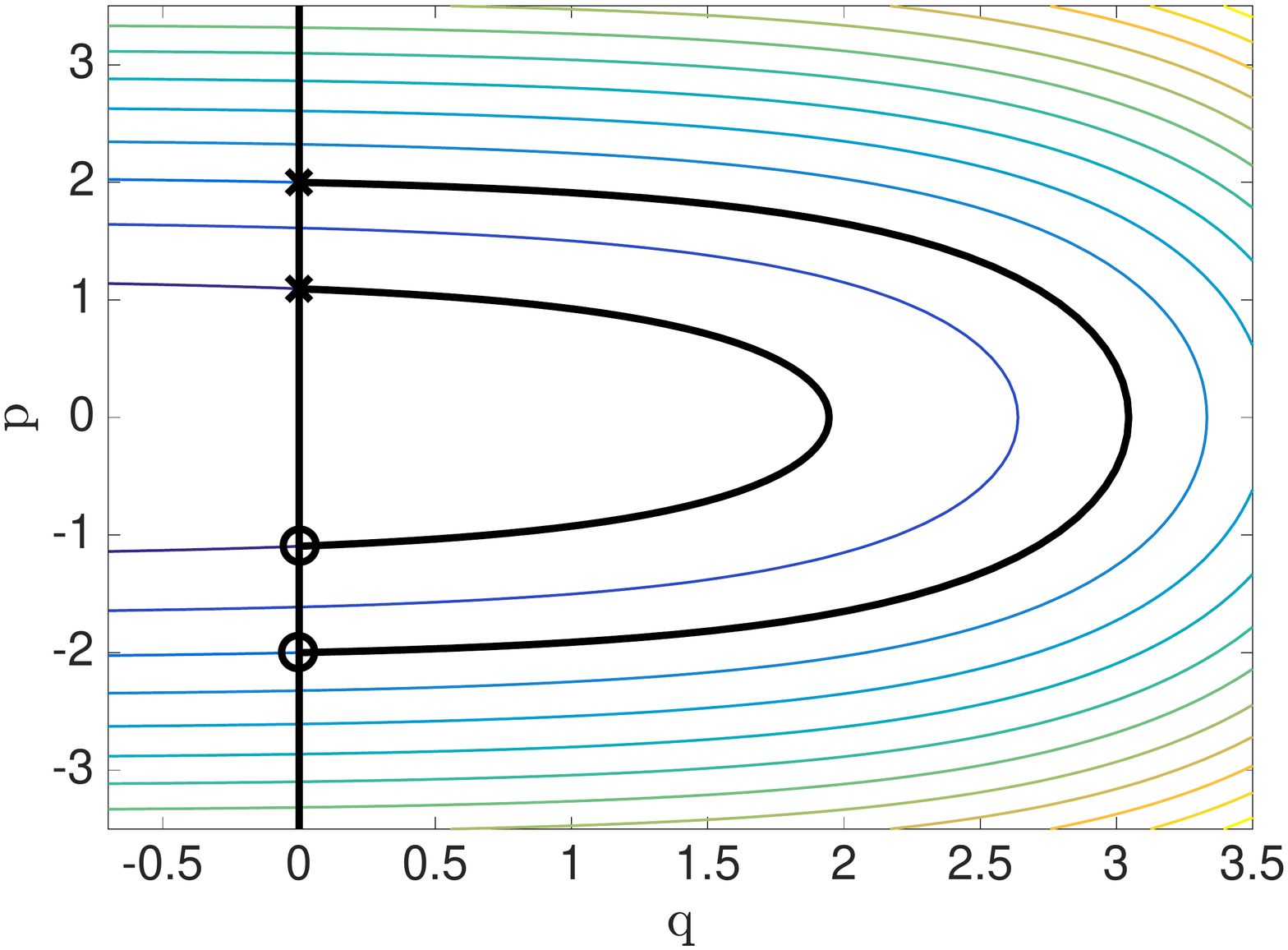}
\end{center}
\caption{Illustration of the Bratu problem \eqref{eq:BratuODE} as a boundary value problem for the Hamiltonian system defined by \eqref{eq:HamBratu}.}\label{fig:BratuPhase}
\end{figure}

\subsection{Purpose of the paper}

We would like to

\begin{itemize}
\item
understand which bifurcations occur in Hamiltonian boundary value problems
\item
and develop strategies to capture these numerically.
\end{itemize}

To attack the first task, the authors have linked all generic bifurcations occurring in smooth parameterised families of Lagrangian boundary value problems for symplectic maps with catastrophe theory in \cite{bifurHampaper}. The Bratu problem and classical Dirichlet-, Neumann-, Robin- boundary value problems are instances of this class.
A treatment with a focus on the symplectic geometrical picture is \cite{offen2018}.
The purpose of this paper is to present results the authors obtained for the second objective.

To draw valid conclusions from numerical results one has to make sure that the bifurcations in the boundary value problem for the exact flow are still present after discretisation.
This is important for two reasons: bifurcations of high codimension act as \textit{organising centres} (see \cite[Part I, Ch.7]{Gilmore1993catastrophe}) in the bifurcation diagram. A high codimensional bifurcation determines which lower codimensional bifurcations occur in a neighbourhood of the singular point.
It is, therefore, desirable to capture these correctly. Furthermore, bifurcation diagrams are typically calculated using \textit{continuation methods}: a branch of low codimensional bifurcations is followed numerically to find a bifurcation of higher codimension but these can only be detected correctly if they are not broken in the discretised boundary value problem.
This means, preservation of the bifurcation behaviour is a goal in its own right but also crucial for computations.


Symplecticity in Hamiltonian boundary value problems does not seem to have been addressed in the literature, even in very detailed numerical studies like \cite{Beyn2007,GalanVioque2014}.
The AUTO software \cite{AutoManual} is based on Gauss collocation, which is symplectic when the equations are presented in canonical variables. The two-point boundary-value codes MIRKDC \cite{MIRKDC} and TWPBVP and TWPBVPL \cite{TWPBVPL} are based on non-symplectic Runge-Kutta methods. MATLAB’s {\tt bvp4c} uses 3-stage Lobatto IIIA \cite{KierzenkaBvP4c}, which is not symplectic. Note that symplectic integration sometimes requires the use of implicit methods. For initial value problems, these are typically computationally more expensive than explicit methods. But for boundary value problems solved in the context of parameter continuation, this distinction largely disappears as excellent initial approximations are available.

Hamiltonian structure can be understood as the existence of a variational principle. In \cite{PDEbifur} the authors use the variational viewpoint to extend ideas of this paper to a PDE setting.

\subsection{Introduction of Lagrangian boundary value problems and their connection to catastrophe theory}

Let us recall results from \cite{bifurHampaper} and provide more examples of Lagrangian boundary conditions.

%% file: introduction.tex

\begin{definition}[Lagrangian boundary value problem for a symplectic map]\label{def:LagBVP}
Consider a symplectic map $\phi\colon (M,\omega) \to ( M' ,  \omega')$ and projections $\pi \colon M\times  M \to M$ and $ \pi' \colon M\times  M' \to  M'$. Define the symplectic form $\omega \oplus (- \omega'):= \pi^\ast\omega - {\pi'}^\ast \omega'$ on the manifold $M \times  M'$. 
The graph $\Gamma$ of $\phi$ constitutes a Lagrangian submanifold in the symplectic manifold $(M \times  M', \omega \oplus (- \omega'))$. 
Let $\Pi$ be another submanifold in $(M \times  M', \omega \oplus (-\omega'))$. The intersection of $\Gamma$ with $\Pi$ is called a \textit{Lagrangian boundary value problem (for $\phi$)} if and only if $\Pi$ is a Lagrangian submanifold.
\end{definition}


\begin{example} If $(M,\omega)=( M',  \omega')$ than the periodic boundary value problem $\phi(z) = z$ is a Lagrangian boundary value problem. Here $\Pi=\{(m,m)\, |\, m \in M\}$ is the diagonal.\end{example}

\begin{example}
Let $M$ be an open subset of $\R^{2n}$ with the standard symplectic form $\omega=\sum_{j=1}^n \d x^j \wedge \d y_j$ and a symplectic map $(x,y)\mapsto\phi(x,y)$ on $M$ with $x=(x^1,\ldots,x^n)$-component $\phi^X = x\circ \phi$. Fix $x^\ast,X^\ast \in \R^n$. The boundary value problem
\begin{equation}\label{eq:Dirichletforphi}
\phi^X(x^\ast,y) = X^\ast
\end{equation}
is a Lagrangian boundary value problem.
\end{example}

\begin{example}[Dirichlet-/Neumann-/Robin problems]\label{ex:DNRbdflow}
The classical Dirichlet problem
\begin{equation}\label{eq:DiribdcondODE}
u(t_0)=x^\ast, \quad u(t_1)=X^\ast
\end{equation}
with $t_0<t_1 \in \R$, $x^\ast,X^\ast \in \R^n$ for the second order ordinary differential equation
\begin{equation}\label{eq:2ndorderODE}
\ddot u(t) = \nabla_u G(t,u(t))
\end{equation}
with a scalar-valued map $G$ defined on a sufficiently large subdomain of $\R \times \R^n$ can be formulated as in \eqref{eq:Dirichletforphi}: let us rewrite \eqref{eq:2ndorderODE} as the first order problem
\begin{align}\label{eq:1storderODE}
\dot u(t) &= v(t) \\ \nonumber
\dot v(t) &=\nabla_u G (t,u(t))
\end{align}
and denote by $\phi$ the map which sends a point $(x,y)$ taken from a subdomain of $\R^{2n}$ to the solution of the initial value problem \eqref{eq:1storderODE} with initial data
\[
u(t_0)=x, \quad v(t_0)=y
\]
(assuming existence and uniqueness). The map $\phi$ is symplectic because it arises as the Hamiltonian flow of
\begin{equation}\label{eq:mechanicalHam}
H(t,x,y)=\frac 12 \|y\|^2 - G(t,x).
\end{equation}
Now the problem \eqref{eq:2ndorderODE} with Dirichlet boundary conditions \eqref{eq:DiribdcondODE} corresponds to \eqref{eq:Dirichletforphi}. The boundary value problem can be visualised as in the plot to the left of figure \ref{fig:DirichletFish}. The manifold $\Pi$ is $\Pi=\{x^\ast\} \times \R \times \{X^\ast\} \times \R$.
\begin{figure}
\begin{center}
\includegraphics[width=0.32\textwidth]{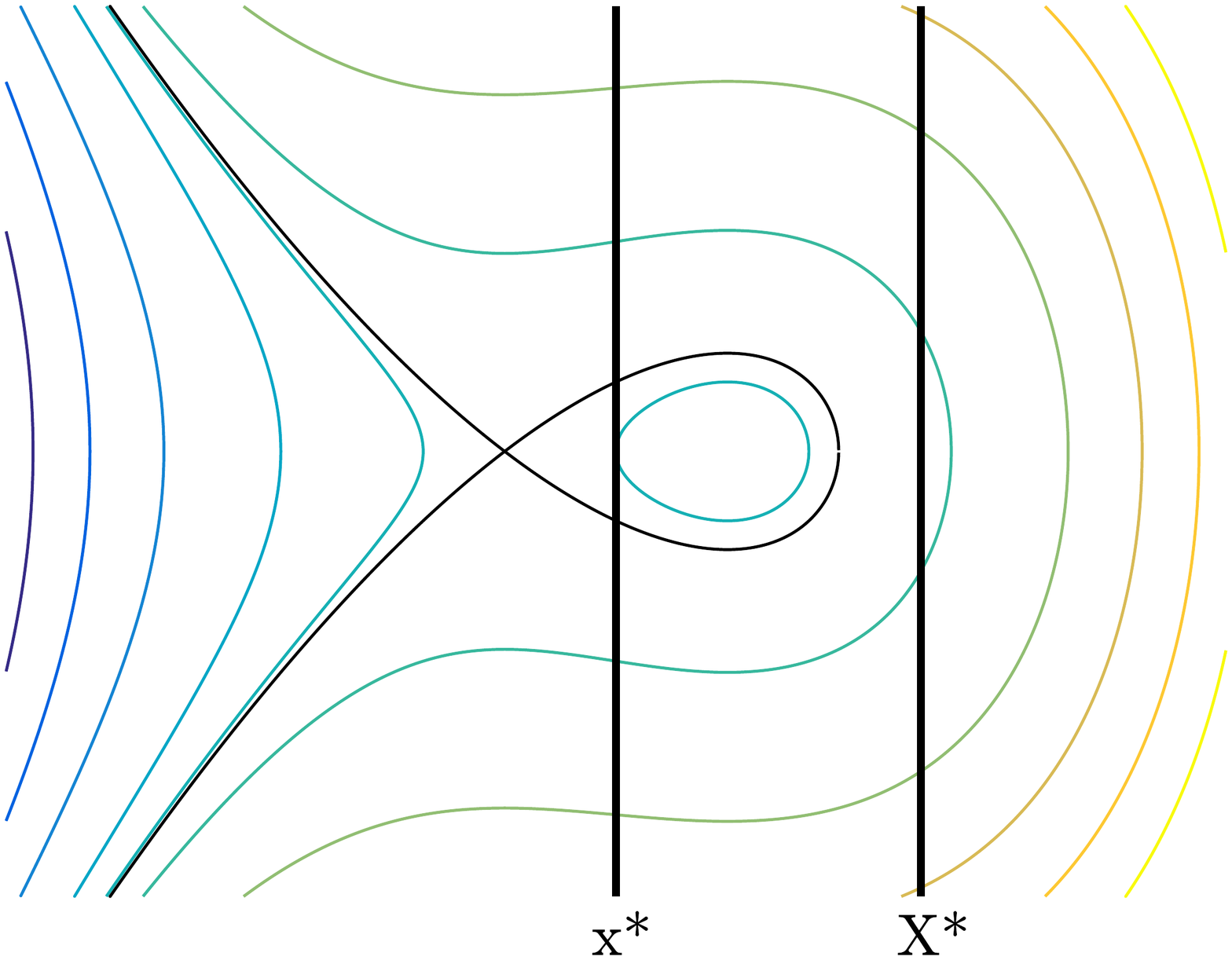}
\includegraphics[width=0.32\textwidth]{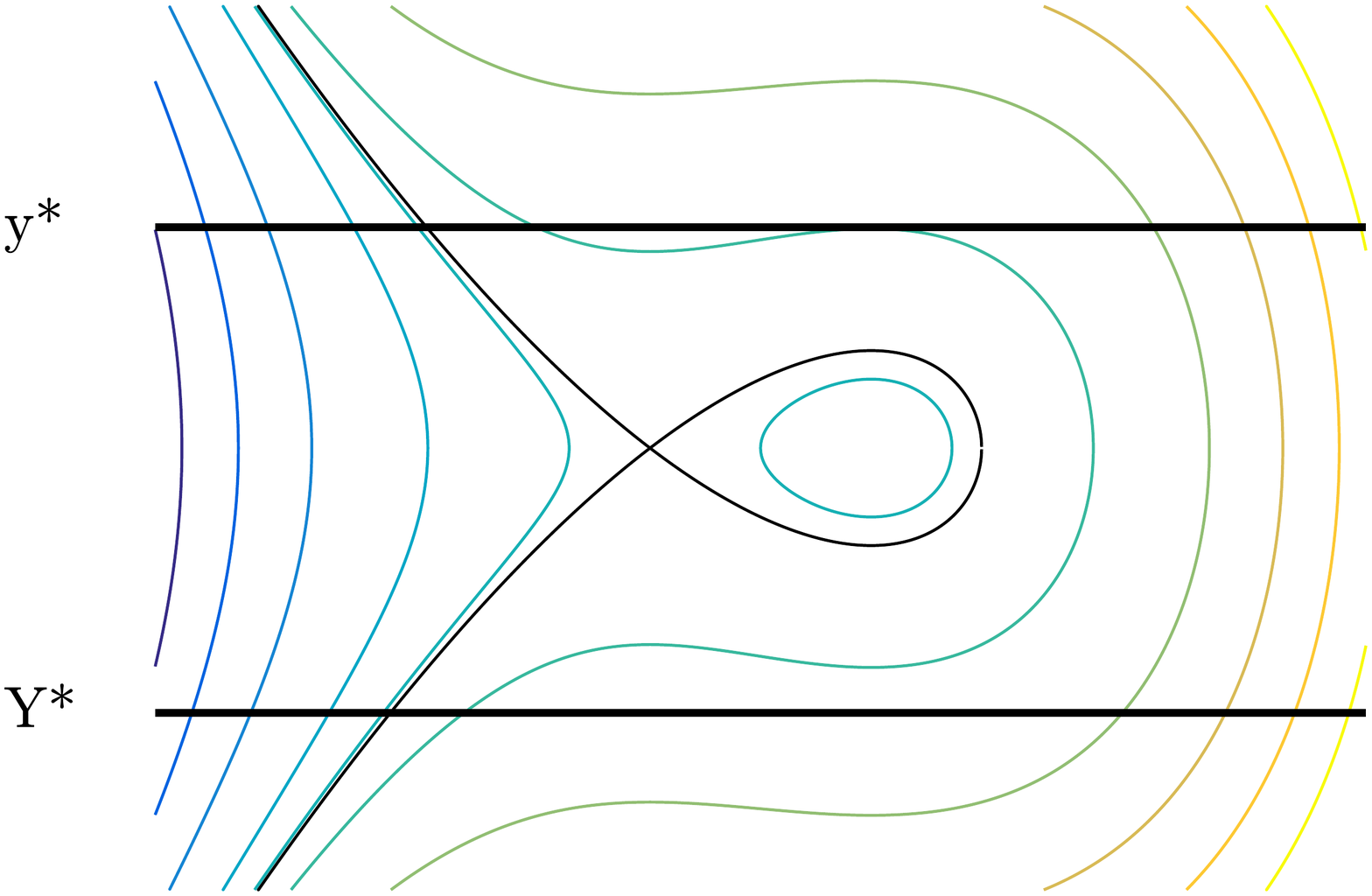}
\includegraphics[width=0.32\textwidth]{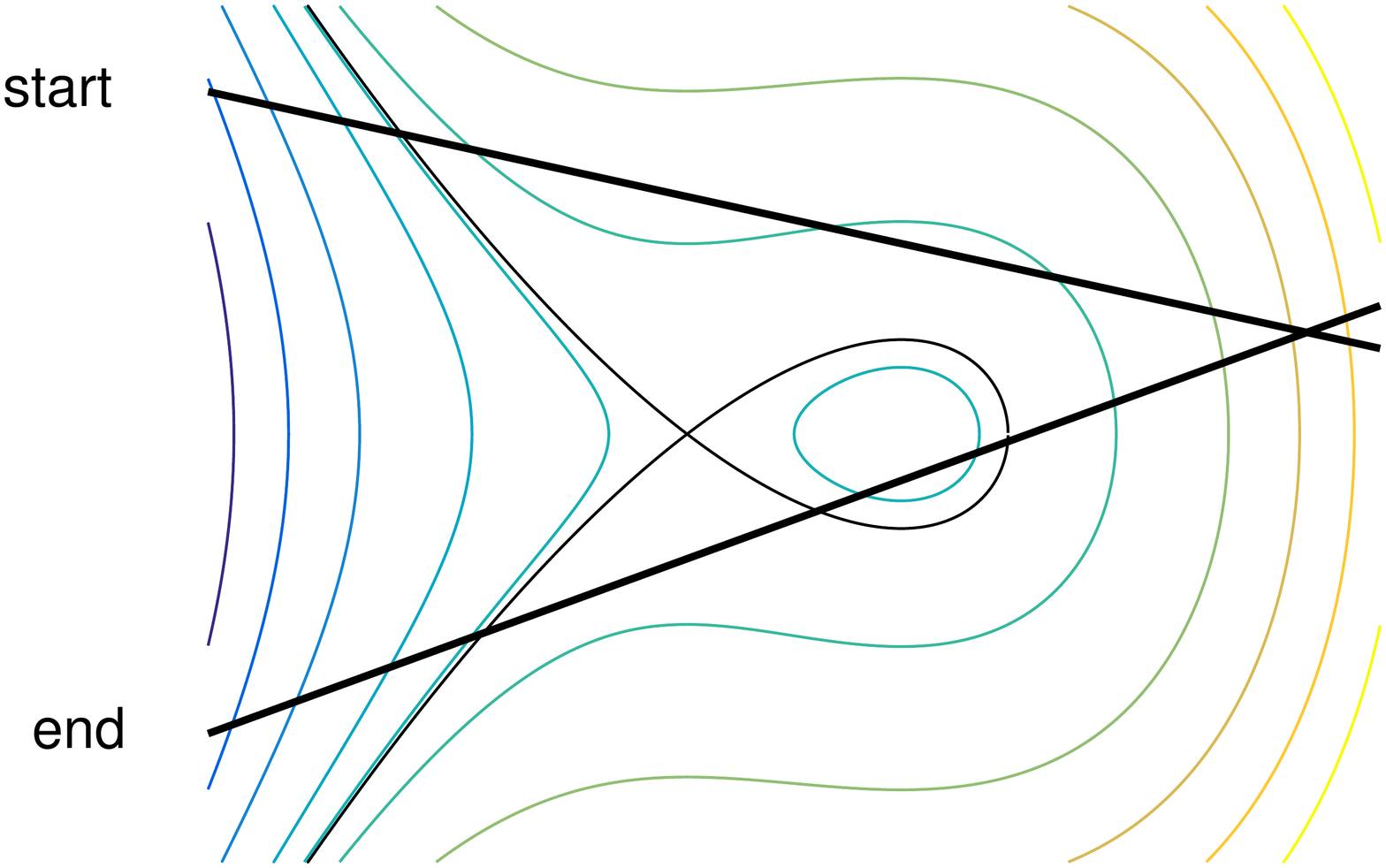}
\end{center}
\caption{Illustration of Dirichlet- \eqref{eq:DiribdcondODE}, Neumann- \eqref{eq:NeumancondODE} and Robin- \eqref{eq:RobinBD} boundary conditions for the ODE \eqref{eq:2ndorderODE} with the potential $G(t,u)=\tilde G(u)=-u^3-\lambda u$ as a Lagrangian boundary value problem for the time $t_1-t_0$-flow map of the Hamiltonian system \eqref{eq:mechanicalHam}. A motion solves the boundary value problem if it starts on the line $x^\ast\times \R$ and ends on $X^\ast\times \R$ after time $t_1-t_0$ (Dirichlet), or starts on $\R\times y^\ast$ and ends on $\R\times Y^\ast$ (Neumann) or starts on the start-line and ends on the end-line (Robin).}\label{fig:DirichletFish}\label{fig:NeumannFish}\label{fig:RobinFish}
\end{figure}
Analogously, a Neumann problem 
\begin{equation}\label{eq:NeumancondODE}
\dot u(t_0)=y^\ast, \quad \dot u(t_1)=Y^\ast
\end{equation}
with $y^\ast,Y^\ast \in \R^n$
corresponds to the Lagrangian boundary value problem
\begin{equation}\label{eq:Neumannforphi}
\phi^Y(x,y^\ast) = Y^\ast,
\end{equation} 
where $\phi^Y = y \circ \phi$ is the $y=(y_1,\ldots,y_n)$-component of $\phi$. The plot in the centre of figure \ref{fig:NeumannFish} shows a visualisation of the boundary condition in the phase space of the Hamiltonian system. Here the manifold $\Pi$ is $\Pi = \R \times \{y^\ast\}\times \R \times \{Y^\ast\}$.
Moreover, if we denote the $j$-th component of a solution $u$ to \eqref{eq:2ndorderODE} by $u^j$ then Robin-type boundary conditions
\begin{align}\label{eq:RobinBD}
u^j(t_0)+\alpha_0^j \dot u^j(t_0) &= \beta_0^j\\ \nonumber
u^j(t_1)+ \alpha_1^j \dot u^j(t_1) &= \beta_1^j, \qquad j \in \{1,\ldots,n\}
\end{align}
for $\alpha_0^j,\alpha_1^j,\beta_0^j,\beta_1^j\in \R$ are Lagrangian boundary conditions. A visualisation is shown in the plot to the right in figure \ref{fig:RobinFish}. Here $\Pi = \{(x,y,X,Y)\, | \, x^j+\alpha_0^j y^j = \beta_0^j,\, X^j+ \alpha_1^j Y^j = \beta_1^j\}$.




\end{example}

\begin{example}[Non-example] As before, let $\omega=\sum_{j=1}^n \d x^j \wedge \d y_j$ be the standard symplectic form on $\R^{2n}$. Consider a symplectic map $\phi \colon (\R^{2n},\omega) \to (\R^{2n},\omega)$ and the problem
\[
\phi(x,y)=(y,x).
\]
This problem does \textit{not} fulfil the definition of a Lagrangian boundary value problem because it corresponds to the intersection of the Lagrangian manifold $\Gamma$ with the non-Lagrangian submanifold $\Pi = \{(x,y,y,x)\, | \, (x,y) \in \R^{2n}\}$.
\end{example}


Parameter dependent changes of critical points in smooth families of smooth maps
\[
g_\mu\colon \R^k \to \R, \quad z \mapsto g_\mu(z),\]
are called \textit{critical-points-of-a-function} problems or \textit{gradient-zero} problems. They have been treated under the headline \textit{catastrophe theory} and generic bifurcation phenomena have been classified. 

In \cite{bifurHampaper} the authors use generating functions of the considered symplectic maps and the Lagrangian boundary conditions to assign local critical-points-of-a-function problems to local Lagrangian boundary value problems. An introduction to generating functions can be found in \cite[VI.5]{GeomIntegration}. The idea is illustrated in the following example.

\begin{example}\label{ex:bvpFold}
Consider a symplectic map $(X,Y) = \phi_\mu(x,y)$ on $\R^{2n}$ with a generating function of the form 
\begin{equation}\label{eq:generatingFunctionExample}
\begin{split}
x &= \phantom{-}\nabla_y S_\mu(y,Y)\\
X &= -\nabla_Y S_\mu(y,Y).
\end{split}
\end{equation}
The 2-point boundary value problem $x=0$, $X=0$ corresponds to the critical-points problem $\nabla S_\mu(y,Y)=0$. For a concrete example consider the family of functions \[S_\mu(y,Y) = y_1^3+\mu y_1 + \sum_{j=1}^n(Y_j+y_j)^2 + \sum_{j=2}^n y_j^2.\] The matrix $\left(\frac{\p^2 S_\mu}{\p y_i \p Y_j}\right)_{i,j = 1,\ldots,n}$ is invertible such that $S_\mu$ induces a family of symplectic maps $(X,Y) = \phi_\mu(x,y)$ via \eqref{eq:generatingFunctionExample}. By construction, $S_\mu$ is a generating function for $\phi_\mu$. The problem $x=0$, $X=0$ corresponds to $\nabla S_\mu(y,Y)=0$ which is equivalent to
\[
\nabla (y_1^3+\mu y_1) =0, \qquad y_2,\ldots,y_n,Y_2,\ldots,Y_n =0, \qquad y_1 = -Y_1.
\]
In other words, solutions to the boundary value problem correspond to critical points of the function $y_1 \mapsto y_1^3+\mu y_1$. See figure \ref{fig:FoldNormal} for an illustration.
\end{example}

\begin{figure}
\begin{center}
\includegraphics[width=0.4\textwidth]{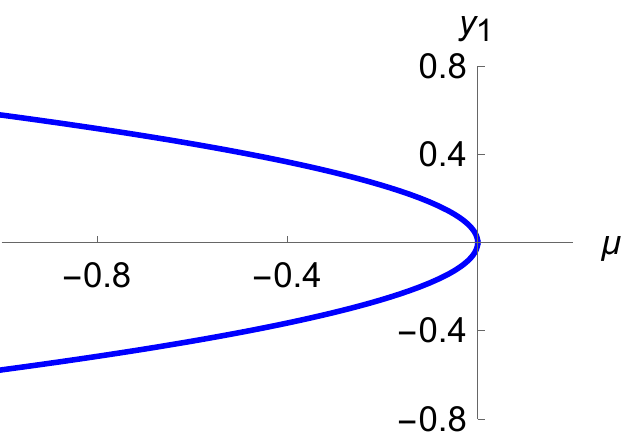}
\end{center}
\caption{The boundary value problem described in example \ref{ex:bvpFold} undergoes a {\em fold bifurcation} as the parameter $\mu$ is varied, i.e.\ two solution branches annihilate each other. At $(\mu,y_1)=(0,0)$ there is a fold point. }\label{fig:FoldNormal}
\end{figure}

A more general way of establishing the correspondence is as follows.
\begin{itemize}
\item
To any intersection point $p$ of the graph $\Gamma$ and the boundary condition $\Pi$ there exists an open neighbourhood $U$ of $p$ and a symplectomorphism $\Psi$ which maps $U$ to an open neighbourhood in the cotangent bundle $\pi \colon T^\ast\Pi \to \Pi$ such that $\Pi \cap U$ is mapped into the zero section of $T^\ast\Pi$ and the restriction of $\pi$ to $\Psi(\Gamma \cap U)$ is an injective immersion.
\item
The Lagrangian manifold $\Psi(\Gamma \cap U)$ is the image of $\pi(\Psi(\Gamma \cap U))$ under a 1-form $\beta$ on $\pi(\Psi(\Gamma \cap U))$ which can be considered as a map $\pi(\Psi(\Gamma \cap U)) \to  T^\ast\Pi$.
\item
The manifold $\Psi(\Gamma \cap U)$ is Lagrangian. Therefore, $\beta$ admits a primitive $S$, defined on an open neighbourhood of $\Psi(p)$ in the zero section of $T^\ast\Pi$. 
\item
Critical points of $S$ correspond to solutions of the Lagrangian boundary value problem.
\end{itemize}

A classification due to Thom building on the work of Whitney asserts that all generic singularities in the critical-points-of-a-function problem with no more than 4 parameters are stably right-left equivalent to one of the seven elementary catastrophes given in table \ref{tab:ThomCatastrophes}.
%
%
%
%
%
\begin{table}
\begin{center}
  \begin{tabular}{| c | l | c | r | c |}
    \hline
	ADE class & name & germ & miniversal unfolding \\
    \hline
    \hline
$A_2$ &fold & $x^3$ & $x^3+ \mu_1 x$ \\
$A_3$ &cusp & $x^4$ & $x^4+\mu_2 x^2+ \mu_1 x$ \\
$A_4$ &swallowtail & $x^5$ & $x^5+\mu_3 x^3+\mu_2 x^2+ \mu_1 x$  \\
$A_5$ &butterfly & $x^6$ & $x^6+\mu_4 x^4+\mu_3 x^3+\mu_2 x^2+ \mu_1 x$  \\
\hline 
$D_4^+$ &hyperbolic umbilic & $x^3+xy^2$ & $ x^3+xy^2 + \mu_3 (x^2-y^2) +\mu_2 y + \mu_1 x$  \\
$D_4^-$ &elliptic umbilic & $x^3-xy^2$ & $ x^3-xy^2 + \mu_3 (x^2+y^2) +\mu_2 y + \mu_1 x$  \\
$D_5$ &parabolic umbilic & $x^2y+y^4$ & $ x^2y+y^4 +\mu_4 x^2  + \mu_3 y^2 +\mu_2 y + \mu_1 x$  \\
 \hline
  \end{tabular}
\caption{Thom's seven elementary catastrophes \cite[p.89]{lu1976singularity},\cite[p.66]{Gilmore1993catastrophe}. Generically$^2$
, a family of functions $g_\mu\colon \R^k \to \R$ locally around zero coincides up to parameter-dependent changes of coordinates either with a quadratic form in $k$ variables or is of the form $\hat g_\mu(x) + Q(y,z_1,\ldots,z_{k-2})$ or $\hat g_\mu(x,y) + Q(z_1,\ldots,z_{k-2})$, where $\hat g$ is one of the miniversal unfoldings given in the table above and $Q$ is a quadratic form in the remaining variables provided that the parameter $\mu$ is at most 4-dimensional. Illustrations of the first three catastrophes can be found in figure \ref{fig:AseriesIllus}. Illustrations of the hyperbolic and elliptic umbilic appear later in the paper in figure \ref{fig:breakD4m} and \ref{fig:breakD4p}.}
\label{tab:ThomCatastrophes}
\end{center}
\end{table}
Well-presented lecture notes explaining the notion of right-left/stably equivalences, introducing the necessary algebraic tools and giving a proof of Thom's result are given in \cite{lu1976singularity}. A reference including the extended classification by Arnold is \cite{Arnold1}. 
An elementary reference with an extensive analyses of singularities of low multiplicity is \cite{Gilmore1993catastrophe}. 

\begin{figure}
\begin{center}
\includegraphics[width=0.2\textheight]{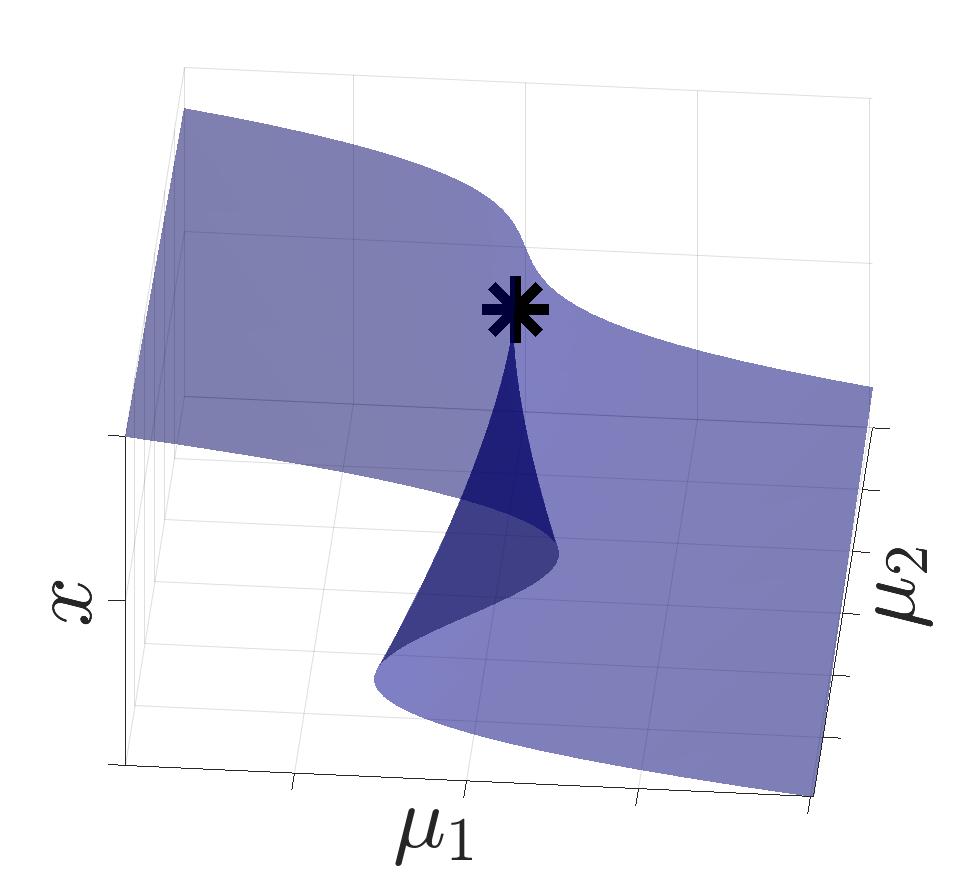}\;
\includegraphics[width=0.2\textheight]{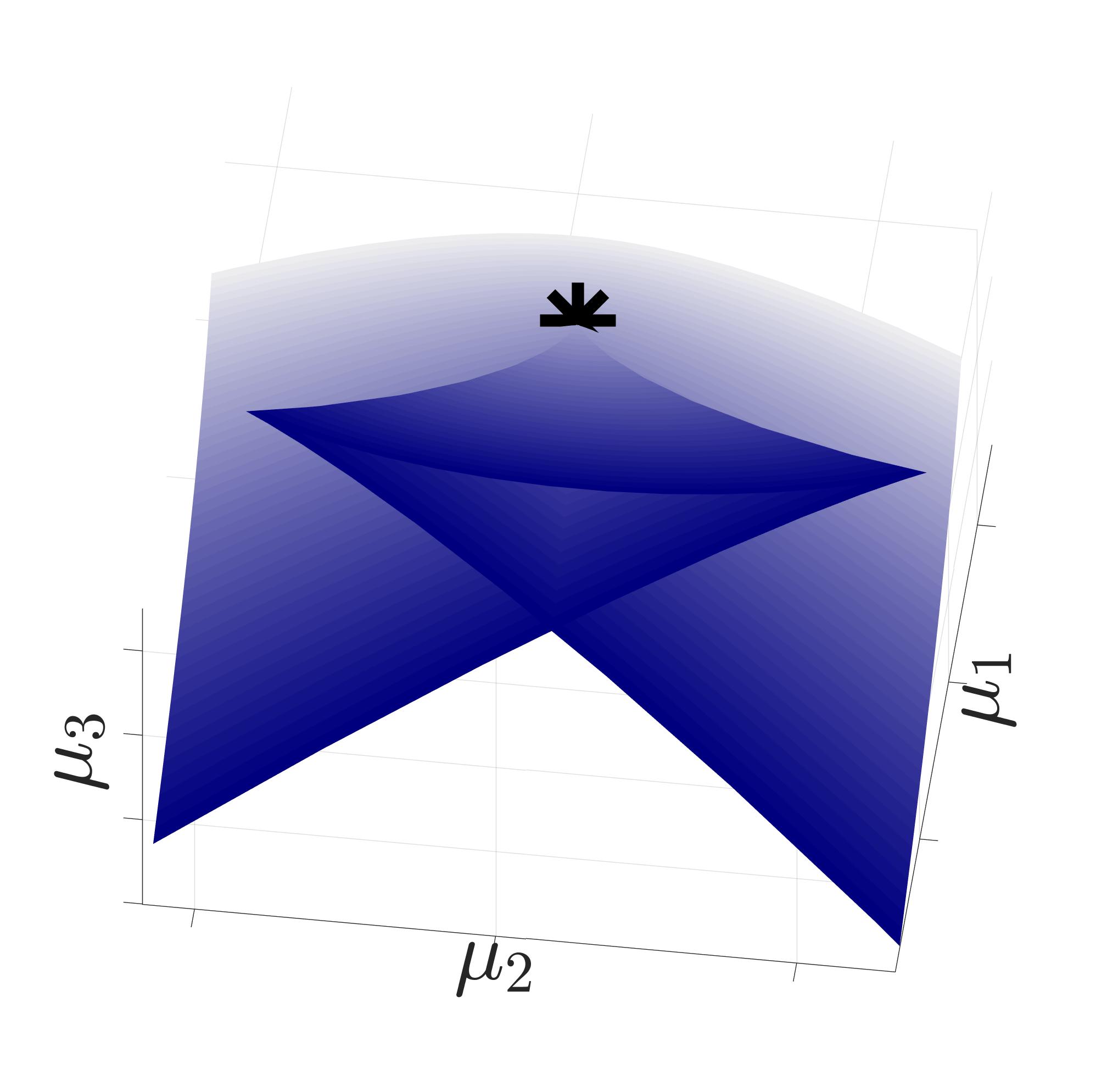}
\end{center}
\caption{The diagrams show the critical points of the miniversal unfoldings of the cusp ($A_3$) and the swallowtail bifurcation ($A_4$) (from left to right). The most singular points are denoted by $\ast$. For an illustration of a fold bifurcation see figure \ref{fig:FoldNormal}. The plot to the right shows a projection of a swallowtail bifurcation ($A_4$) to the parameter space. Each point in the sheet corresponds to a fold singularity, an intersection of sheets means that two fold singularities happen at the same parameter value but at different points $x$ in the phase space, points on the two edges correspond to cusp points. The point where the two edges join with the intersection line is a swallowtail point marked by $\ast$. See \cite{YoutubeASeries} for an animation of the cusp and swallowtail singularity.}\label{fig:AseriesIllus}
\end{figure}

The gradient-zero problem $\nabla g_{\mu}(z)=0$ is \textit{not} equivalent to the special case $k=l$ in the \textit{zeros-of-a-function} problem $F_\mu(z)=0$ for smooth families of smooth functions $F_\mu \colon \R^{k} \to \R^{l}$. In other words, the gradient structure has an effect on which bifurcations can occur generically, i.e.\ which bifurcations are persistent under small perturbations. The zeros-of-a-function problem is treated in \cite{TopoStability}.
In the lectures by Wall \cite{walllectures} the technique of unfolding singularities\footnote{truncated miniversal deformations in Arnold's nomenclature} in the gradient-zero problem and in the roots-of-a-function problem are treated within the same framework.

Our paper is structured as follows.

\addtocounter{footnote}{1}\footnotetext{The classification applies to a dense set w.r.t.\ the Whitney $C^\infty$-topology on the space of smooth function families $\R^m\times\R^k\to \R$, where $m \le 4$ is the dimension of the parameter $\mu$  \cite{TopoStability}.}

\begin{itemize}
\item
In section \ref{sec:brokenBifur} we analyse how generic bifurcations of the gradient-zero problem break if perturbed with a map which does not admit a primitive. This models the effect of using a non-symplectic integrator to solve Lagrangian boundary value problems in Hamiltonian systems.
\item
In section \ref{sec:DirichletExtra} we analyse separated Lagrangian boundary conditions which include classical Dirichlet-, Neumann- and Robin- boundary conditions.
\begin{itemize}
\item
We describe structure present in data which is computed when numerically solving separated Lagrangian boundary value problems. This can be helpful to locate bifurcations. As an example, a $D$-series bifurcation in a H{\'e}non-Heiles type Hamiltonian system is calculated numerically. 
\item
We explain the role of symplecticity in the computation of conjugate loci and illustrate how the RATTLE method can be used in this context.

\end{itemize}


\item
In section \ref{sec:capturePitchfork} we describe how periodic pitchfork bifurcations (introduced by the authors in \cite[Thm. 3.2]{bifurHampaper}) in completely integrable Hamiltonian systems can be captured numerically. This involves
\begin{itemize}
\item
a description of the bifurcation mechanism of the exponentially small broken pitchfork bifurcation in the numerical flow of a planar Hamiltonian system,
\item
the development of a non-trivial, analytical 4-dimensional model system with a periodic pitchfork bifurcation and numerical experiments showing that the bifurcation in the numerical flow is broken up to the order of the integrator,
\item
theoretical considerations showing that the pitchfork bifurcation is captured exponentially well by a symplectic integrator in important cases of completey integrable systems like planar systems or systems with affine-linear integrals of motion. In generic completely integrable systems, however, the bifurcation is captured up to the order of the integrator used to discretise Hamilton's equations.
\end{itemize}
\end{itemize}

%% file: broken_bifur.tex
\section{Broken gradient-zero bifurcations}\label{sec:brokenBifur}

In applications symplectic maps arise as flow maps of Hamiltonian systems (Hamiltonian diffeomorphisms). These can be discretised using different numerical integrators.

\begin{definition}[symplectic integrator]
A symplectic integrator assigns to a time-step-size $h>0$ (discretisation parameter) and a Hamiltonian system a symplectic map which approximates the time-$h$-map of the Hamiltonian flow of the system. 
\end{definition}

\begin{remark}
For a finite sequence of positive time-step-sizes $h_1,\ldots,h_N$ summing to $\tau$ the composition of all time-$h_j$-map approximations obtained by a symplectic integrator yields an approximation to the Hamiltonian-time-$\tau$-map, which is a symplectic map.
\end{remark}

\begin{remark}[non-symplectic integrator]
When using the term non-symplectic integrator (applied to a Hamiltonian system) we require the obtained approximations of Hamiltonian-time-$\tau$-maps to be non-symplectic maps on the phase space. This excludes non-generic examples where an approximation happens to be symplectic.
\end{remark}

The solutions to a family of Hamiltonian boundary value problems on $2n$-dimensional manifolds locally corresponds to the roots of a family of $\R^{2n}$-valued function defined on an open subset of $\R^{2n}$. For a Lagrangian Hamiltonian boundary value problems these maps are exact, i.e.\ each arises as the gradient of a scalar valued map \cite{bifurHampaper}. 
Consider a family of Hamiltonian Lagrangian boundary value problems
and consider an approximation of the Hamiltonian-time-$\tau$-map by an integrator. Roughly speaking, two map-families are (right-left-) equivalent if they coincide up to reparametrisation and parameter dependent changes of variables in the domain and target space.\footnote{One can consider different equivalence relations. Our discussion applies to right-left equivalence. It can also be applied to left-equivalence (finer) or any coarser notion like contact equivalence.}
If the family of maps corresponding to the approximated problems is equivalent to the family of maps for the exact problem then we say \textit{the integrator preserves the bifurcation diagram of the problem}. This means the computed bifurcation diagram qualitatively looks the same as the exact bifurcation diagram.


\begin{prop}\label{prop:capturebifur}
A symplectic integrator with any fixed (but not necessarily uniform) step-size, applied to any autonomous or nonautonomous Hamiltonian Lagrangian boundary value problem, preserves bifurcation diagrams of generic bifurcations of any codimension for sufficiently small maximal step-sizes.
\end{prop}


\begin{proof}[Proof of proposition \ref{prop:capturebifur}] In \cite{bifurHampaper} the authors establish the fact that all generically occurring singularities in Lagrangian boundary value problems for symplectic maps are non-removable under small symplectic perturbations of the map. The statement follows because a Hamiltonian diffeomorphism which is slightly perturbed by a symplectic integrator is a symplectic map near the exact flow map. 
\end{proof}


Proposition \ref{prop:capturebifur} implies that using a symplectic integrator to solve Hamilton's equations in order to solve a Lagrangian boundary value problem we obtain a bifurcation diagram which is qualitatively correct even when computing with low accuracy and not preserving energy.

In contrast, nonsymplectic integrators do not preserve all bifurcations, even for arbitrarily small step-sizes. However, they do preserve the simplest class of $A$-series bifurcations, i.e.\ folds, cusps, swallowtails, butterflies,....

\begin{prop}\label{prop:Aseriespersist} 
A symplectic or non-symplectic integrator with any fixed (but not necessarily uniform) step-size, applied to any autonomous or nonautonomous Hamiltonian Lagrangian boundary value problem, preserves bifurcation diagrams of generic $A$-series singularities for sufficiently small maximal step-sizes. However, each non-symplectic integrator breaks the bifurcation diagram of all generic $D$-series singularities for any positive maximal step-size.\footnote{More generally, a generic singularity $\nabla g(x^1,\ldots,x^n)$ of the exact flow is broken in the numerical flow of a non-symplectic integrator if and only if for a versal unfolding $(g_\mu)_\mu$ of $g$ the family $(\nabla g_\mu(x^1,\ldots,x^n))_\mu$ does not constitute a versal roots-of-a-function-type unfolding.}
\end{prop}

\begin{remark}
For the fold bifurcation in the Bratu problem (figure \ref{fig:BratuBifur}) the proposition says that any integrator with fixed step-size will capture the bifurcation correctly, i.e.\ the obtained bifurcation diagram will qualitatively look the same as figure \ref{fig:BratuBifur}.
\end{remark}


\begin{proof}[Proof of proposition \ref{prop:Aseriespersist}]
As explained in \cite{bifurHampaper}, solutions to Lagrangian Hamiltonian boundary value problems locally correspond to the roots of an $\R^{2n}$-valued function $F$ defined on an open subset of $\R^{2n}$ with $F$ arising as the gradient of a scalar valued map.
A smooth perturbation of the symplectic flow map corresponds to a smooth perturbation $\tilde F$ of $F$. 
The map $\tilde F$ has gradient structure if and only if the perturbation is symplectic. 
$A$-series bifurcations are stable in the roots-of-a-function problem and are, therefore, persistent under any smooth perturbation of $F$. 
This is not true for $D$-series bifurcations: there is no versal unfolding of the roots-of-a-function type singularity $\nabla g$ corresponding to the singularity $D^{\pm}_{k+2}$ ($k \ge 2$) represented by $g(x,y)=x^2y \pm y^{k+1}$  by exact maps as we will see from lemma \ref{lem:unfoldBseries}. Indeed, in proposition \ref{prop:DdeconstructA} we will prove that $D$-series bifurcations either decompose into $A$-series bifurcations or vanish.
\end{proof}



Let us take a closer look at the first two $D$-series bifurcations.
Denote the unfolding of the hyperbolic umbilic singularity $D_4^+$ with parameter $\mu$ by
\[g_\mu(x,y)=x^3+xy^2+\mu_3 (x^2+y^2)+\mu_2 y + \mu_1 x.\]
As the bifurcation diagram to the problem $\nabla g_\mu(x,y)=0$ is too high dimensional to visualise, the left plot in figure \ref{fig:breakD4p} shows the corresponding level bifurcation set, i.e.\ the set of points in the parameter space at which a bifurcation occurs, which is given as
\begin{equation}\label{eq:D4plevelbifur}
\{ \mu\in \R^3 \,|\, \exists (x,y)\in \R^2 : \nabla g_\mu(x,y)=0, \; \det \mathrm{Hess}\, g_\mu(x,y) =0\}.
\end{equation}
The plot to the right of figure \ref{fig:breakD4p} shows a perturbed version of the hyperbolic umbilic bifurcation, which is the set
\begin{equation}\label{eq:D4plevelbifurpert}
\{ \mu \in \R^3 \,|\, \exists (x,y)\in \R^2 : \nabla g_\mu (x,y) +f_{\epsilon}(x,y) =0, \; \det \D (\nabla g_\mu + f_{\epsilon})(x,y) =0\}
\end{equation}
for $\e \not=0$ near $0$ and a smooth family of maps $f_{\epsilon}\colon \R^2\to\R^2$ with $f_0=0$ such that $f_{\epsilon} \not= \nabla h_\e$ for any $h_\e \colon \R^2\to \R$ unless $\e =0$. Here $\D (\nabla g_\mu + f_{\epsilon})(x,y)$ denotes the Jacobian matrix of the map $(x,y) \mapsto  (\nabla g_\mu + f_{\epsilon})(x,y)$.

Each point in the sheets corresponds to a fold singularity ($A_2$) and points on edges to cusp singularities ($A_3$). At parameter values where the sheets self-intersect there are two simultaneous fold singularities in the phase space. In the unperturbed system two lines of simultaneous folds merge with a line of cusps to a hyperbolic umbilic point \cite[I.5]{Gilmore1993catastrophe}. In the perturbed picture the line of cusps decomposes into three segments and two swallowtail points ($A_4$) occur where two lines of cusps merge with a line of simultaneous folds. Notice that there are no swallowtail points in the unperturbed level bifurcation set.

\begin{figure}
\begin{center}
\includegraphics[width=0.49\textwidth]{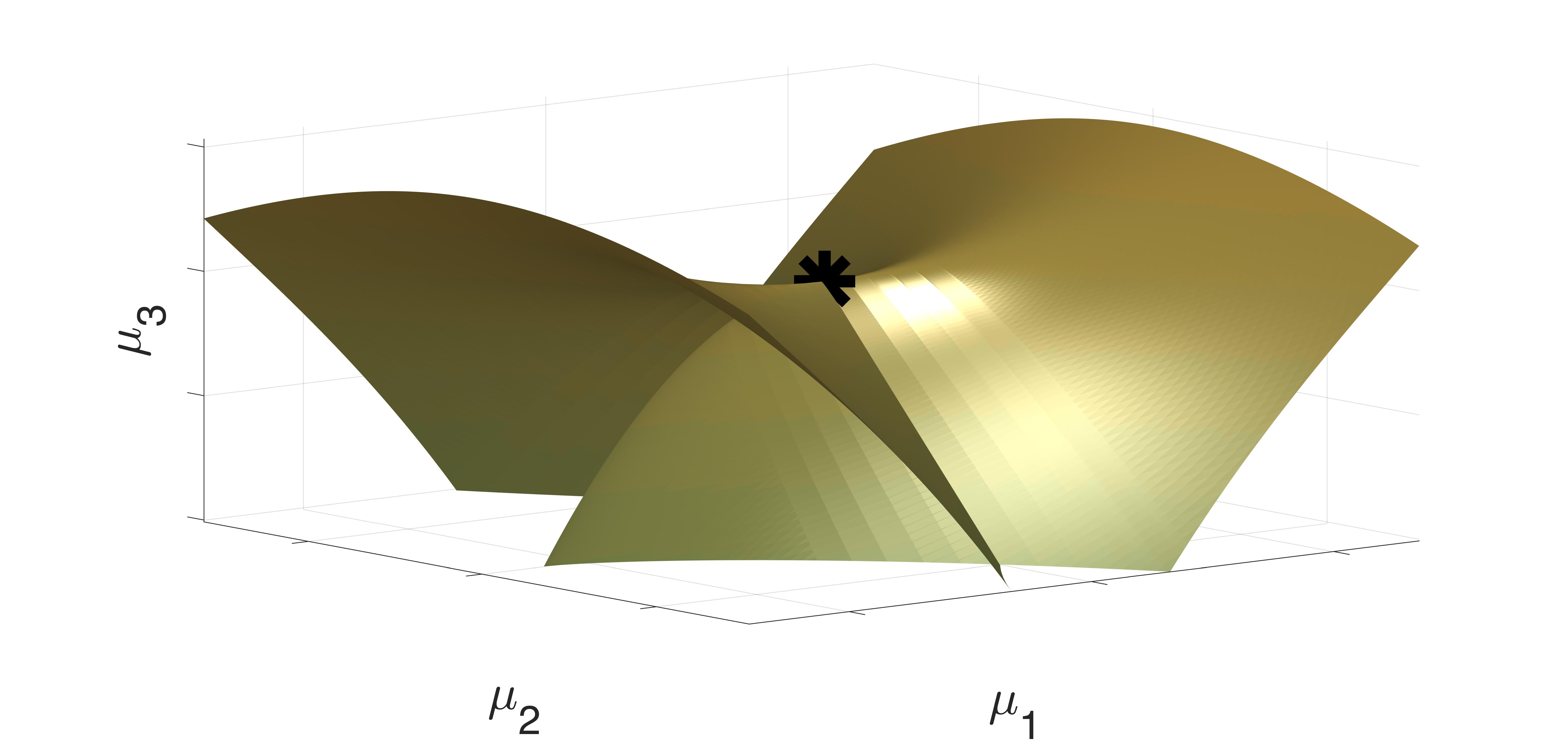}
\includegraphics[width=0.49\textwidth]{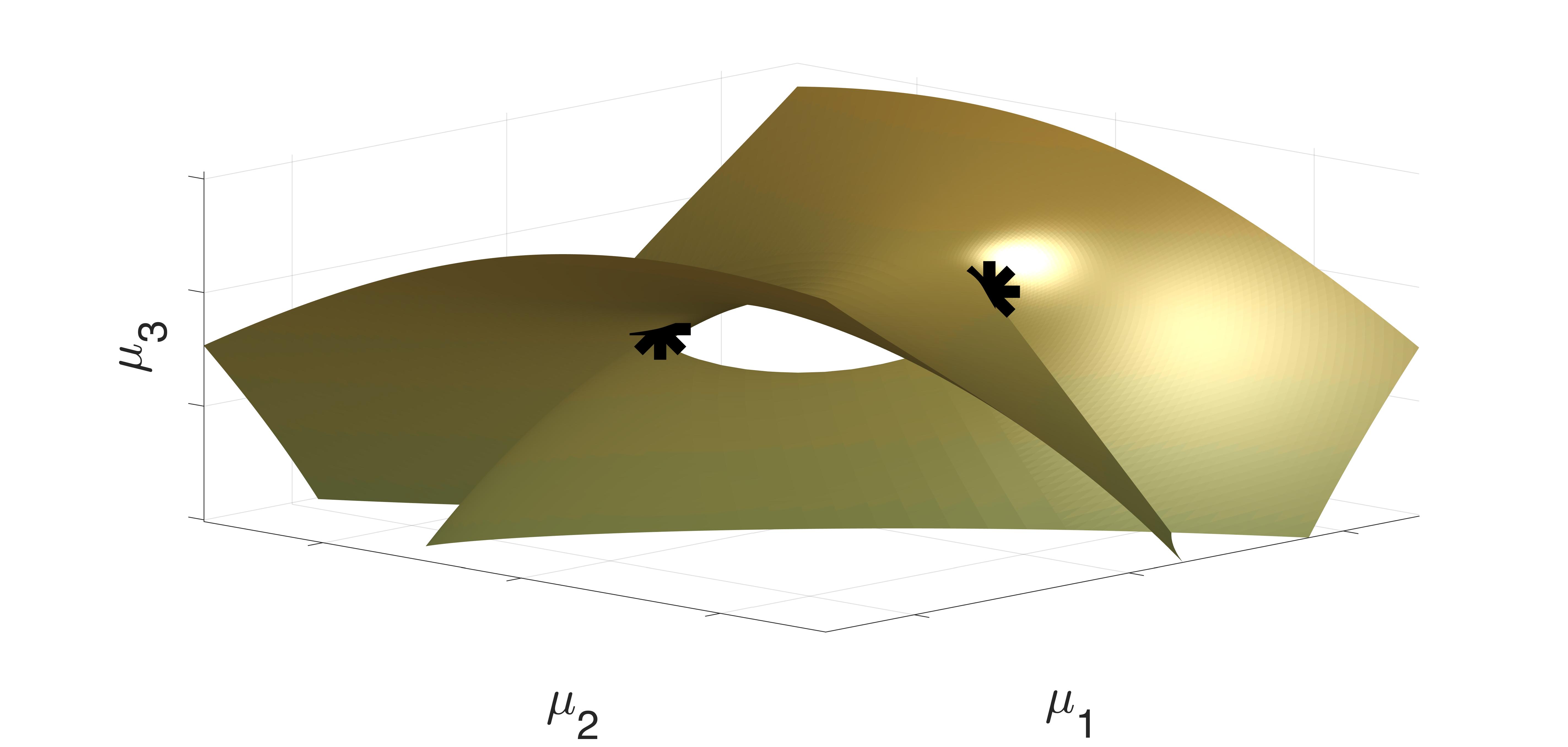}
\end{center}
\caption{The plots show those configurations of the parameters $\mu=(\mu_1,\mu_2,\mu_3)$ for which the problem $\nabla g_\mu(x,y)=0$ or $(\nabla g_\mu + f_\e)(x,y)=0$ becomes singular, i.e.\ a plot of the sets \eqref{eq:D4plevelbifur} or \eqref{eq:D4plevelbifurpert}, respectively.
Imagine moving around the parameter $\mu$ and watching the solutions bifurcating in the phase space. 
As $\mu$ crosses a sheet two solutions merge and vanish or are born (fold - $A_2$).
For $\mu$ in the intersection of two sheets there are two simultaneous fold singularities at different positions in the phase space. Crossing an edge three solutions merge into one (or vice versa).
Points contained in an edge correspond to cusp singularities
At the marked point in the left plot of the unperturbed problem there is a hyperbolic umbilic singularity. Moving the parameter $\mu$ upwards along the $\mu_3$ axis through the singular point four solutions merge and vanish. In the perturbed version to the right the hyperbolic umbilic point breaks up into two swallowtail points. While the left plot models using a symplectic integrator correctly showing a hyperbolic umbilic bifurcation $D_4^+$, the right plot models using a non-symplectic integrator incorrectly showing two nearby swallowtail bifurcations ($A_4$). See \cite{YoutubeHyperbolicUmbilic} for an animated version.}\label{fig:breakD4p}
\end{figure}



Figure \ref{fig:breakD4m} shows a level bifurcation set of an elliptic umbilic singularity $(D_4^-)$ and a generically perturbed version of the gradient-zero problem with a map that does not admit a primitive. Here we use the universal unfolding
\[g_\mu(x,y) =x^3-xy^2 + \mu_3 (x^2+y^2) +\mu_2 y + \mu_1 x.\]
We see that in the perturbed picture the lines of cusps fail to merge such that there is no elliptic umbilic point but only folds and cusp bifurcations. 
\begin{figure}
\begin{center}
\includegraphics[width=0.49\textwidth]{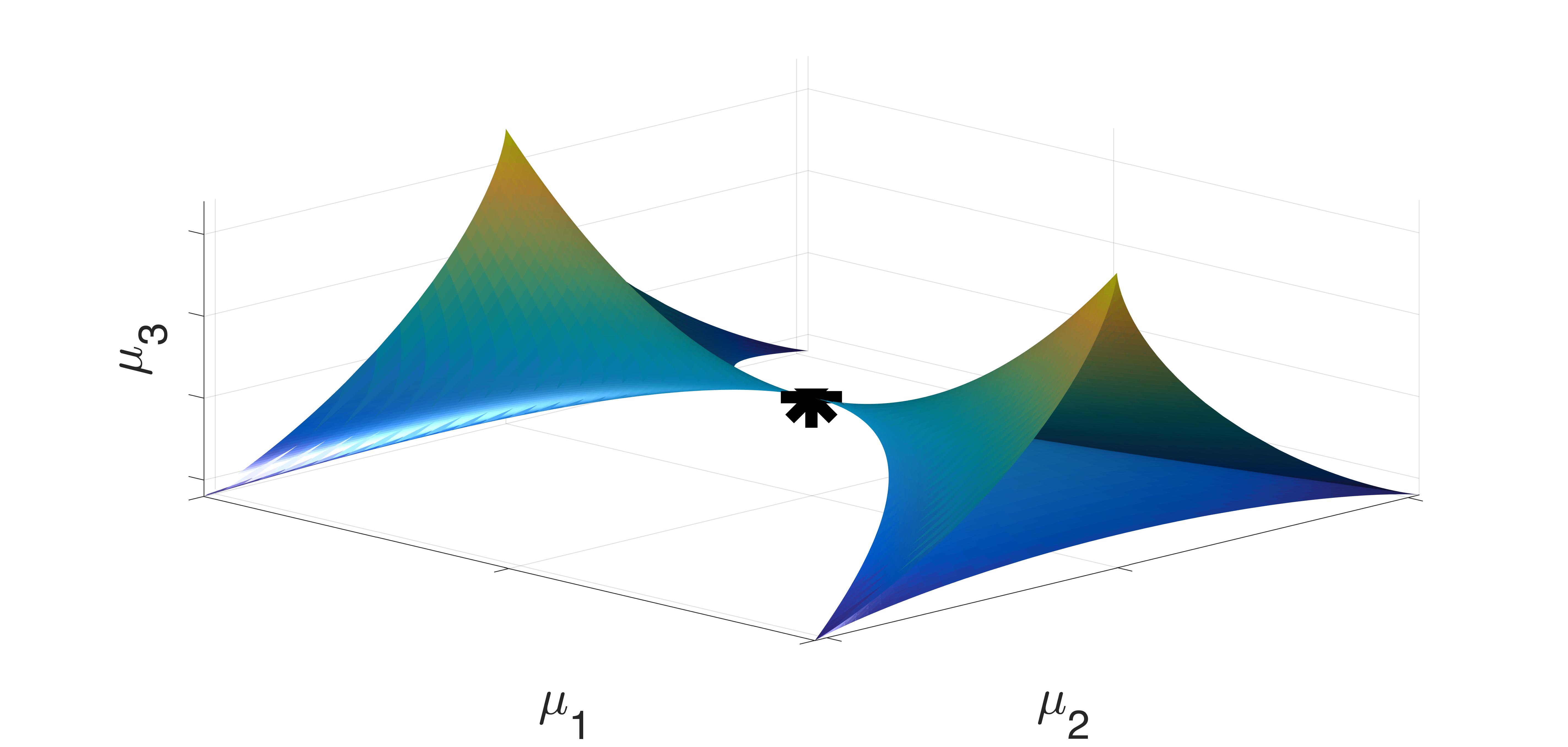}
\includegraphics[width=0.49\textwidth]{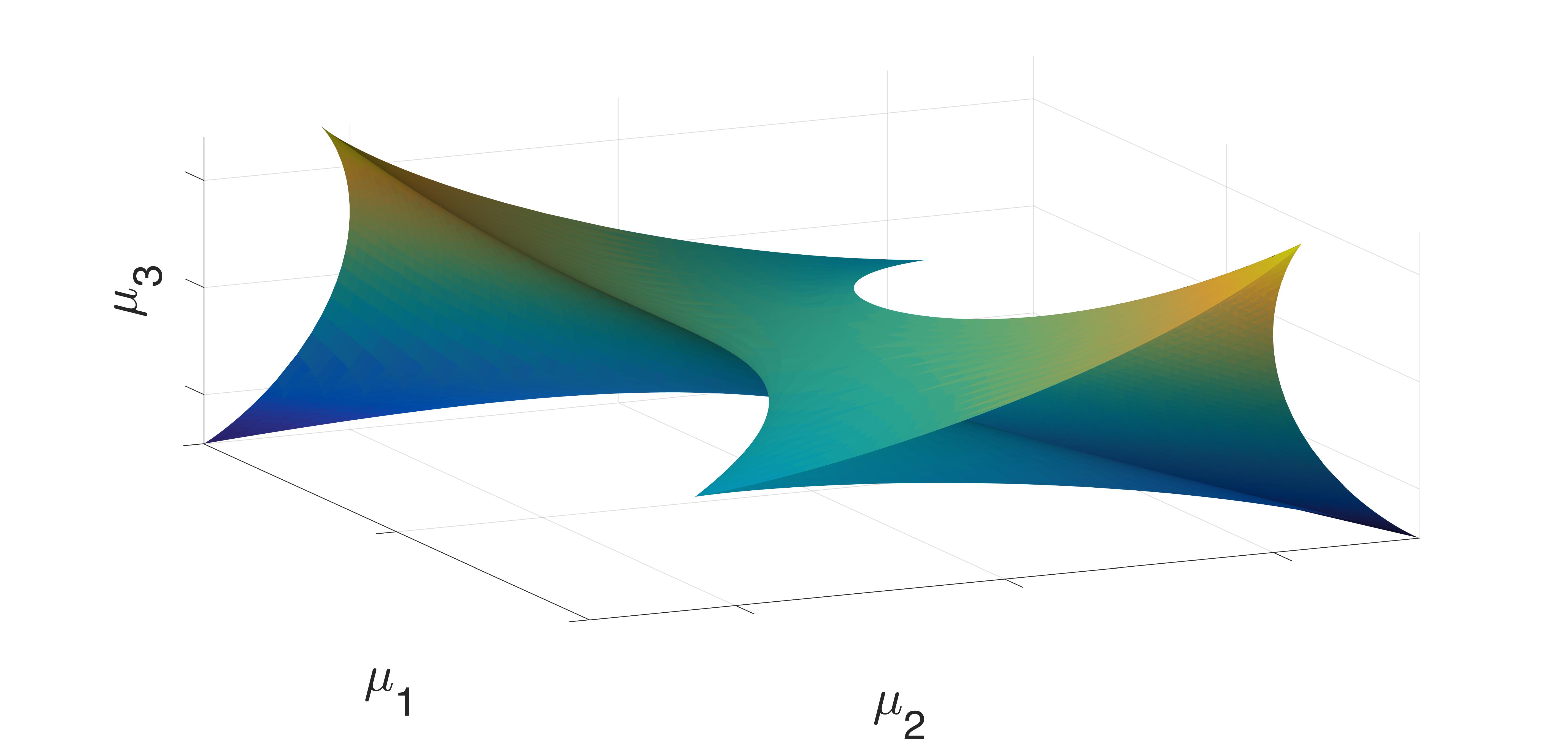}
\end{center}
\caption{
An exact (to the left) and a perturbed (to the right) version of the level bifurcation set of an elliptic umbilic singularity $D_4^-$. The elliptic umbilic point marked by an asterisk in the left plot is not present in the right plot where three lines of cusps fail to merge. The left figure models using a symplectic integrator correctly showing a $D_4^-$ singularity, the plot to the right models using a non-symplectic integrator incorrectly showing no elliptic umbilic point. See \cite{YoutubeEllipticUmbilic} for an animated version.}\label{fig:breakD4m}
\end{figure}
In the remainder of this section we will prove 
\begin{itemize}
\item
a general formula useful to analyse how $D$-series bifurcations deconstruct under perturbations which do not respect the gradient structure modelling a non-symplectic discretisation (lemma \ref{lem:unfoldBseries}),
\item
that the situations displayed in figure \ref{fig:breakD4p} and in figure \ref{fig:breakD4m} are universal, i.e.\ all generic roots-of-a-function type perturbations have the same described effects (proposition \ref{cor:breakD4p} and proposition \ref{cor:breakD4m})
\item
and that non-gradient-like perturbations decompose $D$-series singularities into $A$-series singularities (proposition \ref{prop:DdeconstructA}).
\end{itemize}

For this we will use the algebraic framework developed by Mather and presented, e.g.\ in the second lecture on $C^\infty$ stability and classification by Wall \cite{walllectures}. The following lemma analyses how $D$-series singularities unfold in the roots-of-a-function problem.


\begin{lemma}\label{lem:unfoldBseries} Consider the $D$-series singularity $D^{\pm}_{k+2}$ ($k\ge2$) defined by the germ $g(x,y)=x^2y \pm y^{k+1}$. A universal unfolding of $\nabla g$ in the free module of rank 2 over the ring $\R[[x,y]]$ of formal power series in the variables $x$ and $y$ is given as
\[
f_\mu(x,y) 
=\begin{pmatrix}2xy\\ x^2 \pm (k+1)y^k\end{pmatrix}
+\mu_1\begin{pmatrix}1\\0\end{pmatrix}
+\mu_2\begin{pmatrix}0\\1\end{pmatrix}
+\mu_3\begin{pmatrix}x\\0\end{pmatrix}
+\mu_4\begin{pmatrix}y\\0\end{pmatrix}
+ \sum_{j=1}^{k-2} \mu_j \begin{pmatrix}0\\y\end{pmatrix}.
\]
\end{lemma}

\begin{proof} In \cite[p.187]{walllectures} Wall describes a process developed by Mather to obtain universal unfoldings of topological singularities in the module $\mathcal M \R[[x,y]]^2$, where $\mathcal M$ is the maximal ideal of the local ring $\R[[x,y]]$. The claimed form is obtained by imitating this procedure in the bigger module $\R[[x,y]]^2$.
\end{proof}

\begin{remark}
The roots-of-a-function type singularity $\nabla g$ which corresponds to the singularity $D^{\pm}_{k+2}$ ($k \ge 2$) represented by the map germ $g(x,y)=x^2y \pm y^{k+1}$ is classified as $B_{2,k}$ for $D_{k+2}^+$ and as $B'_{2,k}$ for $D_{k+2}^-$ in \cite[p.268]{TopoStability}. If $k$ is odd then the classes $D_{k+2}^+$ and $D_{k+2}^-$ as well as $B_{2,k}$ and $B'_{2,k}$ coincide. The hyperbolic umbilic $D_4^+$ corresponds to $B_{2,2}$, the elliptic umbilic $D_4^-$ to $B'_{2,2}$ and the parabolic umbilic $D_5$ to $B_{3,2}$.
\end{remark}

\begin{remark}
In Mather's work and in Wall's lecture notes \cite[p.198]{walllectures} the $B$ series is denoted by $I$, $II$, $IV$. In particular, the singularity $B_{2,2}$ is denoted as $I_{2,2}$, $B'_{2,2}$ as $II_{2,2}$ and $B_{3,2}$ as $I_{2,3}$. 
\end{remark}


\begin{prop}\label{prop:DdeconstructA}
A non-symplectic integrator with any fixed (but not necessarily uniform) step-size, applied to any autonomous or nonautonomous generic Hamiltonian Lagrangian boundary value problem decomposes $D$-series singularities into $A$-series singularities for any positive maximal step-size.
\end{prop}

\begin{proof} Let $g_\mu$ be a smooth family of maps with a $D$-series singularity. To prove the assertion, we need to show that in any generic roots-of-a-function type perturbation of the problem $\nabla g_\mu=0$ around the singular point only $A$-series singularities occur.

By singularity theory the family $g_\mu$ is stably right-left equivalent to an unfolding of the $D$-series bifurcation defined by the germ
\[
g(x,y)=x^2y \pm y^{k+1} \quad (k\ge 2).
\]
By lemma \ref{lem:unfoldBseries} a generic perturbation of the problem $\nabla g_\mu=0$ will, after another change of variables, be of the form
\[
f_{\mu,\tilde \mu}(x,y)=\nabla (g_\mu + h_{\mu,\tilde \mu})(x,y) + \begin{pmatrix}y\\0\end{pmatrix}.
\]
Since the Jacobian matrix $\D f_{\mu,\tilde \mu}(x,y)$ cannot vanish, only those singularities can occur which require a rank-drop of at most 1. These are exactly the $A$-series singularities.
\end{proof}

To analyse the breaking of hyperbolic and elliptic umbilic singularities $D^+_4$ and $D^-_4$ it is convenient to formulate the following special case of lemma \ref{lem:unfoldBseries}.

\begin{lemma}\label{lem:unfoldB2}
Consider the germ defined by $g(x,y)=x^2y \pm y^{3}$. A universal unfolding of $\nabla g$ in $\R[[x,y]]^2$ is given by
\[f_\mu(x,y) 
=\begin{pmatrix}2xy\\ x^2 \pm 3y^2\end{pmatrix}
+\mu_1\begin{pmatrix}1\\0\end{pmatrix}
+\mu_2\begin{pmatrix}0\\1\end{pmatrix}
+\mu_3\begin{pmatrix}\mp x\\y\end{pmatrix}
+\mu_4\begin{pmatrix}y\\0\end{pmatrix}.\]
\end{lemma}

\begin{proof} The statement refers to the special case $k=2$ of lemma \ref{lem:unfoldBseries}, where a different basis is used to obtain the miniversal unfolding.
\end{proof}

%
%

The lemmas provide the tools needed to analyse how $D$-series bifurcations decompose if generically perturbed within the roots-of-a-function problem. This models the effect of using a non-symplectic integrator to resolve a $D$-series bifurcation in a Lagrangian Hamiltonian boundary value problem.
The following proposition shows that the situation shown in figure \ref{fig:breakD4p} for the hyperbolic umbilic $D_4^+$ is universal.

\begin{prop}\label{cor:breakD4p}
A hyperbolic umbilic singularity $D_4^+$ in the critical-points problem $\nabla g_\mu=0$ for a smooth family of real valued maps $g_\mu$ decomposes into two swallowtail points $A_4$ if the problem $\nabla g_\mu=0$ is generically perturbed to a problem $f_\mu=0$, where the perturbation does not respect the gradient structure.
\end{prop}

\begin{proof} By singularity theory the family $g_\mu$ is stably right-left equivalent to the universal unfolding
\[g_\mu(x,y)=y^3+x^2y+\mu_3(y^2-x^2) + \mu_2 y + \mu_1 x\]
(see table \ref{tab:ThomCatastrophes}) of the singularity $D^+_4$.
Comparing $\nabla g_\mu$ with the unfolding obtained in lemma \ref{lem:unfoldB2} we see that $\nabla g_\mu$ constitutes a versal unfolding of $\nabla g$ in $\R[[x,y]]^2$ which can be made universal by adding the term $\mu_4 \begin{pmatrix}y\\0
\end{pmatrix}$. This gives
\[
f_\mu(x,y) 
=\begin{pmatrix}2xy\\ x^2 + 3y^2\end{pmatrix}
+\mu_1\begin{pmatrix}1\\0\end{pmatrix}
+\mu_2\begin{pmatrix}0\\1\end{pmatrix}
+2\mu_3\begin{pmatrix}- x\\y\end{pmatrix}
+\mu_4\begin{pmatrix}y\\0\end{pmatrix}.
\]
Let us fix $\mu_4 \not=0$. Only $A$-series bifurcations are possible because the Jacobian $D f_\mu$ cannot vanish. $A$-series bifurcations are determined by their codimension. We have
\[
\det D f_\mu(x,y)
=-4\left( x + \frac{\mu_4}4\right)^2
+12\left( y - \frac{\mu_3}3\right)^2
+\frac 14 \mu_4^2-\frac {16}3 \mu_3^2.
\]
At values $x,y,\mu_3$ with $\det D f_\mu(x,y)=0$ bifurcations occur with parameters $\mu_1$, $\mu_2$ uniquely determined by $f_\mu(x,y)=0$.
If $\mu_3 \not \in \{-\frac {\sqrt{3}} {8} \mu_4, \frac {\sqrt{3}} {8} \mu_4\}$ then we see codimension-2 bifurcations, i.e.\ cusp bifurcations, with cusp points lying on hyperbolas. For $\mu_3 \in \{-\frac {\sqrt{3}} {8} \mu_4, \frac {\sqrt{3}} {8} \mu_4\}$ the cusps merge to codimension-3 bifurcations, i.e.\ swallow tails with swallowtail points at $(x,y)=(-\frac 14 \mu_4,\pm \frac{\sqrt{3}}{24}\mu_4)$.
\end{proof}


The following proposition shows that the situation shown in figure \ref{fig:breakD4m} for the elliptic umbilic $D_4^-$ is universal.

\begin{prop}\label{cor:breakD4m}
If a generic smooth family of maps $\nabla g_\mu=0$ for real valued maps $g_\mu$ has an elliptic umbilic singularity then a small generic perturbation in the module $\R[[x,y]]^2$ will decompose the singularity into three separated lines of cusps. 
\end{prop}

\begin{proof} By singularity theory the family $g_\mu$ is stably right-left equivalent to $g_\mu(x,y)=y^3-x^2y+\mu_3(x^2 + y^2) + \mu_2 y + \mu_1 x$ (see table \ref{tab:ThomCatastrophes}).
By lemma \ref{lem:unfoldBseries} the family $\nabla g_\mu$ constitutes a versal unfolding of $\nabla g$ in $\R[[x,y]]^2$ which can be made universal by adding the term $\mu_4 \begin{pmatrix}y\\0
\end{pmatrix}$. This gives
\begin{equation*}
f_\mu(x,y)=\begin{pmatrix}
-2xy\\
3y^2-x^2
\end{pmatrix}
+\mu_1 \begin{pmatrix}
1\\0
\end{pmatrix}
+\mu_2 \begin{pmatrix}
0\\1
\end{pmatrix}
+2\mu_3 \begin{pmatrix}
x\\y
\end{pmatrix}
+\mu_4 \begin{pmatrix}
y\\0
\end{pmatrix}.
\end{equation*}

We have
\begin{equation}\label{eq:breakD4meq}
\det \D f_\mu(x,y)
=-4\left(x-\frac {\mu_4}4\right)^2 
-12\left(y-\frac {\mu_3}3\right)^2
+\frac {16}3 \mu_3^2 + \frac 14 \mu_4^2.
\end{equation}
At any $\mu_3,\mu_4,x,y$ with $\det \D f_\mu(x,y)=0$ a bifurcation takes place with parameters $\mu_1,\mu_2$ uniquely determined by $f_\mu(x,y)=0$. Inspecting \eqref{eq:breakD4meq} we see that the generic codimension-2 cusp bifurcations, which occur for $\mu_4=0$, survive a perturbation but cannot merge to a higher codimensional bifurcation if $\mu_4\not=0$.
\end{proof}

The fact that singularities break under small perturbations which are generic in the roots-of-a-function problem and singularities occur which do not exist in the exact problem illustrates that the roots-of-a-function problem is different to the gradient zero problem and demonstrates the importance of the preservation of symplectic structure when calculating bifurcation diagrams for Hamiltonian boundary value problems. 


%% file: dirichlet_extra_structure.tex
\section{Separated Lagrangian problems}\label{sec:DirichletExtra}

Given the significance of Dirichlet-, Neumann-, Robin- boundary value problems in applications, let us analyse the bifurcation behaviour in a problem class which we refer to as \textit{separated Lagrangian boundary value problems}.
We recall results from \cite{bifurHampaper} and analyse structure that is present in the data of such problems which can help to locate bifurcation points numerically. As an example, we locate a $D$-series bifurcation in a H{\'e}non-Heiles-type system.

\subsection{Definitions and set-up}\label{subsec:introDiri}


\begin{definition}[Separated (Lagrangian) boundary value problem]\label{def:LagDirproblem}
Let $(M,\omega)$ and $(M',\omega')$ be two $2n$-dimensional symplectic manifolds. Consider a symplectic map $\phi \colon M \to M'$ and $n$-dimensional submanifolds $\Lambda \subset M$ and $\Lambda' \subset M'$. The collection $(\phi, \Lambda,\Lambda')$ is called a {\em separated boundary value problem}. Its solution is given as
\[
\{ z \in \Lambda \, | \, \phi(z) \in \Lambda'\} = \phi^{-1}(\Lambda') \cap \Lambda.
\]
If $\Lambda \subset M$ and $\Lambda' \subset M'$ are Lagrangian manifolds then $(\phi, \Lambda,\Lambda')$ is called a {\em separated Lagrangian boundary value problem}.
\end{definition}

\begin{remark}
A submanifold $\Lambda \times \Lambda'$ is Lagrangian in $(M \times M', \omega \oplus -\omega')$ if and only if $\Lambda \subset M$ and $\Lambda' \subset M'$ are Lagrangian submanifolds. Therefore, the separated boundary value problems which are Lagrangian boundary value problems (definition \ref{def:LagBVP}) are exactly the separated Lagrangian boundary value problems.
\end{remark}



\begin{observation}
Separated Lagrangian boundary value problems $(\phi, \Lambda,\Lambda')$ can be localized near a solution $z \in \Lambda$ with $z'=\phi(z) \in M'$: shrink $M$ to a neighbourhood $\tilde M$ of $z$, $M'$ to a neighbourhood $\tilde M'$ of $z'$ and consider $(\phi, \Lambda\cap \tilde M,\Lambda' \cap \tilde M')$.
\end{observation}

\begin{observation}[universal local coordinate description]
Classical Dirichlet problems, Neumann problems and Robin-type boundary value problems represented as in figure \ref{fig:RobinFish} constitute separated Lagrangian boundary value problems. Moreover, all separated Lagrangian boundary value problems are locally equivalent:
by Darboux-Weinstein's theorem neighbourhoods of Lagrangian submanifolds are locally symplectomorphic to neighbourhoods of the zero section of the cotangent bundle over the submanifolds \cite[Corollary 6.2.]{WEINSTEIN1971329}. Therefore, a separated Lagrangian boundary value problem $(\phi\colon (M,\omega) \to (M',\omega'), \Lambda,\Lambda')$ is locally given as
\begin{equation}\label{eq:DirichletProblem}
x=x^\ast, \quad  \phi^X (x,y)=X^\ast,
\end{equation}
with local Darboux coordinates $(x,y)=(x^1,\ldots x^n, y_1,\ldots y_n)$ for $M$, $(X,Y)=(X^1,\ldots X^n, Y_1,\ldots Y_n)$ for $M'$ and $x^\ast, X^\ast \in \R^{2n}$. In \eqref{eq:DirichletProblem} the symbol $\phi^X$ is a shortcut for $X \circ \phi$. 
This means that Dirichlet-, Neumann- and Robin- boundary conditions can be treated on the same footing in the bifurcation context. In contrast, periodic boundary conditions are not separated. This manifests in a different bifurcation behaviour \cite[Prop. 3.3.+3.4]{bifurHampaper}.
\end{observation}

\subsection{Structures induced by separated Lagrangian boundary conditions}

%
%
%
%

Let us consider the separated Lagrangian boundary value problem \eqref{eq:DirichletProblem}
on the phase space $M=M'=\R^{2n}$ with the standard symplectic form $\sum_{j=1}^n \d x^j \wedge \d y_j$.
Introducing a generic parameter $\mu$ in the map $\phi$ or in the boundary condition, the bifurcation diagram of  \eqref{eq:DirichletProblem} can be viewed as
\begin{equation}\label{eq:DirichletBifurSet}
\{(\mu,y) \,  | \, h_\mu(y) = 0 \} 
\end{equation}
with
\begin{equation}\label{eq:mapDirichlet}
h_\mu(y) = \phi_\mu^X(x^\ast,y)-X_\mu^\ast.
\end{equation}
On the other hand, as \eqref{eq:DirichletProblem} is a separated Lagrangian boundary value problem, the problem is locally equivalent to a gradient-zero problem $\nabla g_\mu(z)=0$ in $n$ variables \cite[Prop.3.3]{bifurHampaper}.
In contrast to $h_\mu$, the maps $\nabla g_\mu$ arise as gradients of smooth maps such that the bifurcation behaviour is governed by catastrophe theory. Indeed, the bifurcations which occur as generic bifurcations in gradient-zero problems $\nabla g_\mu(z)=0$ with smooth families of maps $g_\mu$ in $n$ variables occur as generic bifurcations in separated Lagrangian boundary value problems with $2n$-dimensional phase spaces. 

The gradient structure is not visible in \eqref{eq:mapDirichlet}. Naively, it appears like the problem $h_\mu(y)=0$ should behave like a generic roots-of-a-function-type problem for maps in $n$ variables. However, we know that it behaves like the gradient-zero-problem, so where has the gradient structure gone?
The map $\phi_\mu^X$ is a component of the map $\phi_\mu$ which is symplectic. However, symplecticity of the Jacobian matrix $\D \phi(x,y)$ does \textit{not} force any extra structure on the submatrix $D_y \phi^X(x,y)$ at points $(x,y)$ in the phase space. Indeed, the extra structure hides away in the following detail:
for $n>1$ those small perturbations $\tilde h(y) = h_\mu(y)+\xi_\mu(y)$ of $h_\mu$ which are required to break gradient-zero bifurcations leading to a roots-of-a-function-type behaviour do \textit{not} come from \textit{small} perturbations of $\phi$ through symplectic maps.
In other words, for a fixed parameter $\mu^\ast$ it is impossible to obtain a versal roots-of-a-function-type unfolding of $h_{\mu^\ast}$ by varying $\phi$ through symplectic maps to produce an unfolding $h_\mu$ via \eqref{eq:mapDirichlet}.

For practical purposes, formulating the problem \eqref{eq:DirichletProblem} in the form \eqref{eq:DirichletBifurSet} is beneficial because passing to generating functions to analyse the bifurcation behaviour becomes obsolete as the following proposition justifies. 

\begin{prop}\label{prop:kernelDIM}
If the dimension of the kernel of the Jacobian matrix of the map \eqref{eq:mapDirichlet} at a parameter value $\mu$ and a value $y$ is $m$ then the kernel of the Hessian of the map $g_\mu$ of the corresponding gradient-zero-problem at $(\mu,y)$ is $m$.
\end{prop}

\begin{remark} In other words, proposition \ref{prop:kernelDIM} says that if the dimension of the kernel of the Jacobian matrix of the map \eqref{eq:mapDirichlet} at a parameter value $\mu$ and a value $y$ is $m$ then the fully reduced form of the singularity in the corresponding gradient-zero-problem has $m$ variables. This is helpful when looking up singularities in classification tables and allows, for instance, to easily differentiate between $A$- and $D$-series singularities.
\end{remark}

\begin{proof} For a given parameter $\mu$ solutions to \eqref{eq:DirichletProblem} correspond to points in the intersection of the graph $\Gamma_\mu = \{(x,y,\phi^X_\mu(x,y),\phi^Y_\mu(x,y)) \, | \, (x,y) \in \R^{2n}\}$ and  $\Lambda = \{(x^\ast,y,X^\ast,Y)\, | \, y,Y \in \R^n\}$. In the frame $\frac{\p}{\p x^1},\ldots,\frac{\p}{\p x^n},\frac{\p}{\p y_1},\ldots,\frac{\p}{\p y_n}$ the tangent spaces to $\Gamma_\mu$ are spanned by the columns of the matrix $M_1$ while the tangent spaces to $\Lambda_\mu$ are spanned by the columns of the matrix $M_2$ with $M_1$ and $M_2$ given as

\begin{equation*}
M_1=\begin{pmatrix}
\Id_{ n}&0\\ 
0 & \Id_{n}  \\
D_x \phi^X_\mu & \D_y \phi^X_\mu\\
D_x \phi^Y_\mu & \D_y \phi^Y_\mu\\
\end{pmatrix}
\qquad
M_2=\begin{pmatrix}
0&0\\
\Id_{n}&0\\
0&0\\
0&\Id_{n }
\end{pmatrix}.
\end{equation*}
Here $\Id_n$ denotes an $n$-dimensional identity matrix. 
Since the Jacobian matrix of $h_\mu$ coincides with $\D_y \phi^X_\mu$, the dimension of the kernel of $\D h_\mu$ determines the dimension of the intersection of the tangent spaces at solutions in $\Gamma_\mu \cap \Lambda$.\end{proof}

We can formulate the following

\begin{corollary}\label{cor:DseriesifJac0}
If $n=2$ then the Jacobian matrix $\D_y \phi^X_\mu$ vanishes at a $D$-series bifurcation.
\end{corollary}

\subsection{Numerical example. H{\'e}non-Heiles Hamiltonian system}\label{subsec:HenonHeilsD4}

In the following we consider a H{\'e}non-Heiles system, which is a Hamiltonian system originally derived to model galactic dynamics and is known to exhibit chaotic behaviour \cite{Henon1983,Tabor1989}.
Consider the Hamiltonian
\begin{equation}\label{eq:HenonHeilsHam}
H(x,y)=\frac 12 \|y\|^2+\frac 12 \|x\|^2-10\left(x_1^2 x_2-\frac{x_2^3}3 \right)
\end{equation}
on the phase space $\R^2\times \R^2$.
In \eqref{eq:HenonHeilsHam} the norm $\| . \|$ denotes the Euclidean norm on $\R^2$. We obtain a symplectic map $\phi$ by integrating Hamilton's equations
\[
\dot x = \nabla_y H(x,y),\quad \dot y = -\nabla_x H(x,y)
\]
up to time $\tau=1$ using the 2nd order symplectic St\"ormer-Verlet scheme with 10 time-steps. 
We consider the Dirichlet-type problem
\begin{equation}\label{eq:DirichletProblem2}
x=x^\ast, \quad  \phi^X (x,y)=X^\ast.
\end{equation}
This time we let $x^\ast$ and $X^\ast$ be the parameters of the problem. The (high-dimensional) bifurcation diagram can be thought of as the graph of $\phi$ plotted over the parameter space $(x^\ast,X^\ast)=(x,X)$.
To reduce dimensionality we fix the parameter $x_1^\ast=0$ leaving the parameters $x_2$, $X_1$, $X_2$ free. The level bifurcation set, i.e.\ the set of points in the parameter space at which a bifurcation occurs in a chosen subset $U$ of the phase space, is given as
\[ \{ (x_2,\phi^X(0,x_2,y_1,y_2)) \; | \; \det D_y \phi^X(0,x_2,y_1,y_2)=0, (0,x_2,y_1,y_2)\in U \}. \]
Figure \ref{fig:D4m} shows the level bifurcation set of the problem near an elliptic umbilic singularity $D_4^-$. Derivatives of the symplectic approximation to $\phi$ were obtained using automatic differentiation. The $D$-series bifurcation was found numerically by solving $(x_2,y_1,y_2) \mapsto D_y \phi(0,x_2,y_1,y_2)=0$ as justified in corollary \ref{cor:DseriesifJac0}. We see that the elliptic umbilic bifurcation is captured correctly. A non-symplectic integrator, however, breaks this bifurcation, see \ref{app:breakD4RK}.

\begin{figure}
\begin{center}
\includegraphics[width=0.7\textwidth]{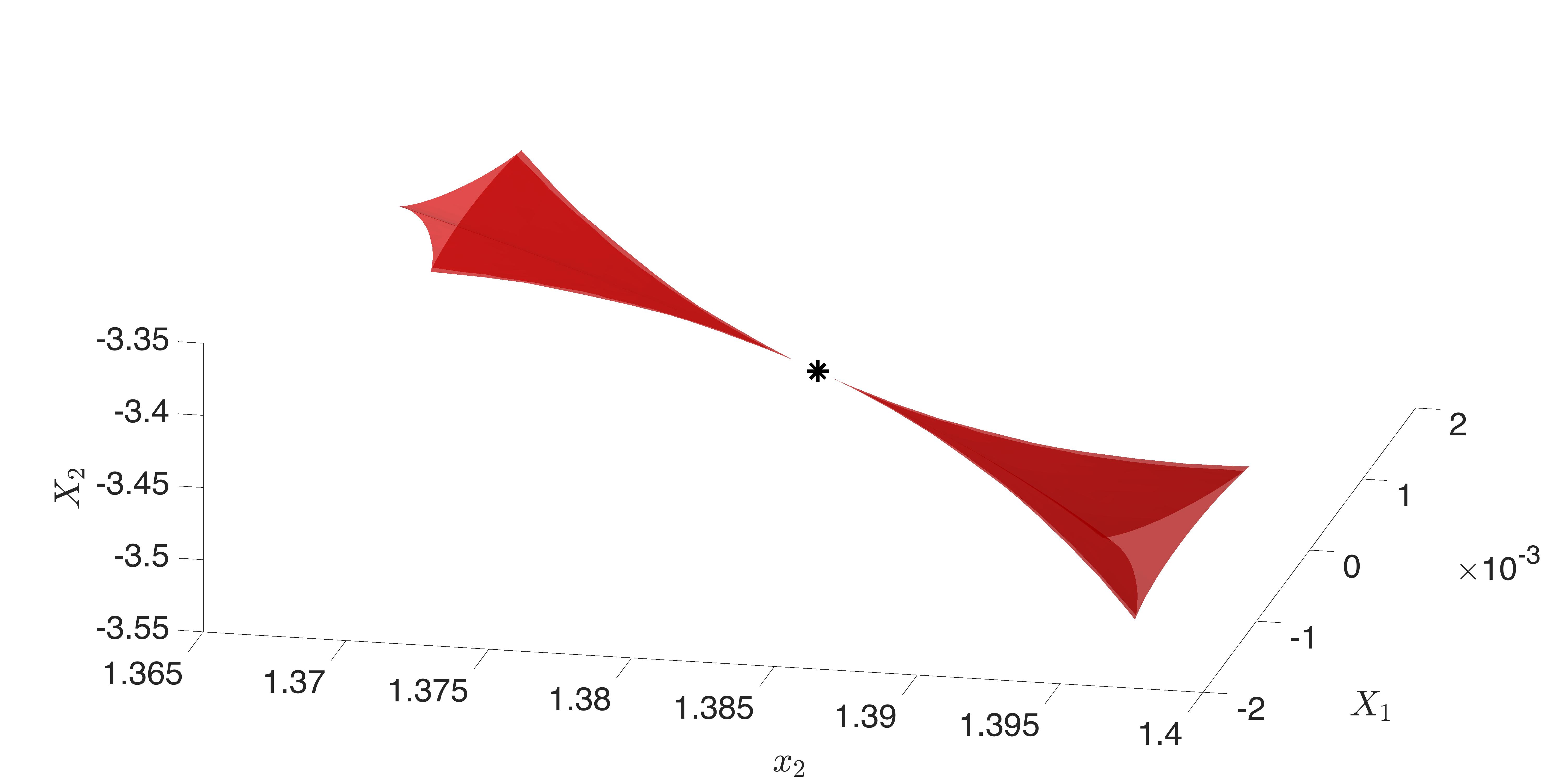}
\end{center}
\caption{Elliptic umbilic $D_4^-$ in the problem \eqref{eq:DirichletProblem} for the numerical time-1-map of the H{\'e}non-Heiles system \eqref{eq:HenonHeilsHam} where the boundary values are parameters and $x_1^\ast=0$ is fixed to reduce dimensionality. The numerical flow was obtained using the 2nd order symplectic St\"ormer-Verlet scheme with 10 time-steps. Derivatives were obtained using automatic differentiation.}\label{fig:D4m}
\end{figure}

\subsection{Computation of conjugate loci}\label{subsub:applgeod}

\subsubsection{Symplectic structure in the geodesic conjugate-points problem}

Consider an $n$-dimensional Riemannian manifold $(N,g)$ with cotangent bundle $T^\ast N$ equipped with the canonical symplectic structure $\omega$
and local Darboux coordinates $q^1,\ldots,q^n,p_1,\ldots,p_n$. 
Define the Hamiltonian
\begin{equation}\label{eq:HamGeolocal}
H(q,p) = \frac 12 \sum_{i,j=1}^{n} g^{ij}(q) p_i p_j,
\end{equation}
where $g^{ij}$ is the $(i,j)$-entry of the inverse of the matrix representation of the Riemannian metric $g$ in the coordinate frame $\frac{\p}{\p q_1},\ldots,\frac{\p}{\p q_{n}}$. 
Using $g$ as a bundle isomorphism between the tangent bundle $TN$ and the co-tangent bundle $T^\ast N$, the motions in $(T^\ast N,\omega, H)$ correspond to the velocity vector fields of geodesics $\gamma$ on $(N,g)$, i.e.\ solutions of the geodesic equation
\[
\ddot \gamma^k = - \Gamma^k_{ij}\circ \gamma \cdot \dot \gamma^i \dot \gamma ^j,
\]
where $\gamma^j = q^j \circ \gamma$ denote the components of $\gamma$ and
\[
\Gamma^k_{ij} = \frac 12 \sum_{l=1}^{n} g^{kl}
\left(\frac{\p g_{jl}}{\p q_i}+\frac{\p g_{li}}{\p q_j}-\frac{\p g_{ij}}{\p q_l}\right)
\] the Christoffel symbols w.r.t.\ the Levi-Civita connection.
Recall that if $\gamma$ is a geodesic starting at time $t=0$ at $x$ with $\dot \gamma(0) = y$ then the dimension of the kernel of the exponential map evaluated at a point $y \in T_xN\subset TN$ corresponds to the number of linearly independent Jacobi fields\footnote{vector fields along a geodesic arising as variational vector fields for variations through geodesics} along $\gamma$ which vanish at $x$ and $\gamma(1)$. If there exists a non-trivial Jacobi vector field then the points $x$ and $\gamma(1)$ are called \textit{conjugate points} and the number of linearly independent Jacobi vector fields vanishing at $x$ and $\gamma(1)$ is called the \textit{multiplicity} of $x$.
The multiplicity of conjugate points on an $n$-dimensional Riemannian manifold cannot exceed $n-1$ \cite[Ch.5]{doCarmo}.





%
%
%
There is a variety of aspects to the conjugate-points problem on a Riemannian manifold. An analysis of how cusp points in a conjugate locus can bifurcate as the reference point moves is presented in \cite{WATERS20171}. A functional analytic approach to the geodesic bifurcation problem (in an extended sense) can be found in \cite{Piccione2004}.
Calculating geodesics on submanifolds of the Euclidean space is often motivated by the task of finding distance minimising curves between two points. Several methods are presented in \cite{GeodesicCourseWork}. An approach using geodesics as homotopy curves can be found in \cite{Thielhelm2015}.
Moreover, the authors show in \cite{obstructionPaper} how the homogeneity of $H$ in \eqref{eq:HamGeolocal} imposes obstructions on the bifurcation behaviour and generalise this to systems with conformal-symplectic symmetries.

The Hamiltonian formulation reveals the symplectic structure in the problem of connecting two points by a geodesic as the start and endpoints move apart.
On a surface this structure is not relevant because the maximal multiplicity of conjugate points is 1 \cite[Ch.5]{doCarmo}.
Therefore, by proposition \ref{prop:kernelDIM} only singularities of degree 1 can occur.
For a generic setting this means that a small, possibly non-symplectic perturbation of the problem will only move such singularities slightly but would not change their type or remove them (proposition \ref{prop:Aseriespersist}).
However, if the Riemannian manifold is at least 3-dimensional one has to capture the symplectic structure in order to be able to find unbroken $D$-series bifurcations as explained in section \ref{sec:brokenBifur}. This requires using symplectic integrators. 

\subsubsection{Structure preserving discretisation}

The Hamiltonian \eqref{eq:HamGeolocal} is not separable so symplectic integration requires the use of an implicit method.
In the popular 2nd order symplectic St\"ormer-Verlet scheme, for instance, $\dim N$-dimensional equations have to be solved in each time-step causing high computational costs.
However, in applications $(N,g)$ is often given as a low codimensional submanifold of a Euclidean space where $g$ is the induced metric; a fact which can be exploited.
Indeed, on a codimension-$k$ submanifold the symplectic RATTLE method only requires a $k$ dimensional system of equations to be solved in each time-step. This is particularly efficient for hypersurfaces where $k=1$. A derivation of the general RATTLE method can be found in \cite[VII1.4]{GeomIntegration}. For our purposes it is advisable to have a derivative of the RATTLE-approximation of the geodesic exponential map available. This can be achieved using automatic differentiation. We present a 1-jet version of the RATTLE method for geodesics on hypersurfaces in \ref{app:geodesicsRATTLEformulas}.

Using RATTLE to compute geodesics of the hypersurface $f^{-1}(0)$ requires only the value and first derivative of $f$; the metric tensor and Christoffel symbols are not needed.
Jet-RATTLE, needed to reliably detect conjugate points, also requires the second derivative of $f$.


\subsubsection{Numerical examples using the jet-RATTLE method}

\begin{example}
Figure \ref{fig:CuspsGraph} shows the conjugate locus on the graph of the perturbed 2-dimensional Gaussian
\[0= f(q_1,q_2,q_3)= h(q_1,q_2)-q_3 = \exp(-q_1^2-0.9q_2^2)+0.01q_1^3+0.011q_2^3-q_3\]
to the point $q^\ast=(-1,0,h(-1,0))$, i.e.\ the points which are conjugate to $q^\ast$.
There are three geodesics connecting the start point $q^\ast$ marked as $\ast$ in the plot with a point in between the solid black lines in the $Q_1>0$ region.
Keeping $q^\ast$ fixed and varying the end point two of the geodesics merge in a fold bifurcation as the end point crosses over one of the solid black lines.
If the end point crosses the meeting point of the lines of folds all connecting geodesics merge into one.
The meeting point corresponds to a cusp singularity.

For numerical computations notice that the conjugate locus is the level bifurcation set of the Dirichlet problem for the geodesic equations on the graph of $h$, where the $(Q_1,Q_2)$-coordinate of the end point are the parameters of the problem.
We can discretise the 1-jet of the geodesic exponential map by applying the jet-RATTLE method to $f(q) = 0$. Let us refer to the $Q_1$ and $Q_2$-component of the numerical flow as $\Phi^{Q_{1,2}}$.
%
The matrix
\[A = \begin{pmatrix}
1&0\\0&1\\ \frac{\p f}{\p q_1}(q^\ast) & \frac{\p f}{\p q_2}(q^\ast)
\end{pmatrix}
\]
maps $\R^2$ to the tangent space of the graph of $h$ at $q^\ast$.
The level bifurcation set is obtained by calculating the zero-level set of
\[
\begin{pmatrix}p_1\\p_2\end{pmatrix} \mapsto \det D_{p_1,p_2}\Phi^{Q_{1,2}}\left(q^\ast,A \begin{pmatrix}p_1\\p_2\end{pmatrix}\right)\cdot A
\]
and mapping the set to the graph of $h$ using $\Phi$. The bifurcation behaviour persists and is also present in the unperturbed setting, where $h(q_1,q_2)=\exp(-q_1^2-q_2^2)$. 
\end{example}

\begin{example}
The plot to the left of figure \ref{fig:CuspsGraphSphere} shows the conjugate locus on the perturbed 2 dimensional ellipsoid
\[
(0.98q_1^2+0.97q_2^2+1.02q_3^2)-1/\pi^2+0.1(-q_1^3-1.2q_2^3+0.7q_3^3)=0
\]
to $q^\ast=(q^\ast_1,q^\ast_2,q^\ast_3)=(-0.316472,0,0)$ projected along $Q_1$ to the $Q_2$/$Q_3$ plain. Notice that $Q_2$, $Q_3$ constitute a coordinate system in the considered regime near the approximate anti-podal point of $q^\ast$.
We see a formation of cusps. 
On an unperturbed ellipsoid we see four cusp bifurcations as shown in the right hand side plot of figure \ref{fig:CuspsGraphEllipsoid} unless $q^\ast$ is an umbilic point of the ellipsoid, in which case the formation collapses to a point. This is known as the last geometric statement of Jacobi \cite{Itoh2004}. 
\end{example}

\begin{example}
The plot in the centre and to the right of figure \ref{fig:ellipticumbilichypersphere} shows a subset of the conjugate locus to $(q^\ast_1,q^\ast_2,q^\ast_3,q^\ast_4)=(-0.355367,0,0,0)$ on the perturbed 3 dimensional ellipsoid
\[
f(q) = 0.98 q_1^2+0.95 q_2^2+1.05 q_3^2+1.03 q_4^2-\frac 1 {\pi^2}+0.5 (q_1^3+1.1 q_2^3+0.9 q_3^3+1.05 q_4^3) =0.
\]
The functions $Q_2$, $Q_3$, $Q_4$ constitute a coordinate system in the considered regime. The variables $Q_2$, $Q_3$, $Q_4$ act as three parameters in the corresponding boundary value problem such that we can find bifurcations of codimension 3. Indeed, the conjugate locus contains an elliptic umbilic singularity occurring where three lines of cusps merge.
The plot to the left shows the position of the singularities in the tangent space at $q^\ast$ in spherical coordinates: the coordinates $p_1,p_2,p_3,p_4$ can be obtained by first mapping
\[(r,\theta,\phi) \mapsto \rho =(\rho_1,\rho_2,\rho_3)=(r \sin \theta \cos \phi, r \sin \theta \sin \phi,r \cos \theta)
\]
and then mapping $\rho \mapsto A \rho$, where $A$ is a $4 \times 3$ dimensional matrix whose columns are an orthonormal basis of the kernel of $p \mapsto \nabla f(q^\ast)p$ near $\begin{pmatrix}0_{1\times 3}\\ \Id_3\end{pmatrix}$. The plot in the centre and to the right can be obtained from the plot to the left by calculating the $p$ variables and applying the exponential map at $q^\ast$. We see that the elliptic umbilic bifurcation is captured qualitatively correctly. In contrast, a non-symplectic integrator breaks this  bifurcation (see \ref{app:breakD4RK}).
\end{example}

\begin{figure}
\begin{center}
\includegraphics[width=0.7\textwidth]{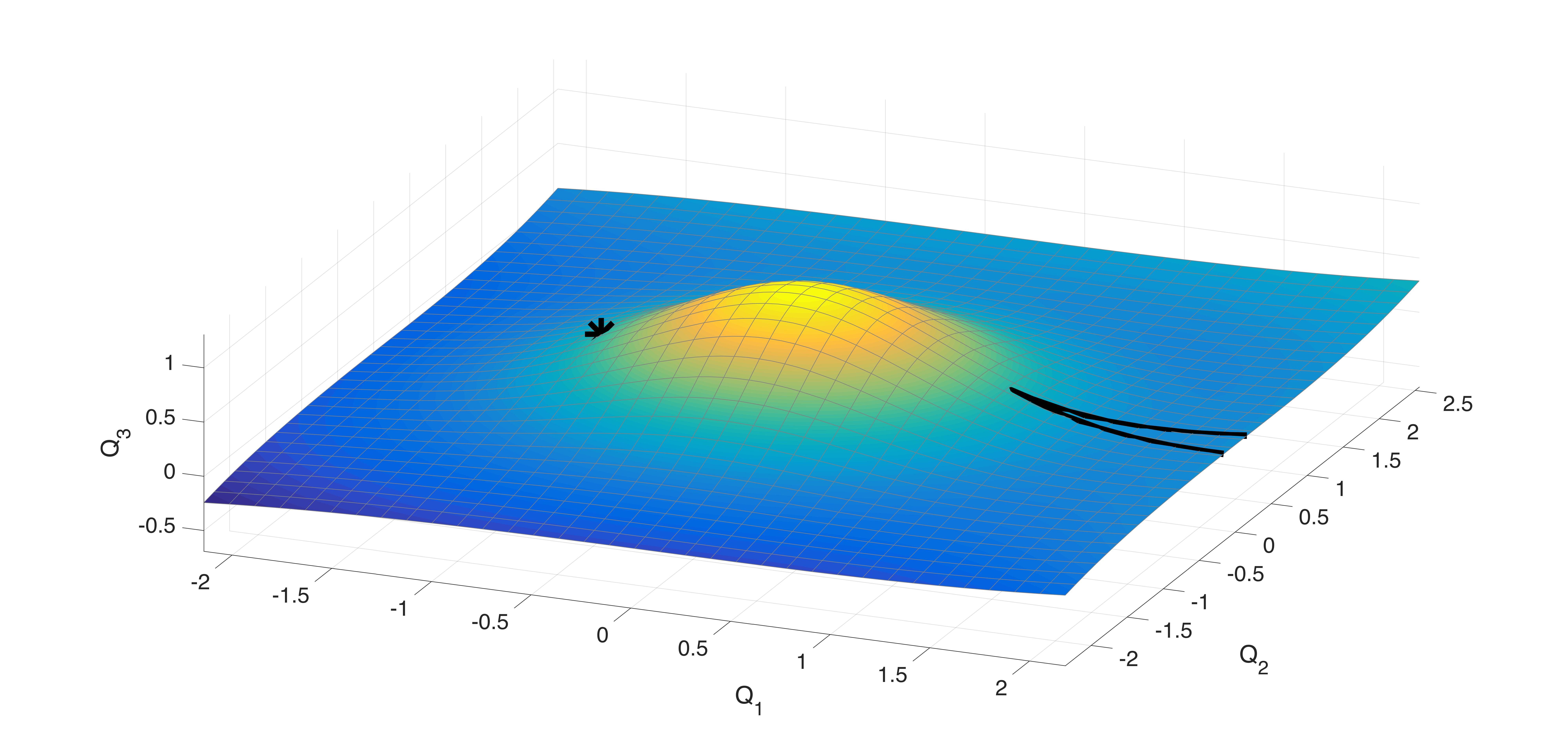}
\end{center}
\caption{Conjugate locus to $q^\ast (-1,0,h(-1,0))$ on the graph of a perturbed 2-dimensional Gaussian. The conjugate locus contains a cusp singularity.}\label{fig:CuspsGraph}
\end{figure}

\begin{figure}
\begin{center}
\includegraphics[width=0.4\textwidth]{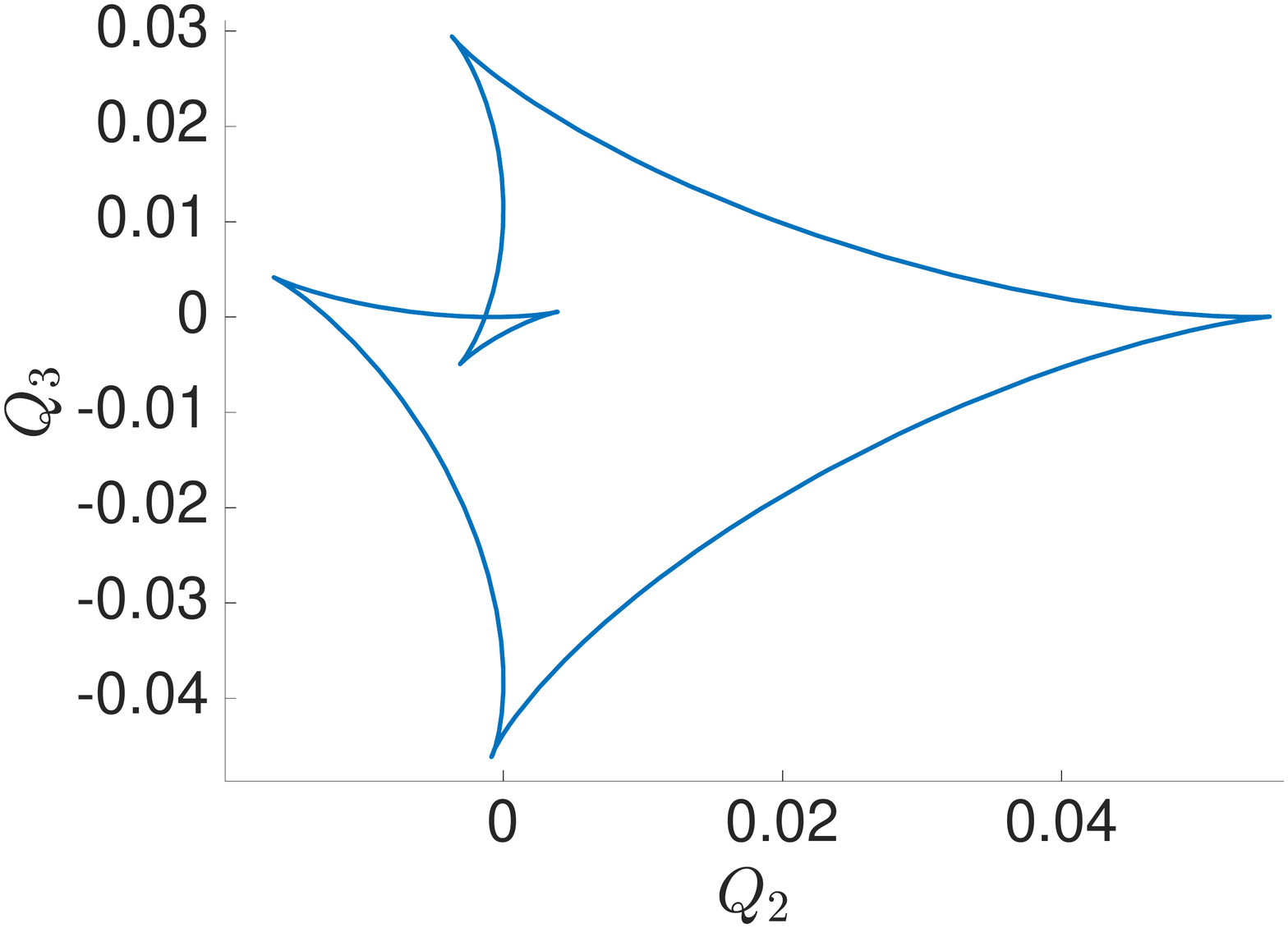}
\quad
\includegraphics[width=0.4\textwidth]{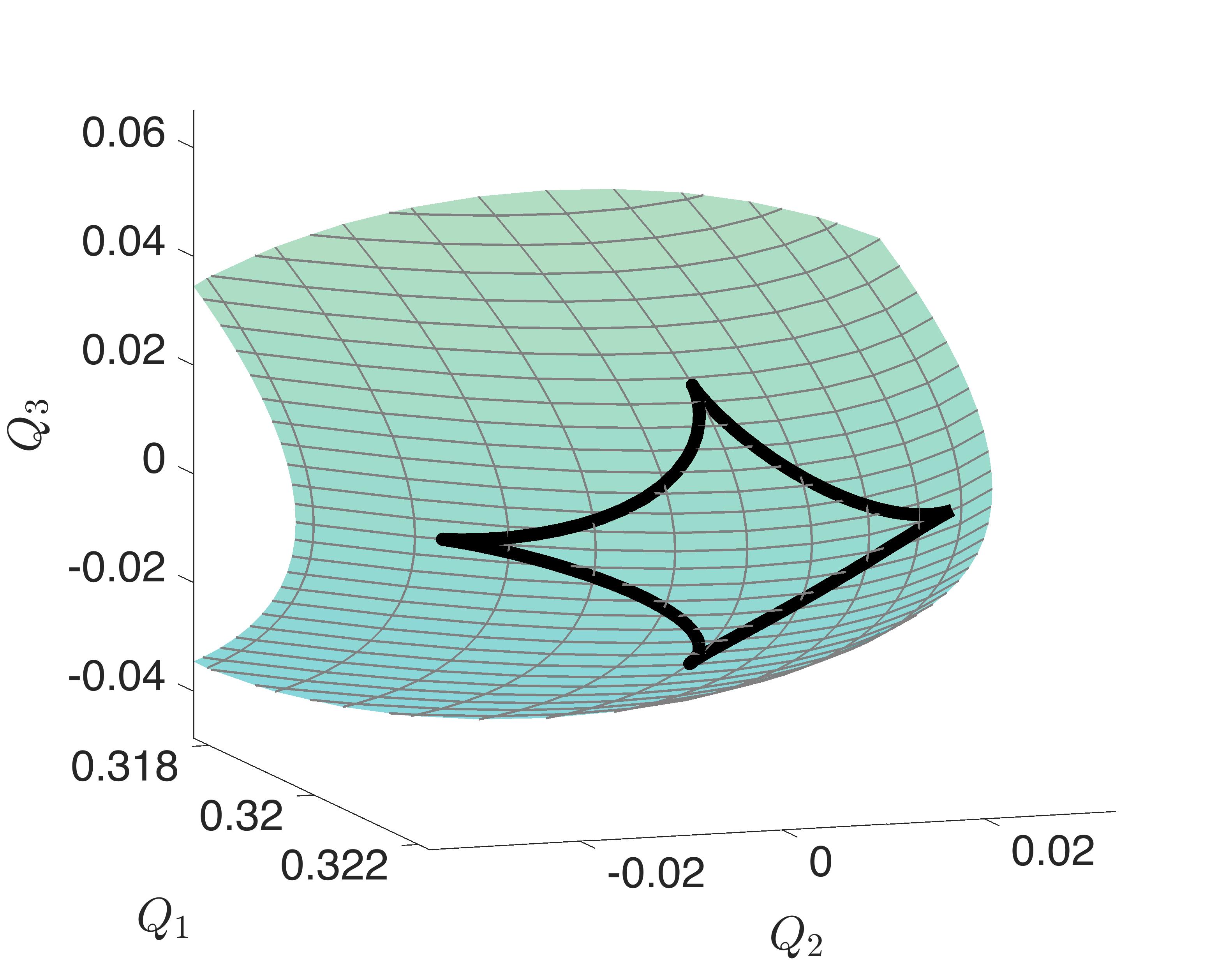}
\end{center}
\caption{Conjugate locus on a perturbed ellipsoid (left) and on an unperturbed ellipsoid (right). We see formations of cusps connected by lines of fold singularities.}\label{fig:CuspsGraphSphere}\label{fig:CuspsGraphEllipsoid}
\end{figure}


\begin{figure}
\begin{center}
\includegraphics[width=0.32\textwidth]{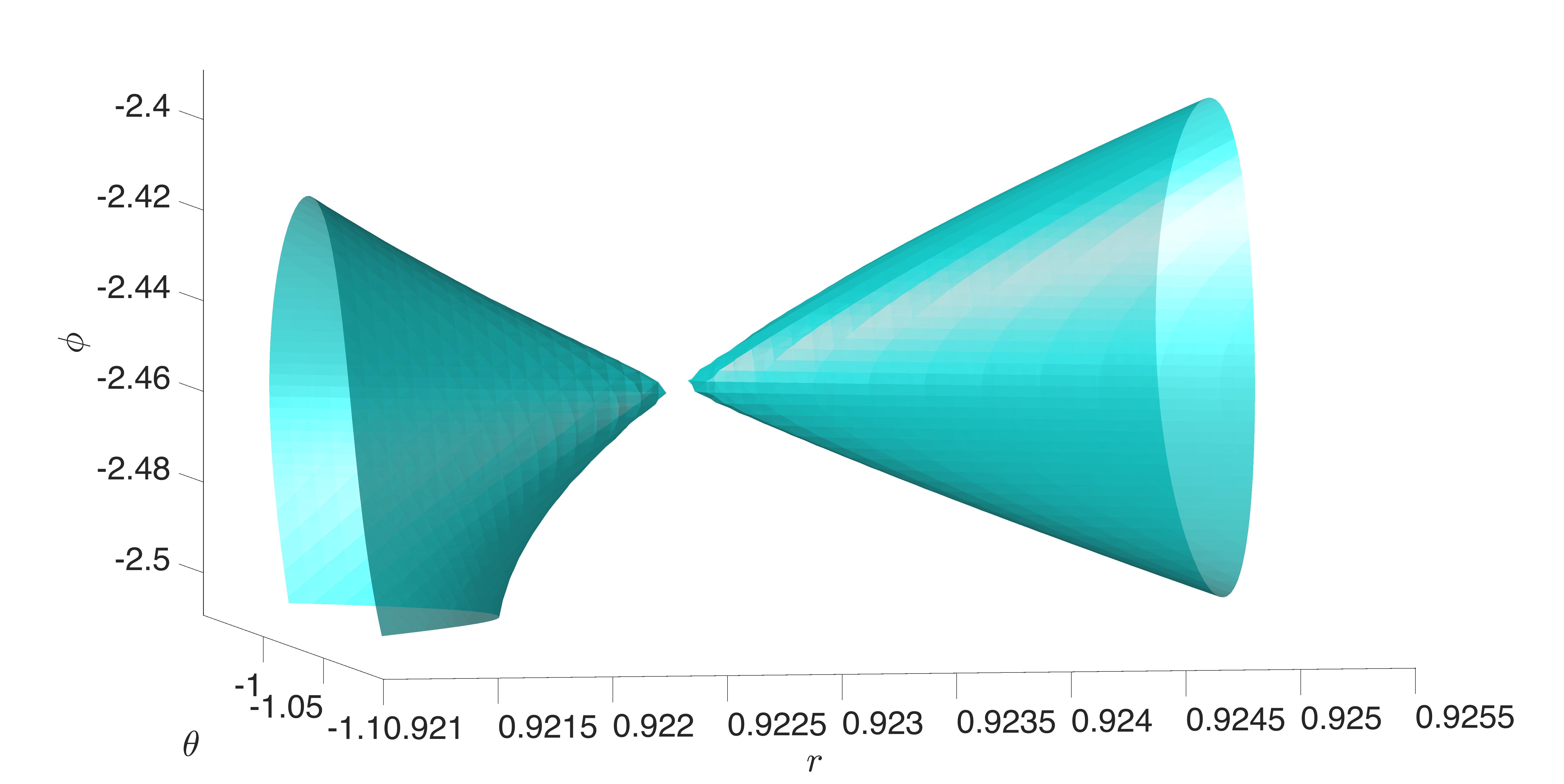}
\includegraphics[width=0.32\textwidth]{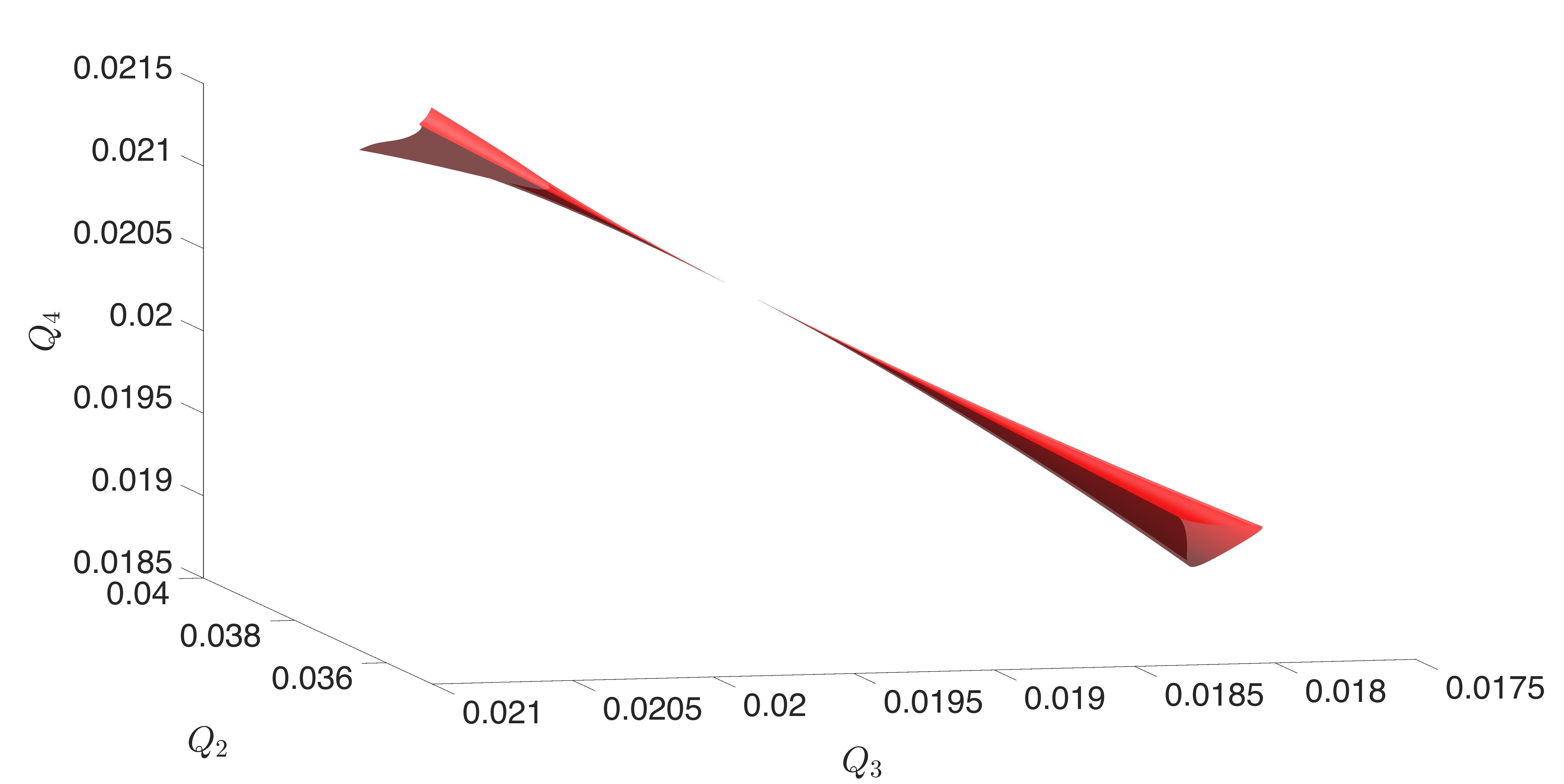}
\includegraphics[width=0.32\textwidth]{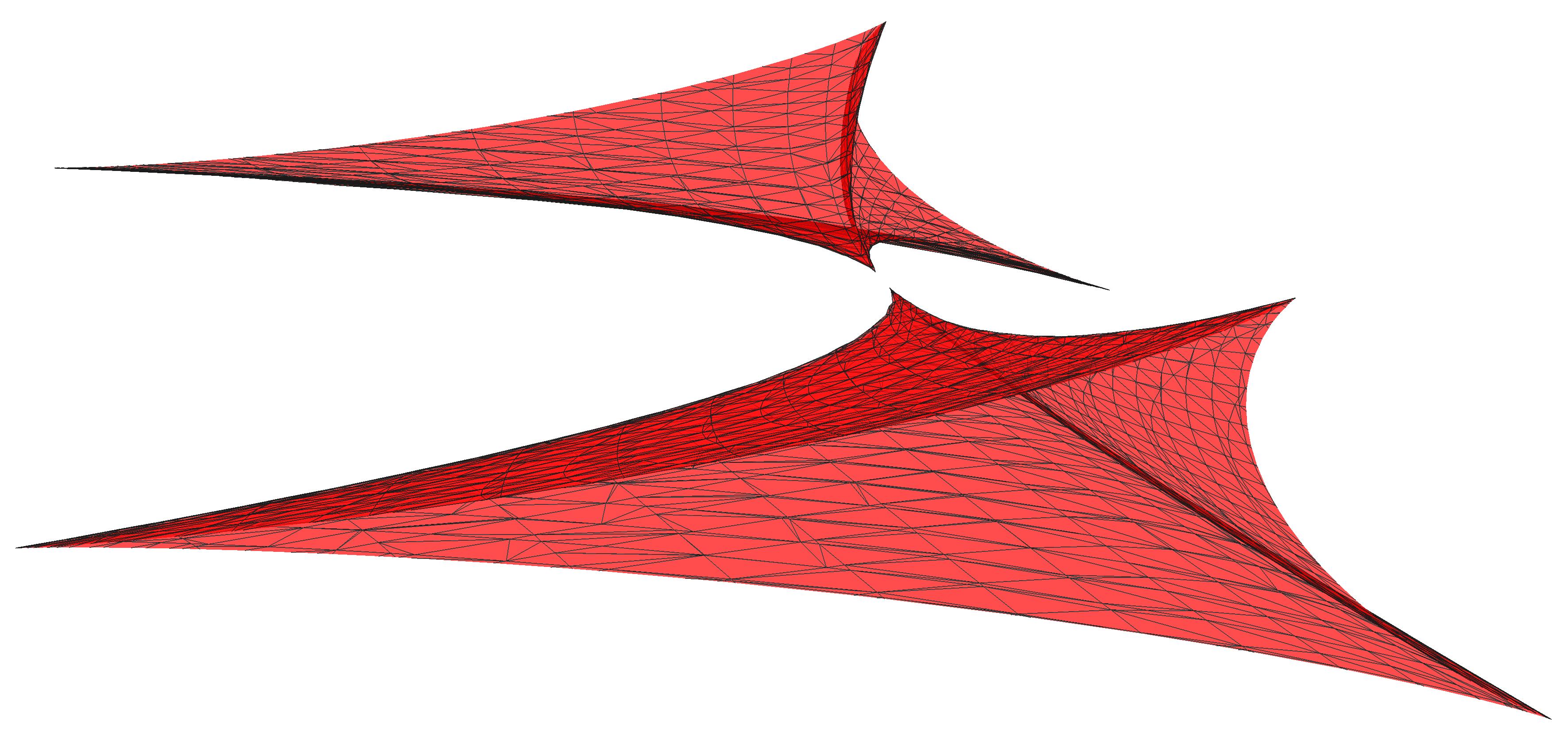}
\end{center}
\caption{Degeneracy of the geodesic exponential map near an approximately anti-podal point on a perturbed 3-dimensional ellipsoid. 
The plot to the left shows the position of the singularities in a parametrisation of the phase space, the plot in the centre the level bifurcation set (which is the conjugate locus) in the $Q_2$, $Q_3$, $Q_4$ coordinates  showing an elliptic umbilic bifurcation. The right plot is obtained from the middle plot by rotation allowing for a convenient rescaling of the axes. }\label{fig:ellipticumbilichypersphere}
\end{figure}

%% file: pitchfork_capture.tex
\section{Capturing periodic pitchfork bifurcations}\label{sec:capturePitchfork}

The minimal number of parameters in a family of problems such that a singularity is generic, i.e.\ unremovable under small perturbations, depends on the class of systems considered. For example, we have shown that a $D^{\pm}_4$ singularity occurs generically in Hamiltonian boundary value problems with 3 parameters.
In a boundary value problem for a flow map without any extra (e.g.\ symplectic) structure a $D^{\pm}_4$ singularity needs at least 4 parameters to become generic.
Restricting the class of systems further, e.g.\ to those with certain symmetries and/or integrals of motion, the count of required parameters can change. Here we consider a special singularity which occurs generically in 1-parameter families of completely integrable\footnote{A $2n$-dimensional Hamiltonian system is completely integrable if it possesses $n$ functionally independent, Poisson commuting integrals of motion.} Hamiltonian systems, e.g.\ planar, autonomous systems.


\subsection{Introduction and the mechanism of periodic pitchfork bifurcations}

\begin{definition}[symmetrically separated Lagrangian boundary value problem]
Let $(\phi,\Lambda,\Lambda')$ be a separated Lagrangian boundary value problem (definition \ref{def:LagDirproblem}). If $\Lambda = \Lambda'$ then the problem is referred to as a {\em symmetrically separated Lagrangian boundary value problem}.
\end{definition}

If $\phi$ is the time-$\tau$-map of a Hamiltonian system then
a motion is a solution if and only if it starts and ends after a fixed time $\tau$ on $\Lambda$. Homogeneous Dirichlet boundary conditions as in figure \ref{fig:BratuPhase} are instances of such boundary conditions.
As the authors prove in \cite[Thm. 3.2]{bifurHampaper}, a periodic pitchfork bifurcation (see figure \ref{fig:bifurpitchHam}) is a generic phenomenon in 1-parameter families of boundary value problems in completely integrable Hamiltonian systems with symmetrically separated Lagrangian boundary conditions.
The bifurcation occurs generically where $\Lambda$ touches a Liouville torus of the system filled with orbits of period $\tau$. The mechanism in a general, high dimensional integrable system is illustrated in figure \ref{fig:periodicpitchforkmechanism}. Figure \ref{fig:orbtstobifurpitchHam} shows the planar case where the Liouville tori are closed orbits.

\begin{figure}
\includegraphics[width=0.325\textwidth]{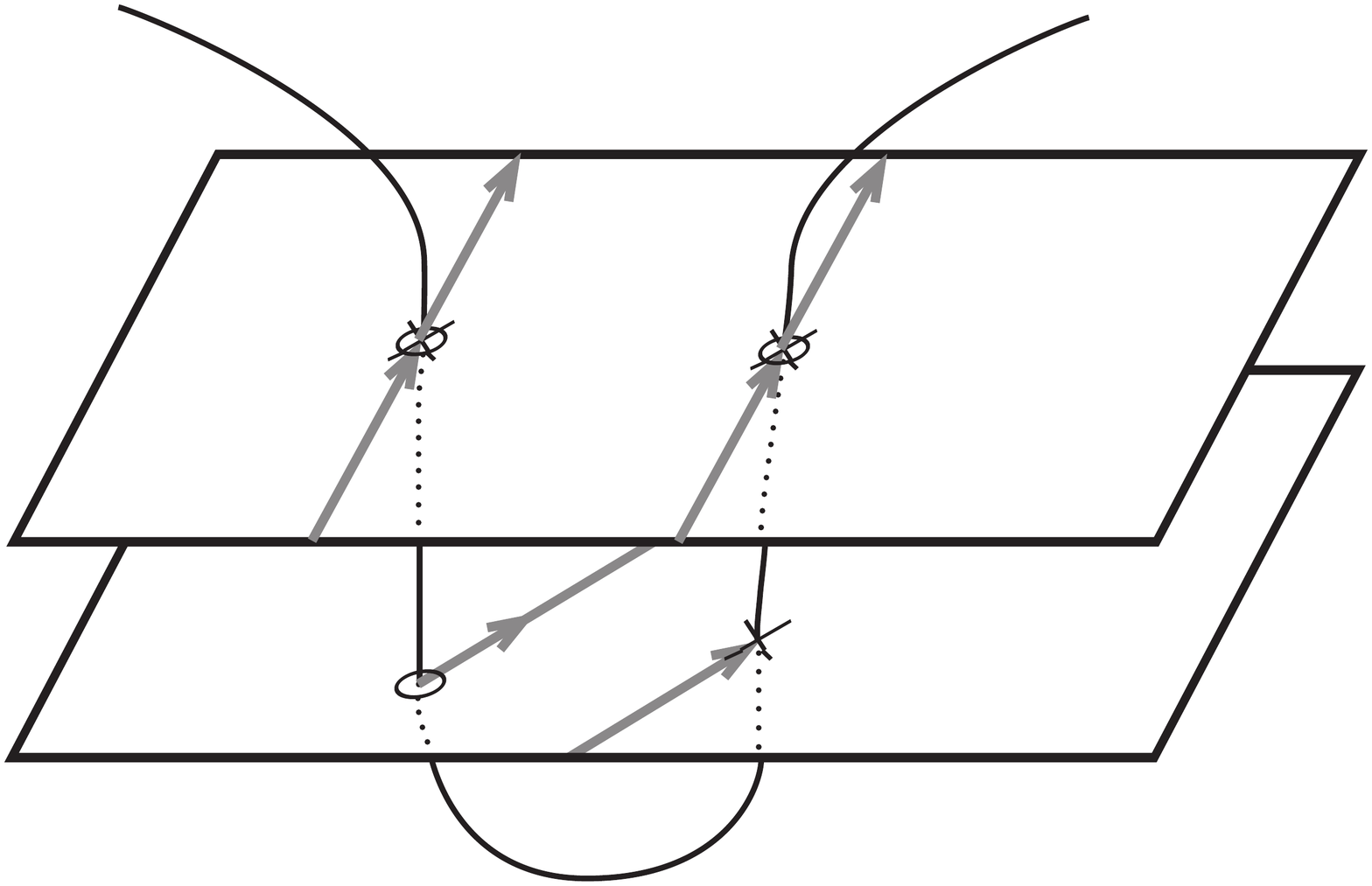}
\includegraphics[width=0.325\textwidth]{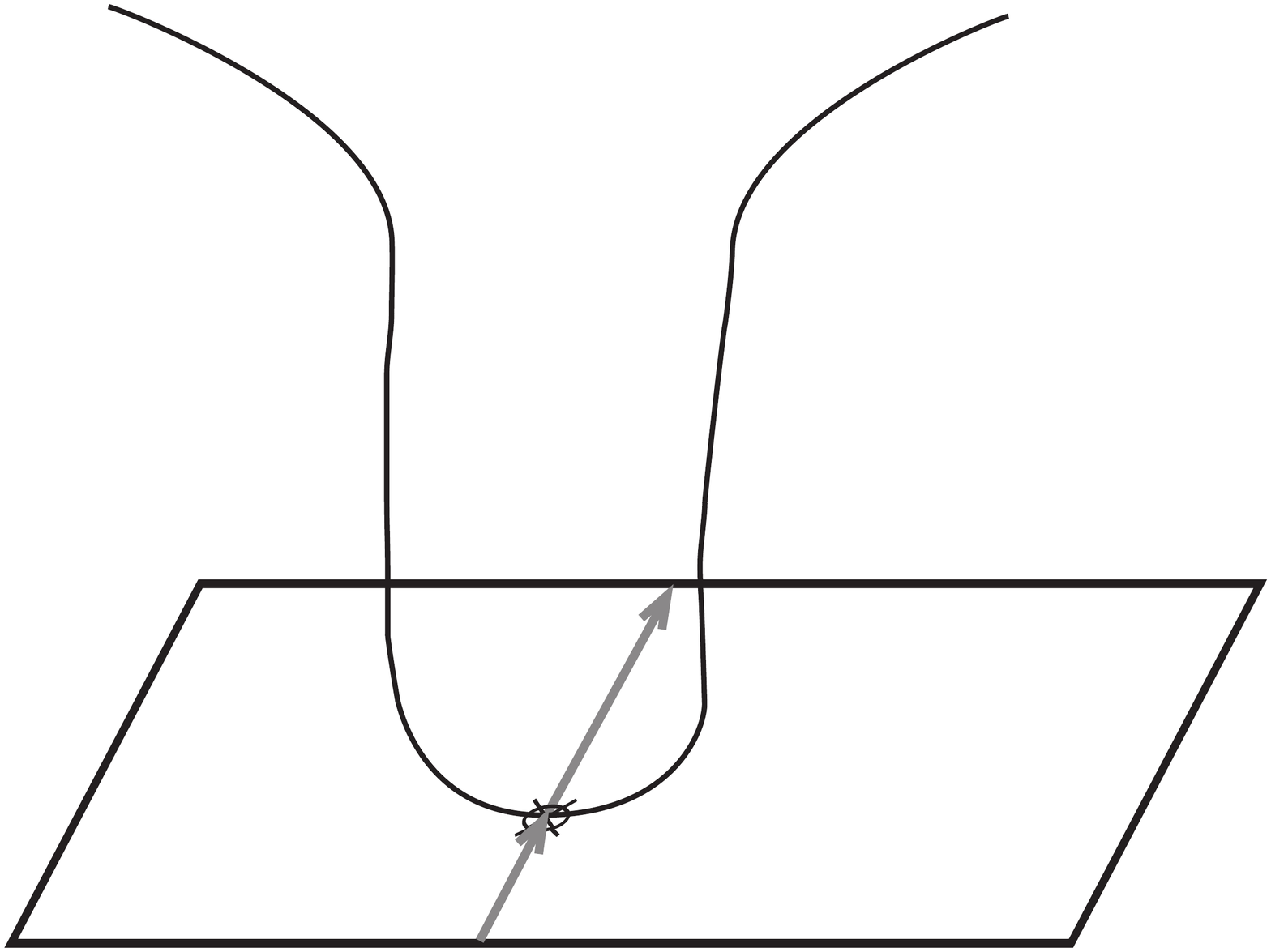}
\includegraphics[width=0.325\textwidth]{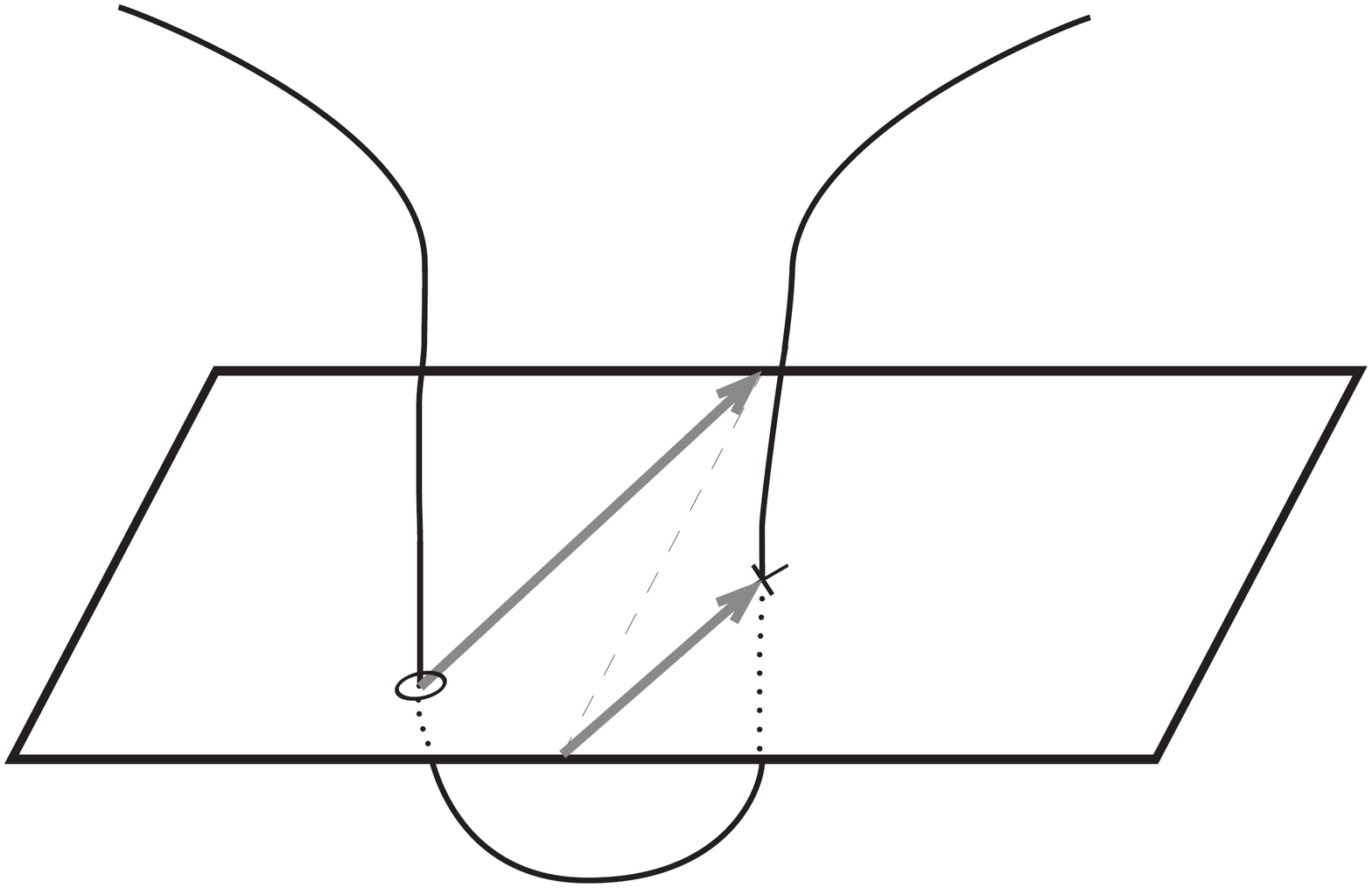}
\caption{Illustration of the mechanism of a periodic pitchfork bifurcation in a 4-dim system. Opposite edges of the parallelograms are identified illustrating Liouville tori. Arrows starting at points marked with $o$ and ending at $\times$ denote solution orbits of the boundary value problem.
Since two 2-dimensional manifolds in a 4-dimensional space generically intersect in isolated points, the boundary condition $\Lambda$ is represented as a curve in the illustrations.
Two periodic solutions of period $\tau$ lying on the same Liouville torus merge with a non-periodic solution exactly where the boundary condition is tangent to the Liouville torus with orbits of period $\tau$. Then a non-periodic solution persists. \cite{bifurHampaper}}\label{fig:periodicpitchforkmechanism}
\end{figure}


\begin{figure}
\begin{center}
\includegraphics[width=0.6\textwidth]{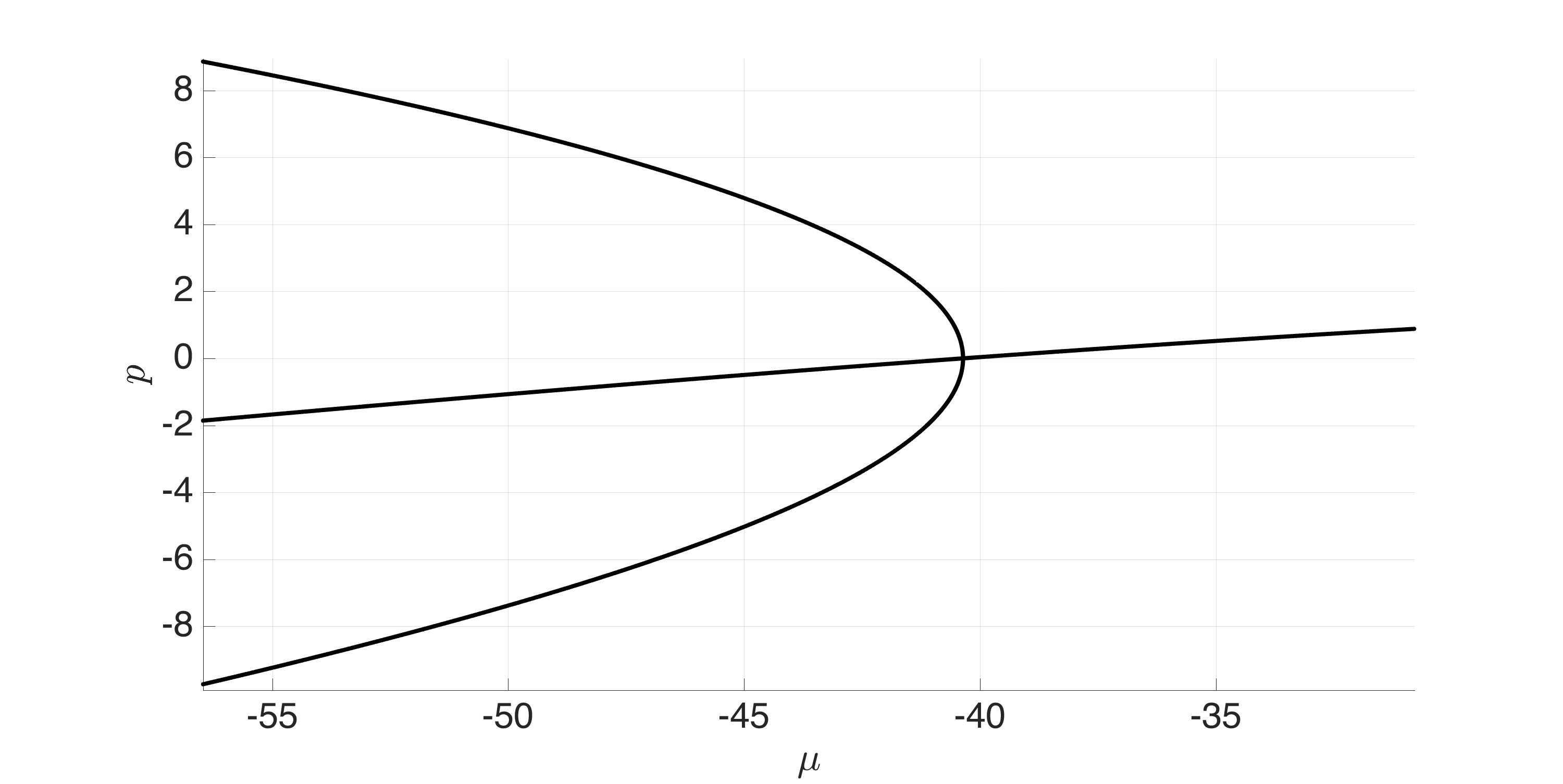}\\

\includegraphics[width=0.325\textwidth]{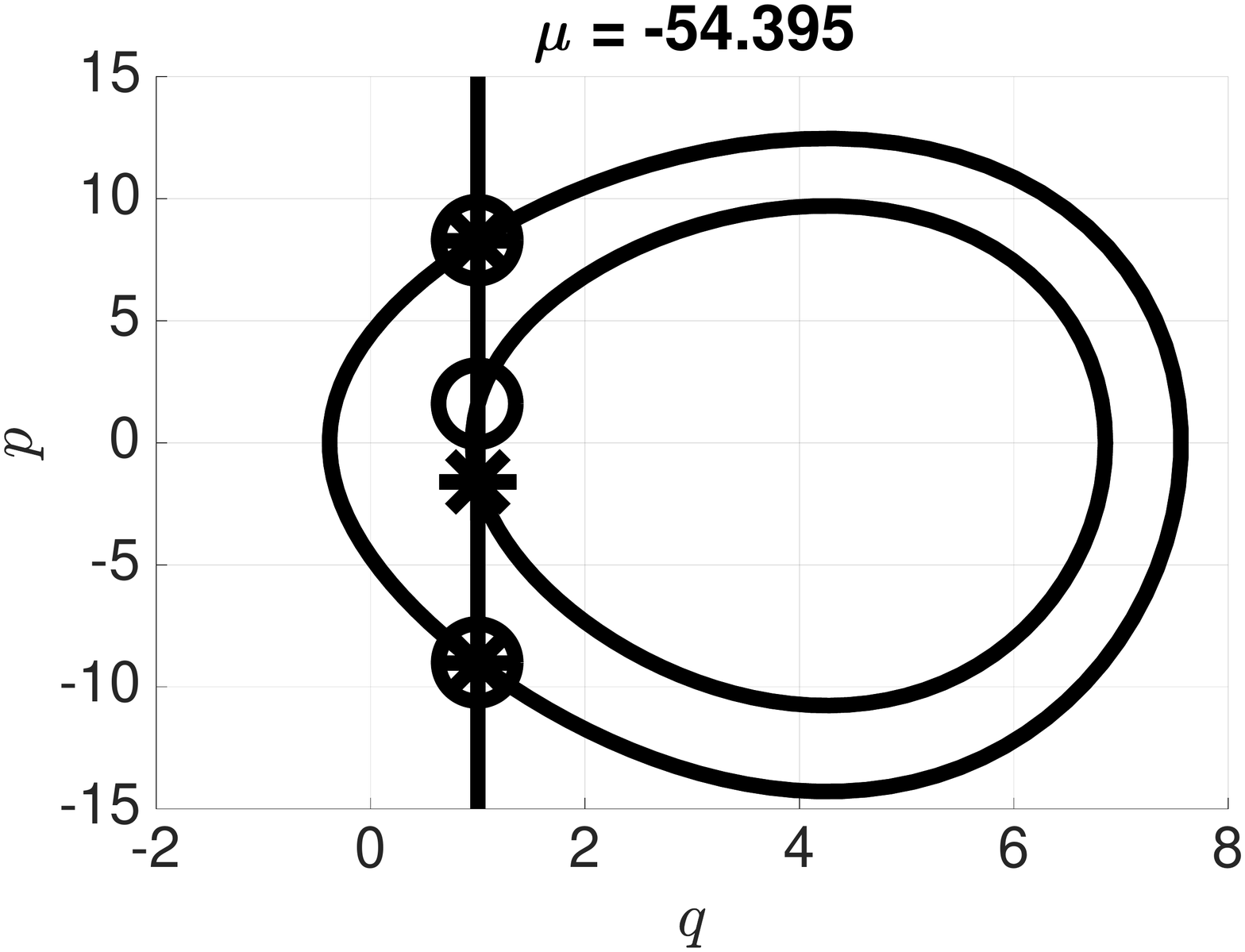}
\includegraphics[width=0.325\textwidth]{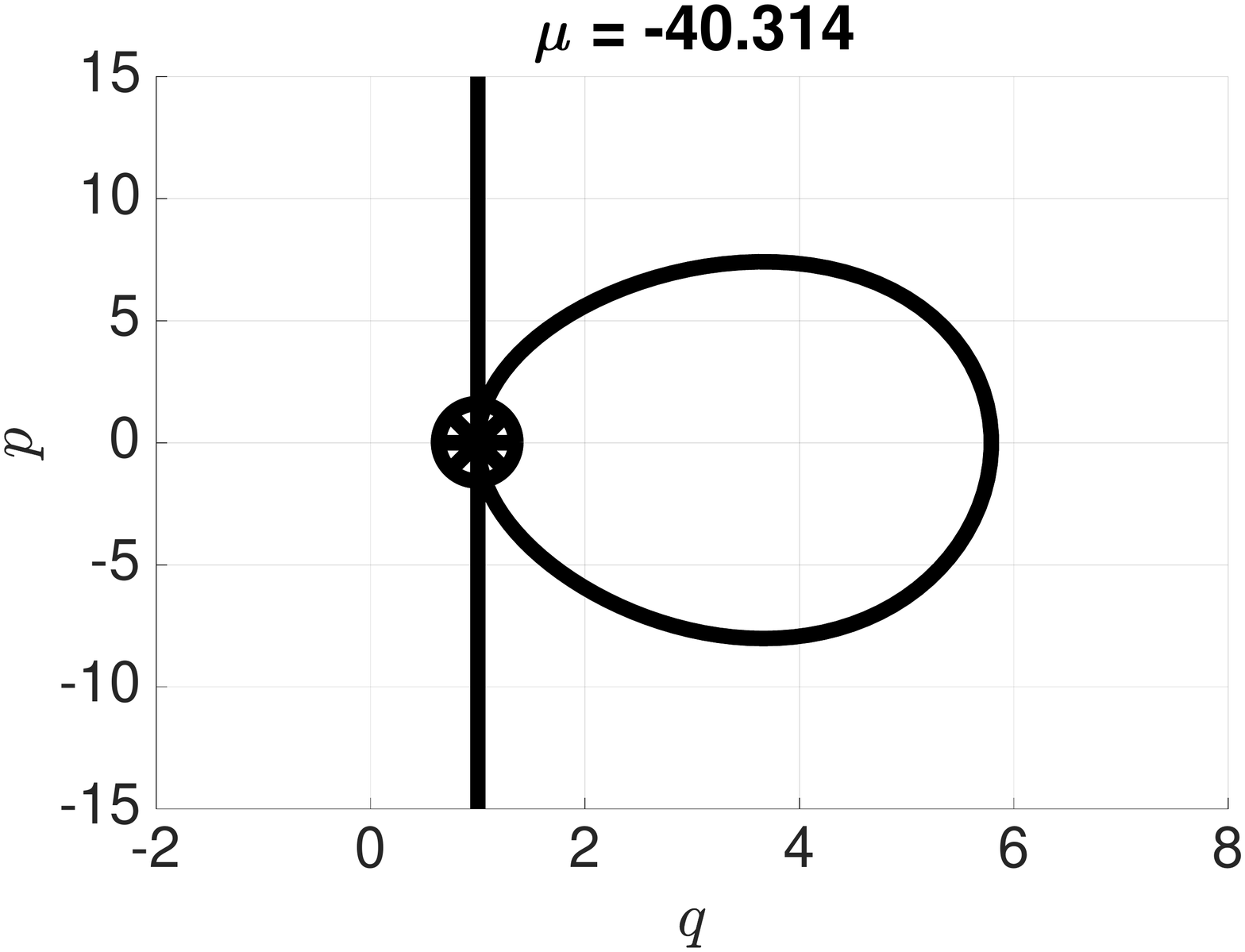}
\includegraphics[width=0.325\textwidth]{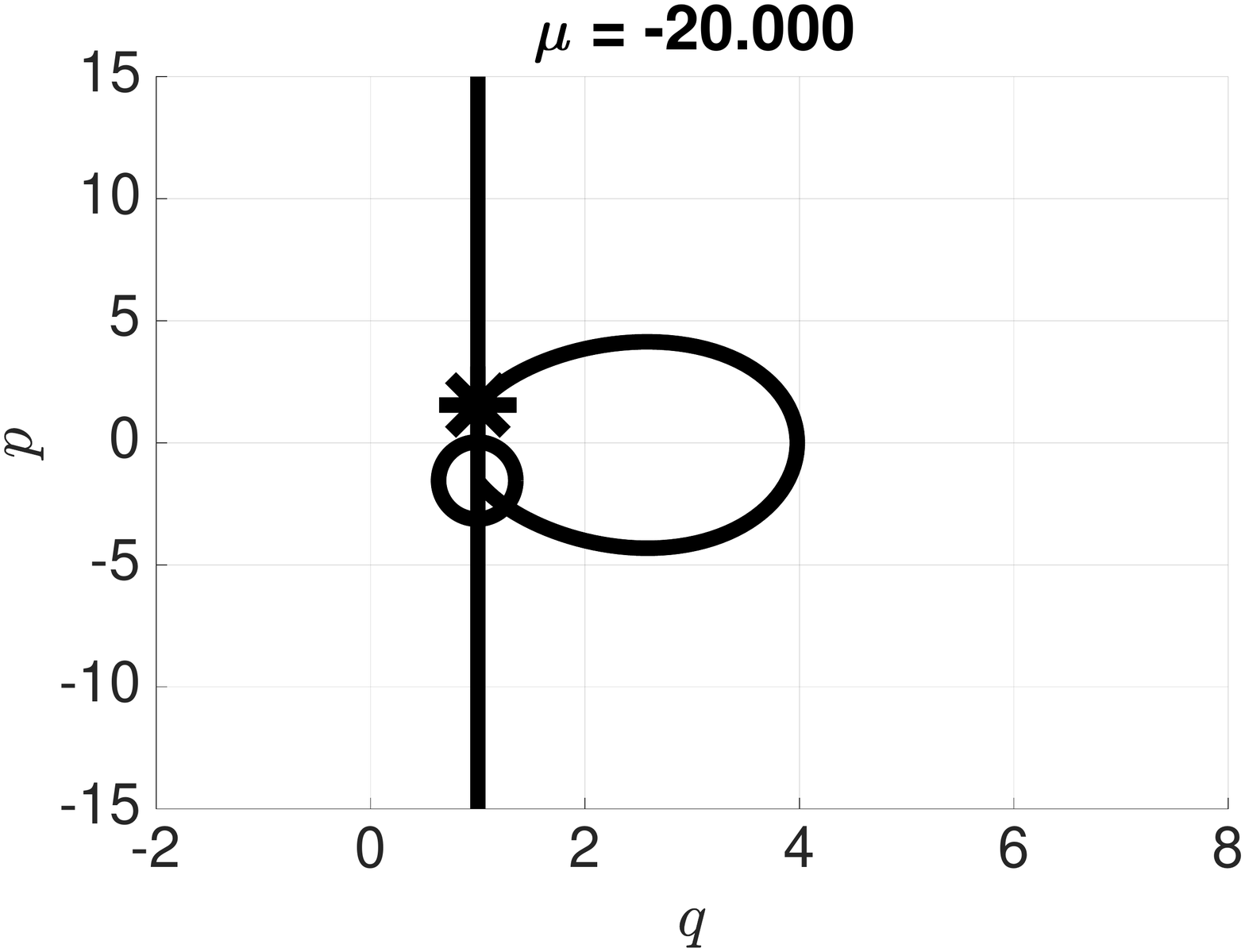}
\end{center}
\caption{Bifurcation diagram of the boundary value problem $q(0)=1=q(1)$ for the Hamiltonian system $H_\mu(q,p)=p^2+0.01p^3+q^3+\mu q$. The flow map is obtained using the symplectic St\"ormer-Verlet method and the boundary value problem is solved using a shooting method. Below the diagram the corresponding solution orbits are plotted for three parameter values. The marker $\ast$ denotes the start point of the orbit and $o$ the end point. Notice that the outer periodic orbit in the left figure constitutes two solutions to the boundary value problem. \cite[figure 3]{bifurHampaper}
}\label{fig:orbtstobifurpitchHam}\label{fig:bifurpitchHam}
\end{figure}

In the remainder of the section we will analyse how symplectic integrators can be useful to capture periodic pitchfork bifurcations. We will
\begin{itemize}
\item
show numerical examples comparing non-symplectic and symplectic methods in generic planar systems
and
describe the mechanism of the bifurcation in the numerical phase space obtained by a symplectic method in planar systems (section \ref{subsec:NumericalPhasePlotPitchfork}),
\item
construct an analytic example of a periodic pitchfork bifurcation in a non-trivial, 4-dimensional completely integrable system and run a numerical experiment studying how it is captured when using a symplectic integrator (section \ref{subsec:modelnontrivialsym}),
\item
run further numerical experiments on how periodic pitchfork bifurcations are captured in higher dimensional systems with affine-linear symmetries (section \ref{subsec:SymmetricCapture}),
\item
give theoretical reasoning for the observed behaviour and explain a general setting (section \ref{subsec:TheoCapture}).
\end{itemize}

\input{pitchfork_capture_numerical_example}
\input{pitchfork_capture_discretisation}

%% file: pitchfork_capture_numerical_example.tex

\subsection{Periodic pitchfork in planar Hamiltonian systems -- quality of preservation and numerical phase plots for a symplectic integrator}\label{subsec:NumericalPhasePlotPitchfork}

Consider a $\mu$-parameter family of Hamiltonian systems with phase space $T^\ast\R$ equipped with the symplectic structure $\d q \wedge \d p$ and Hamiltonians
\[
H_\mu(q,p) = p^2 + 0.1p^3 - 0.01 \cos(p)  + q^3-0.01q^2+\mu q.
\]
Let $Q_\mu$ denote the $q$ component of the Hamiltonian flow to $H_\mu$ at time $\tau=1.7$. We consider the Dirichlet-type boundary value problem 
\[Q_\mu(0.2,p)=0.2.\] 
A motion starting at the line $q=0.2$ in the phase space is a solution to the boundary value problem if and only if it returns to the line after time $\tau=1.7$.
Figure \ref{fig:pitchfork2d} shows how a pitchfork bifurcation in a Dirichlet problem for a generic, 1-parameter family of planar Hamiltonian systems is captured by the symplectic St\"ormer-Verlet  method with 14 and 28 steps. The breaking in the bifurcation for 28 steps is visible in a close-up of the bifurcation diagram. Notice the different scaling of the axes in the plots. We see that only few time-steps are needed to capture the bifurcation very well. The strong improvement of the shape of the pitchfork by doubling the number of steps indicates a convergence to the correct shape which is better than polynomial. Indeed, exponential convergence will be proved in proposition \ref{prop:cappitch}.


\begin{figure}
\begin{center}
\includegraphics[width=0.4\textwidth]{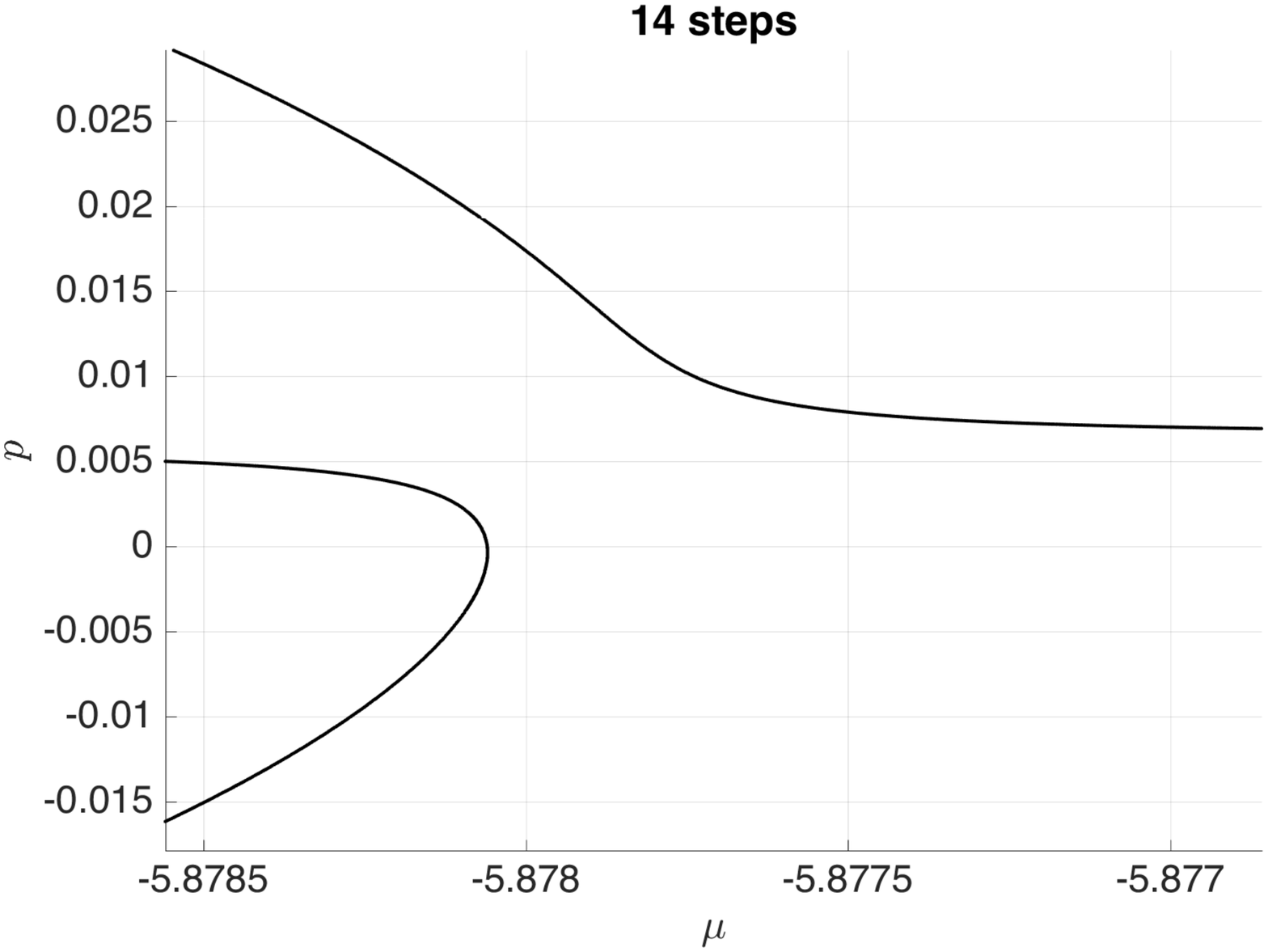}
\quad
\includegraphics[width=0.4\textwidth]{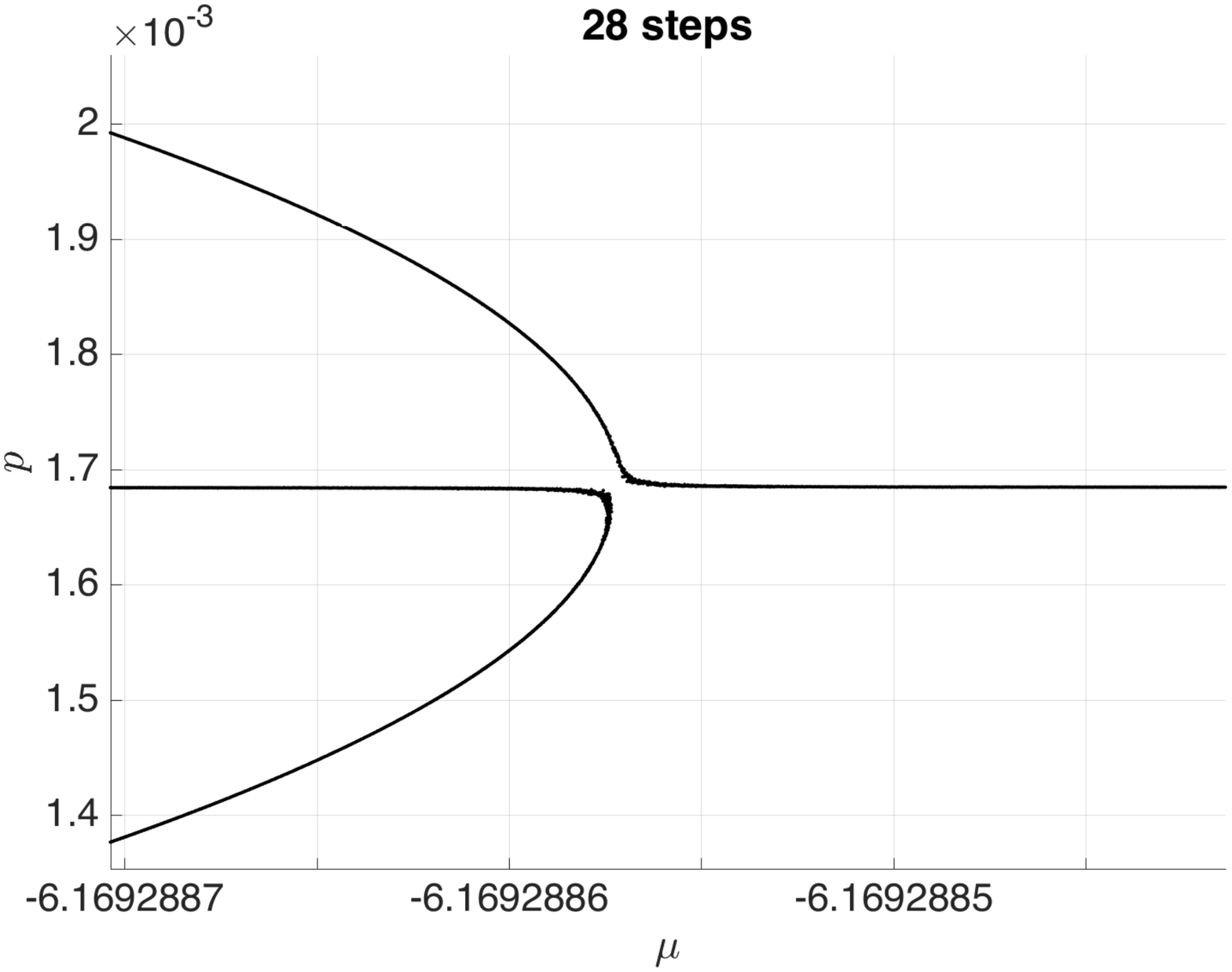}
\end{center}
\caption{Plot of bifurcation diagrams of Hamiltonian boundary value problem solved with the symplectic St\"ormer-Verlet method using 14 steps and 28 steps. Notice the different scaling of the axes.}\label{fig:pitchfork2d}
\end{figure}
%
For the matter of visualisation of the mechanism in the phase space, we reduce the number of time-steps to 11. The bifurcation diagram is displayed in figure \ref{fig:LP11Pitchfork}. For small parameter values there are three solutions. As $\mu$ increases two of them merge in a fold bifurcation.
\begin{figure}
\begin{center}
\includegraphics[width=0.4\textwidth]{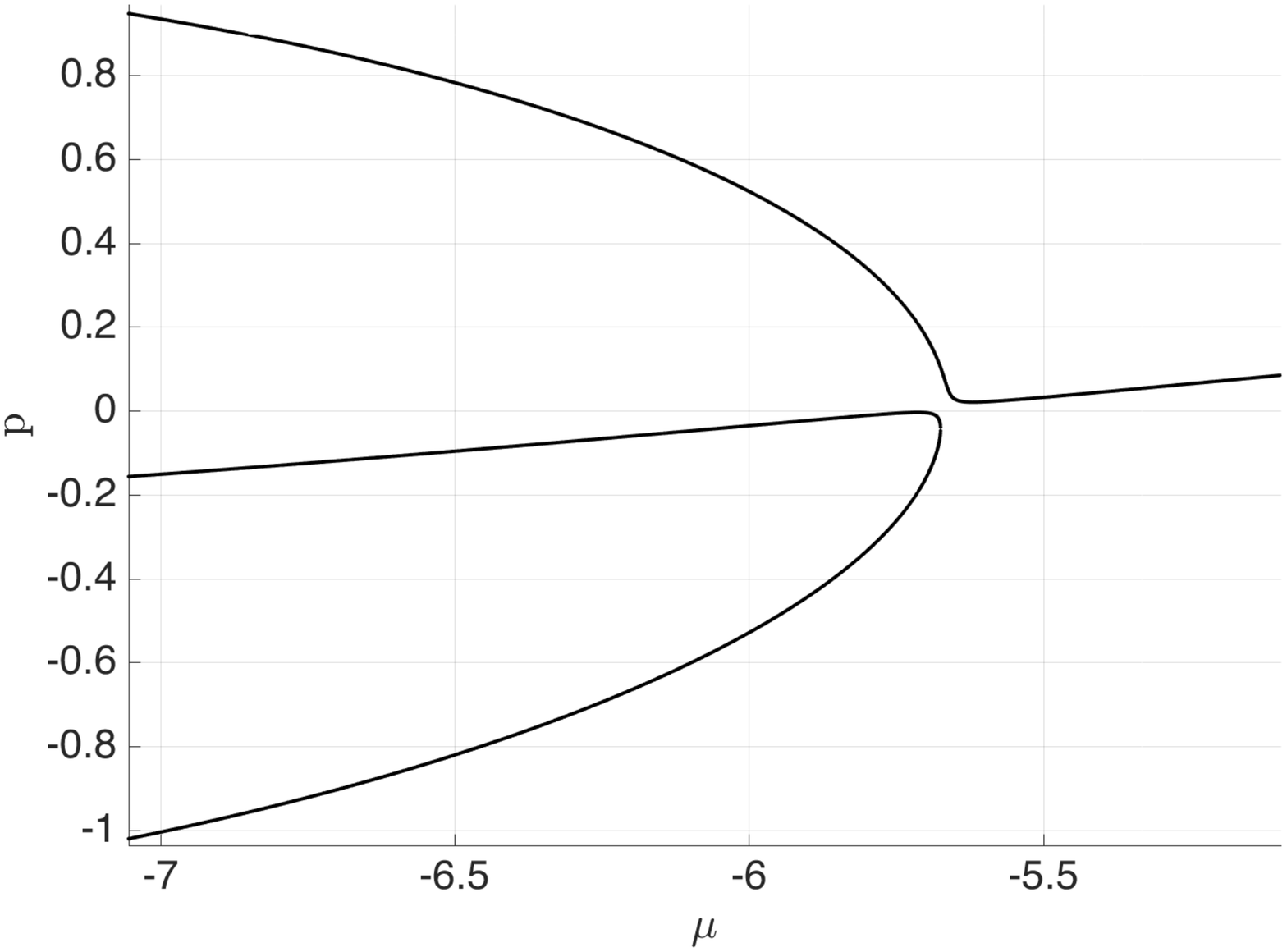}
\end{center}
\caption{Plot of bifurcation diagrams of Hamiltonian boundary value problem solved with the symplectic St\"ormer-Verlet method using 11 steps.}\label{fig:LP11Pitchfork}
\end{figure}
%
%
%
%
%
%
%
\begin{figure}
\begin{center}
\includegraphics[width=0.32\textwidth]{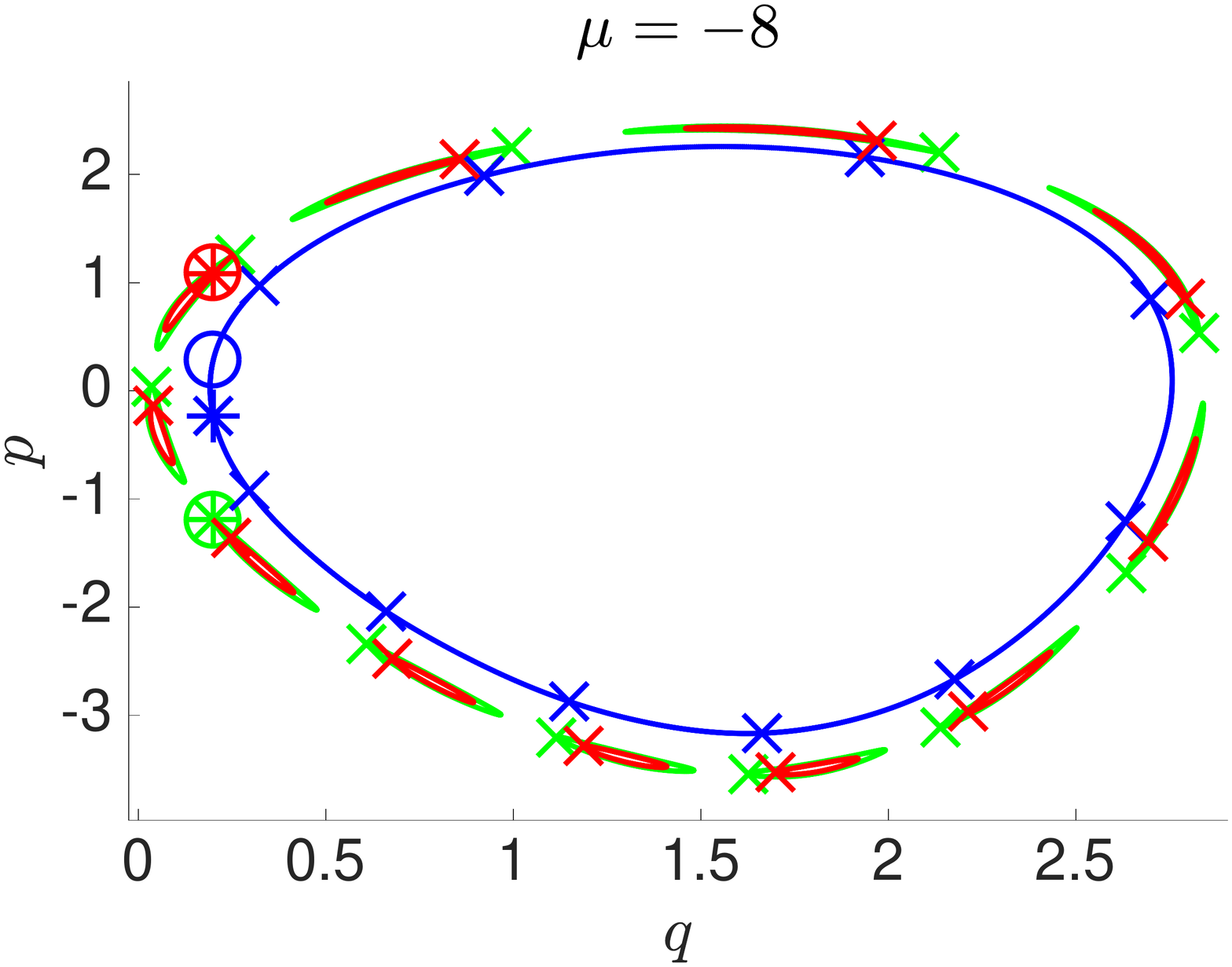}
\includegraphics[width=0.32\textwidth]{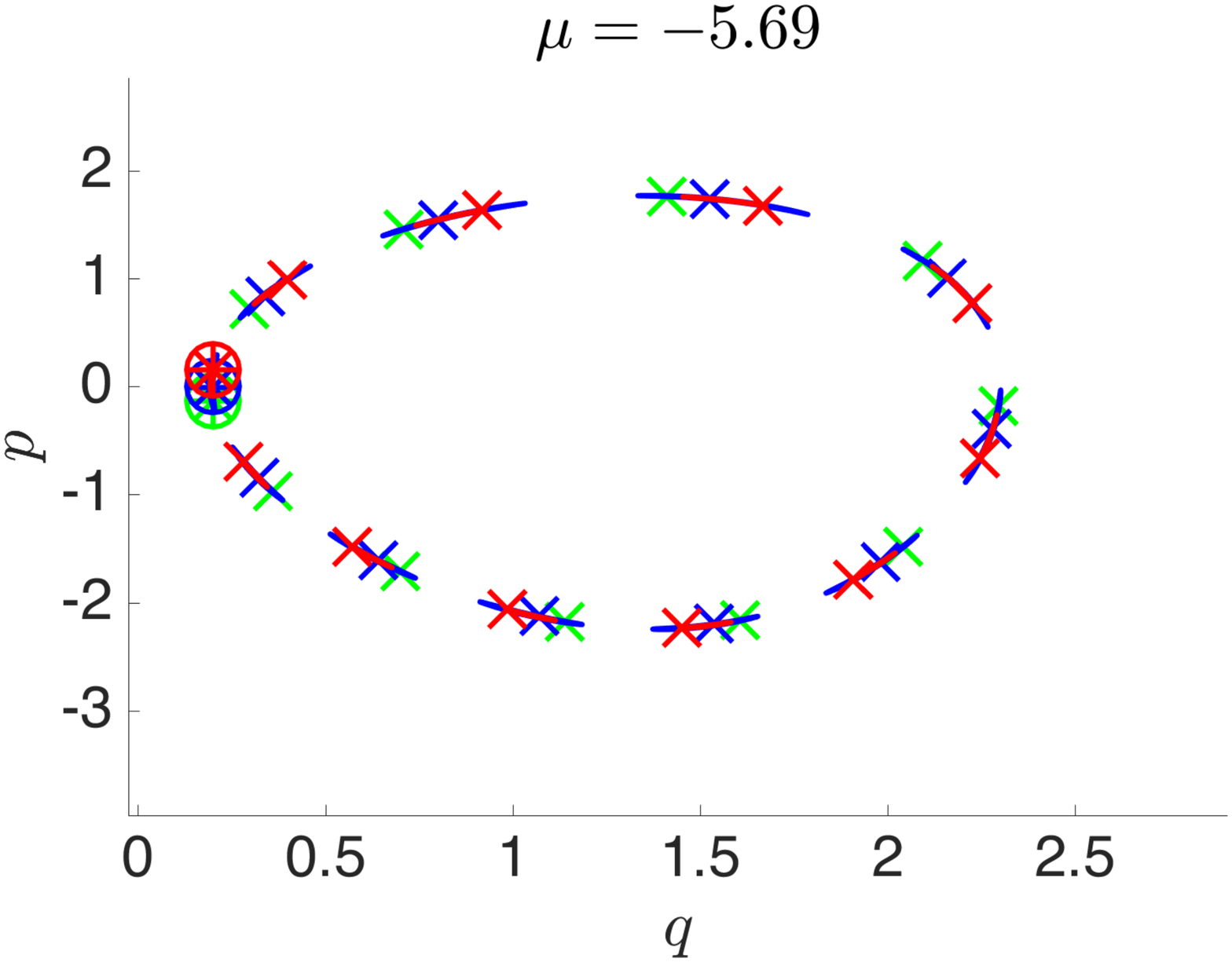}
\includegraphics[width=0.32\textwidth]{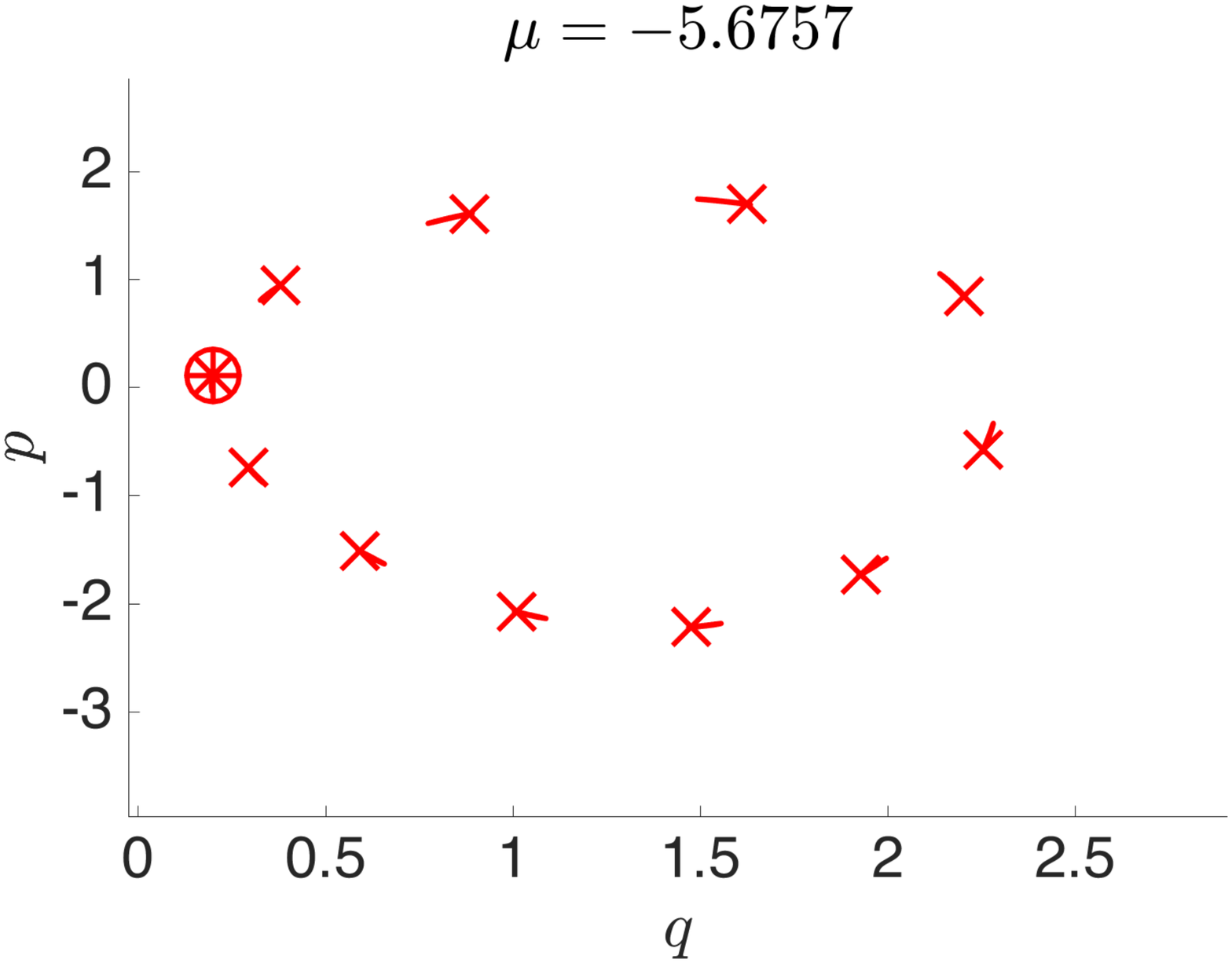}
\includegraphics[width=0.32\textwidth]{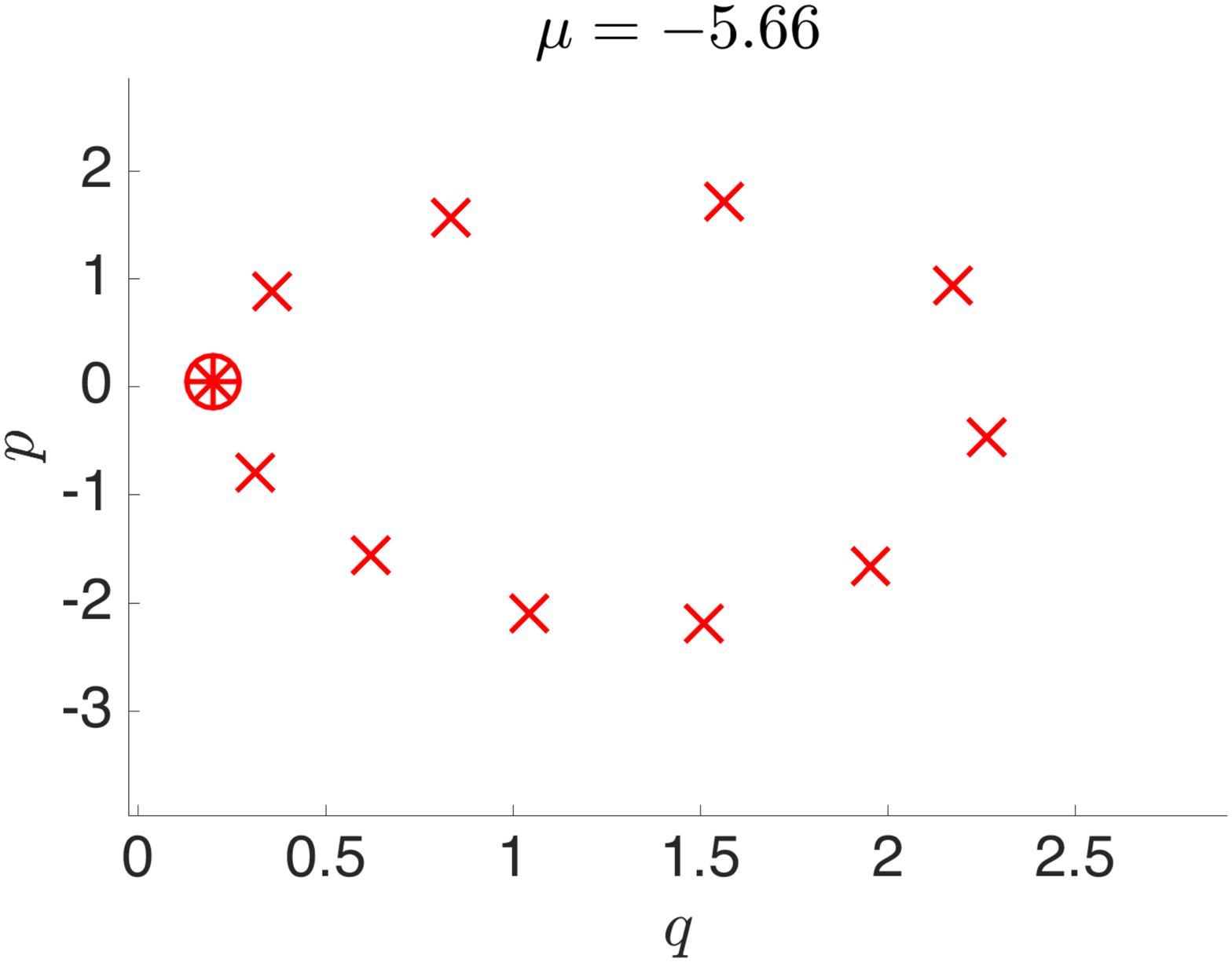}
\includegraphics[width=0.32\textwidth]{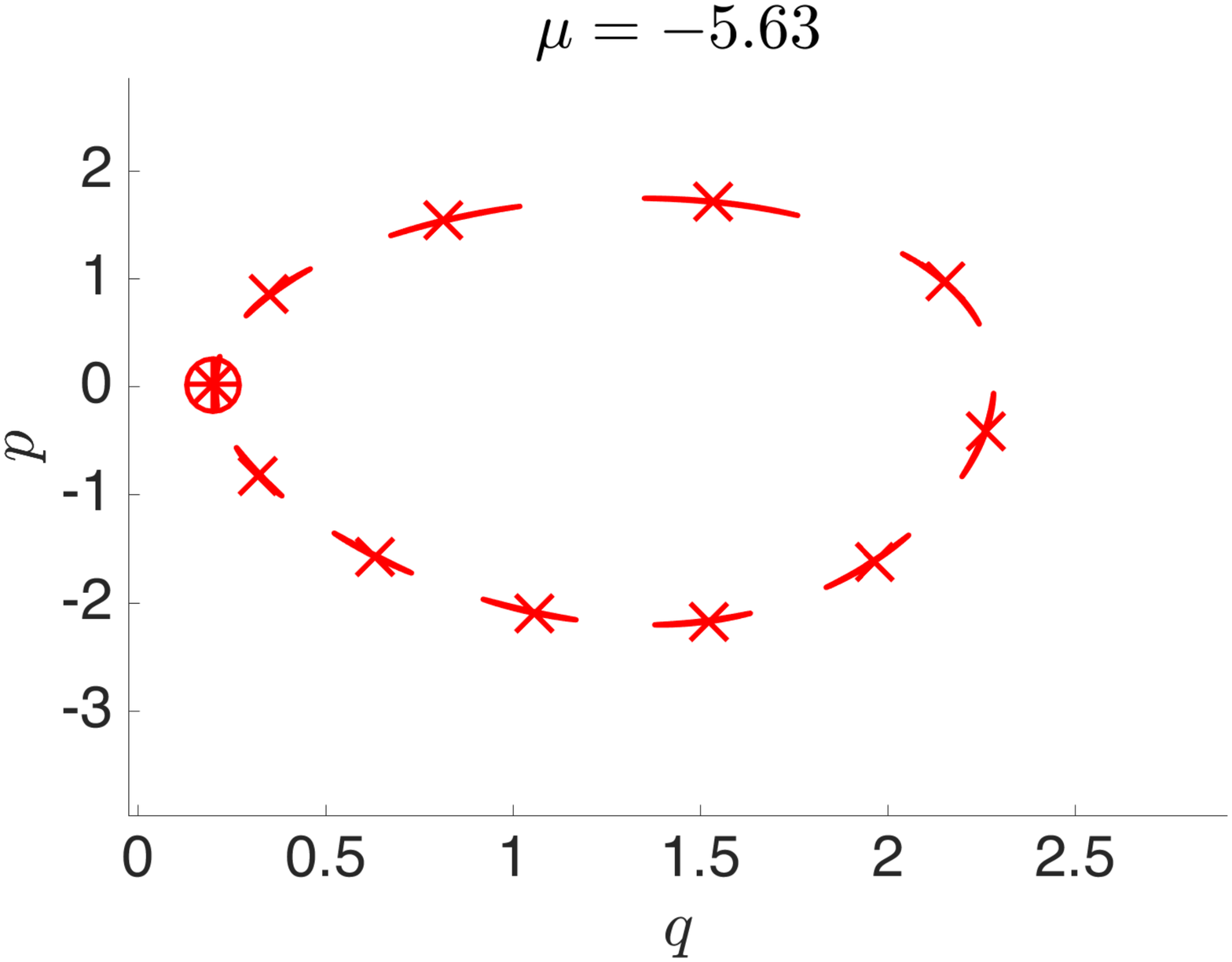}
\includegraphics[width=0.32\textwidth]{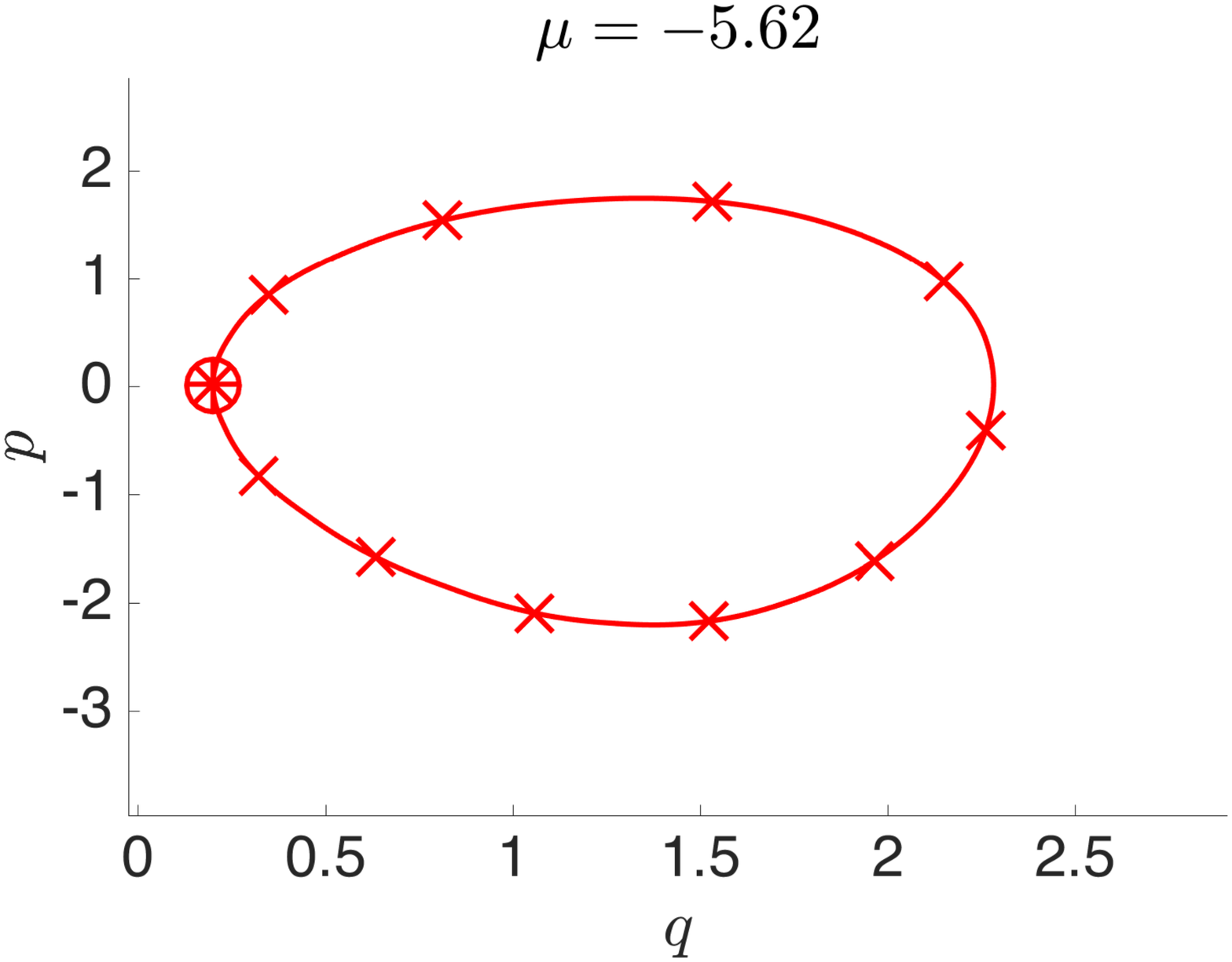}
\end{center}
\caption{Plot of orbits of the numerical flow involved in the bifurcation of figure \ref{fig:LP11Pitchfork}.}\label{fig:pitchfork2dmodphase}
\end{figure}
Figure \ref{fig:pitchfork2dmodphase} illustrates the mechanism of the broken pitchfork bifurcation in the phase space.
A solution to the boundary value problem is represented by 11 points (number of time-steps) in the phase space.  In the plots the initial point is marked by $\ast$, the end point by $o$ and the other points by $\times$. As required by the boundary condition, the $q$-coordinate of the initial- and end point is $0.2$.
The 11 points belong to an invariant set which can be calculated by iterating the discrete flow map $\phi_h$ with time-step $h=\tau/11$.
The solutions corresponding to the outer branches of the broken pitchfork bifurcation constitute nearly periodic orbits of $\phi_h$. Their invariant sets consist of 11 KAM-islands.
The inner branch corresponds to a non-periodic solution belonging to a connected invariant set.
As the parameter $\mu$ increases, the invariant set of the inner branch breaks up into 11 islands which merge with the invariant set of one of the nearly periodic solutions from the outer branches.
We see a fold bifurcation in the bifurcation diagram of the numerical flow.
The remaining outer branch is continued as $\mu$ increases.
It becomes a periodic orbit shortly after the fold bifurcation and then loses periodicity again.
This mechanism can be compared to the pitchfork bifurcation shown in figure \ref{fig:orbtstobifurpitchHam}. There the outer branches correspond to the same periodic orbit in the phase space and the pitchfork bifurcation takes place exactly when the orbit corresponding to the inner branch becomes a periodic orbit of period $\tau$ that is tangent to the boundary condition.

Let us compare the preservation of periodic pitchfork bifurcations using a symplectic integrator (figure \ref{fig:pitchfork2d} and \ref{fig:LP11Pitchfork}) with a non-symplectic integrator of the same order of accuracy. To understand what is happening in the latter case we compute a larger part of the bifurcation diagram using the explicit midpoint rule, which is a non-symplectic second order Runge-Kutta method (RK2).
The upper and middle branch of the pitchfork bifurcation do not exist in the numerical bifurcation diagram until we use more than 25 steps (figure \ref{fig:pitchforkRK}).
With 200 steps the bifurcation is recognisable and with 400 steps we obtain a diagram comparable with the 14-steps St\"ormer-Verlet calculation in figure \ref{fig:pitchfork2d}. As the computational costs per step are similar if the Hamiltonian is separated, we conclude that the symplectic St\"ormer-Verlet method performs significantly better then the non-symplectic method RK2. 

\begin{figure}
\begin{center}
\includegraphics[width=0.32\textwidth]{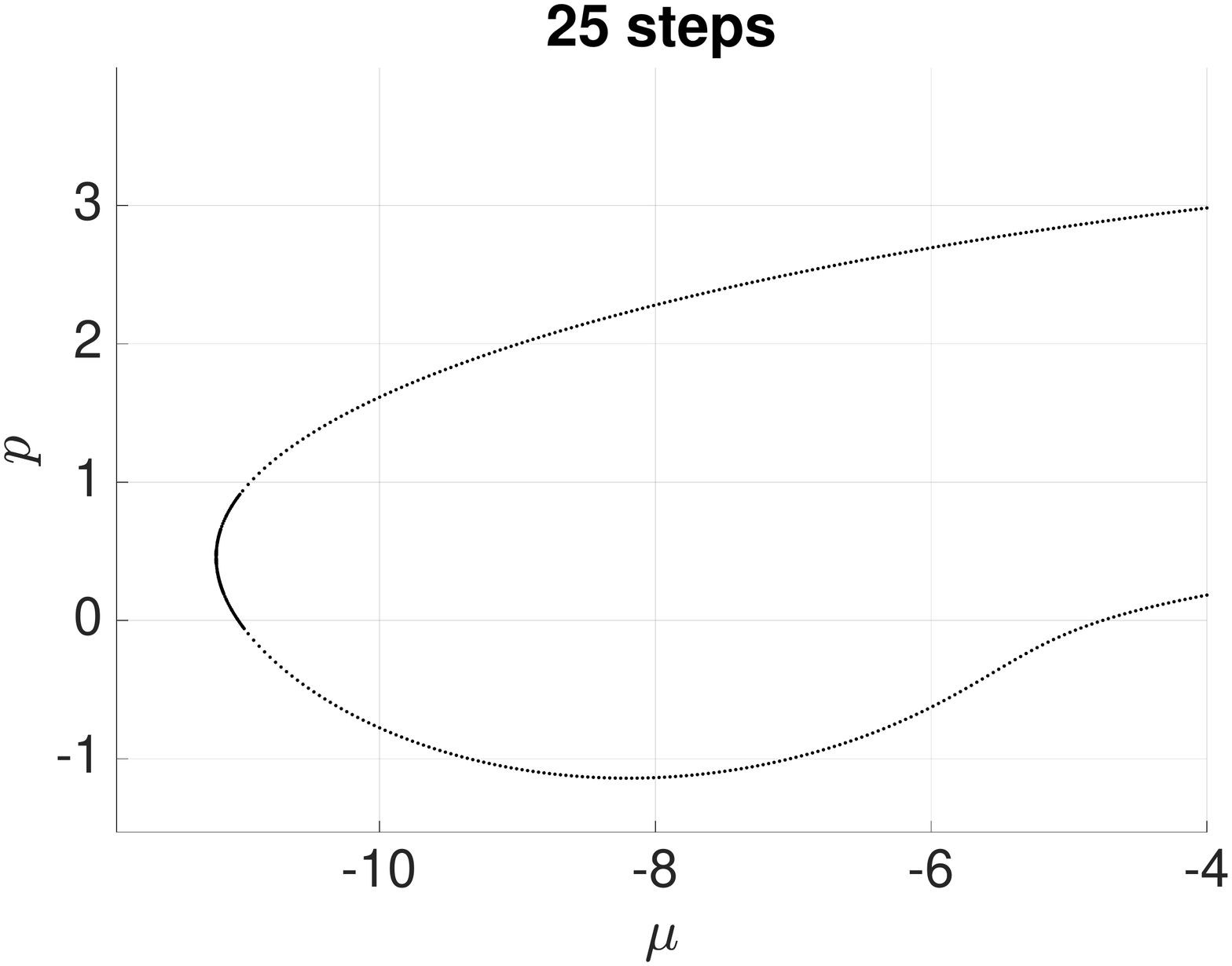}
\includegraphics[width=0.32\textwidth]{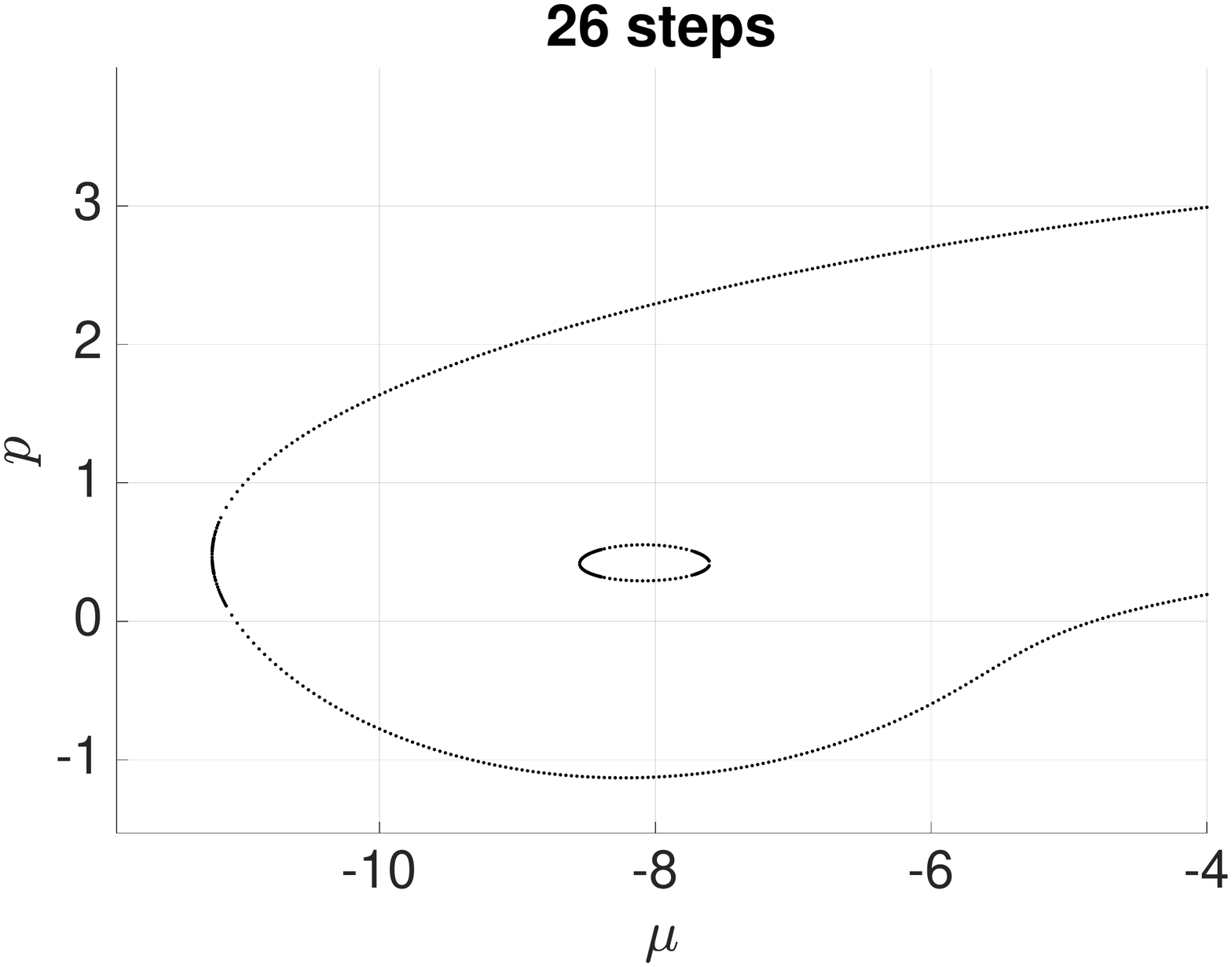}
\includegraphics[width=0.32\textwidth]{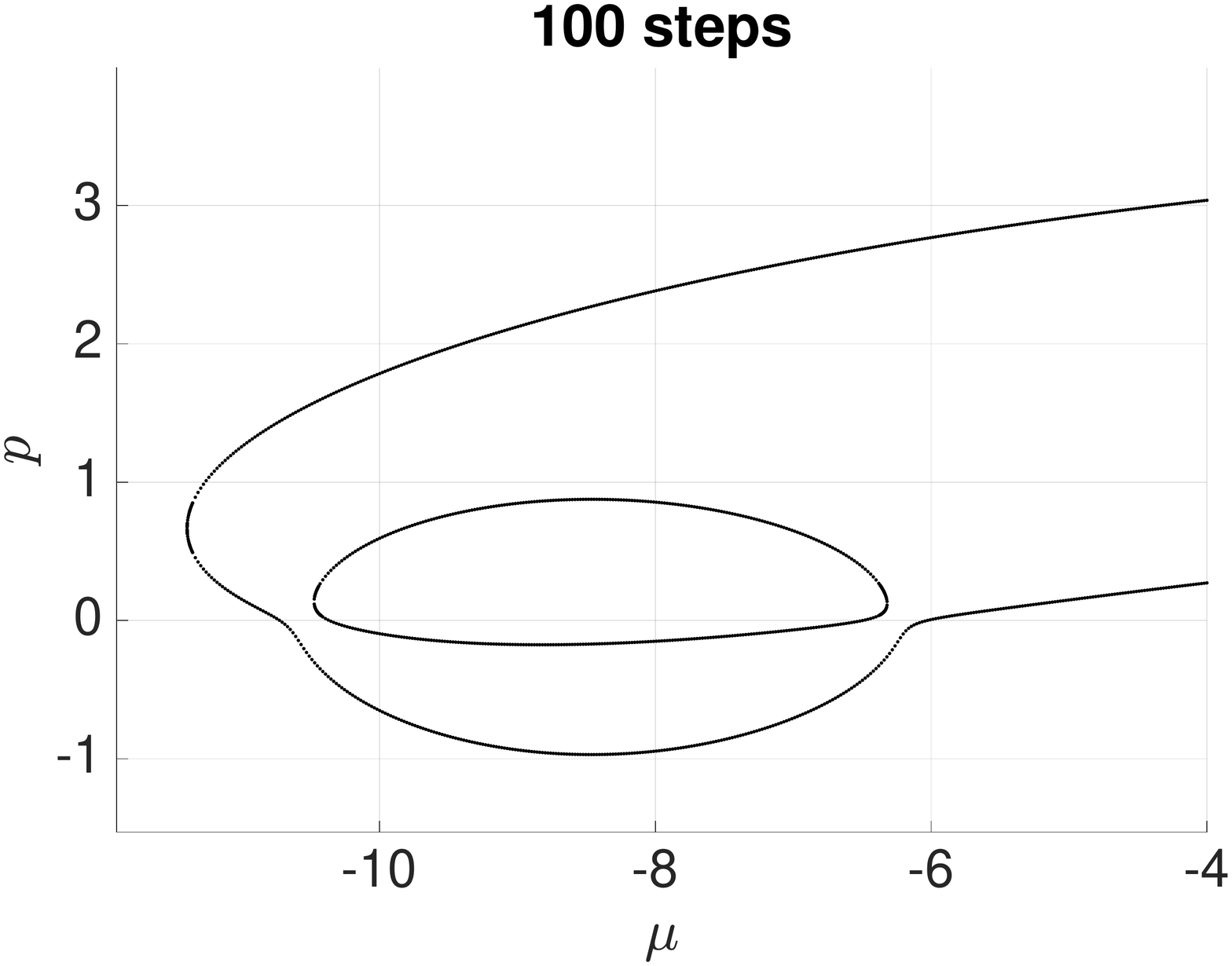}
\includegraphics[width=0.32\textwidth]{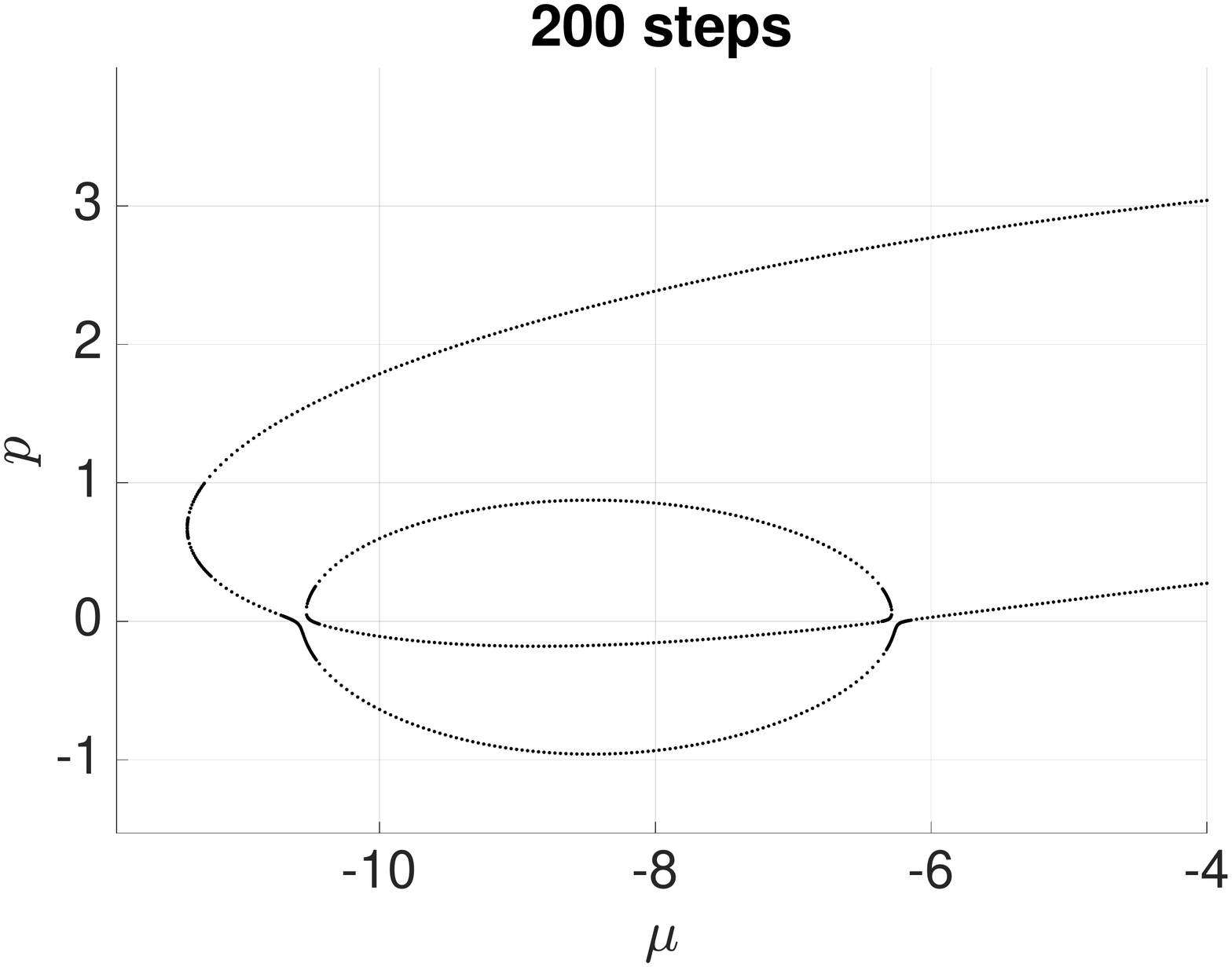}
\includegraphics[width=0.32\textwidth]{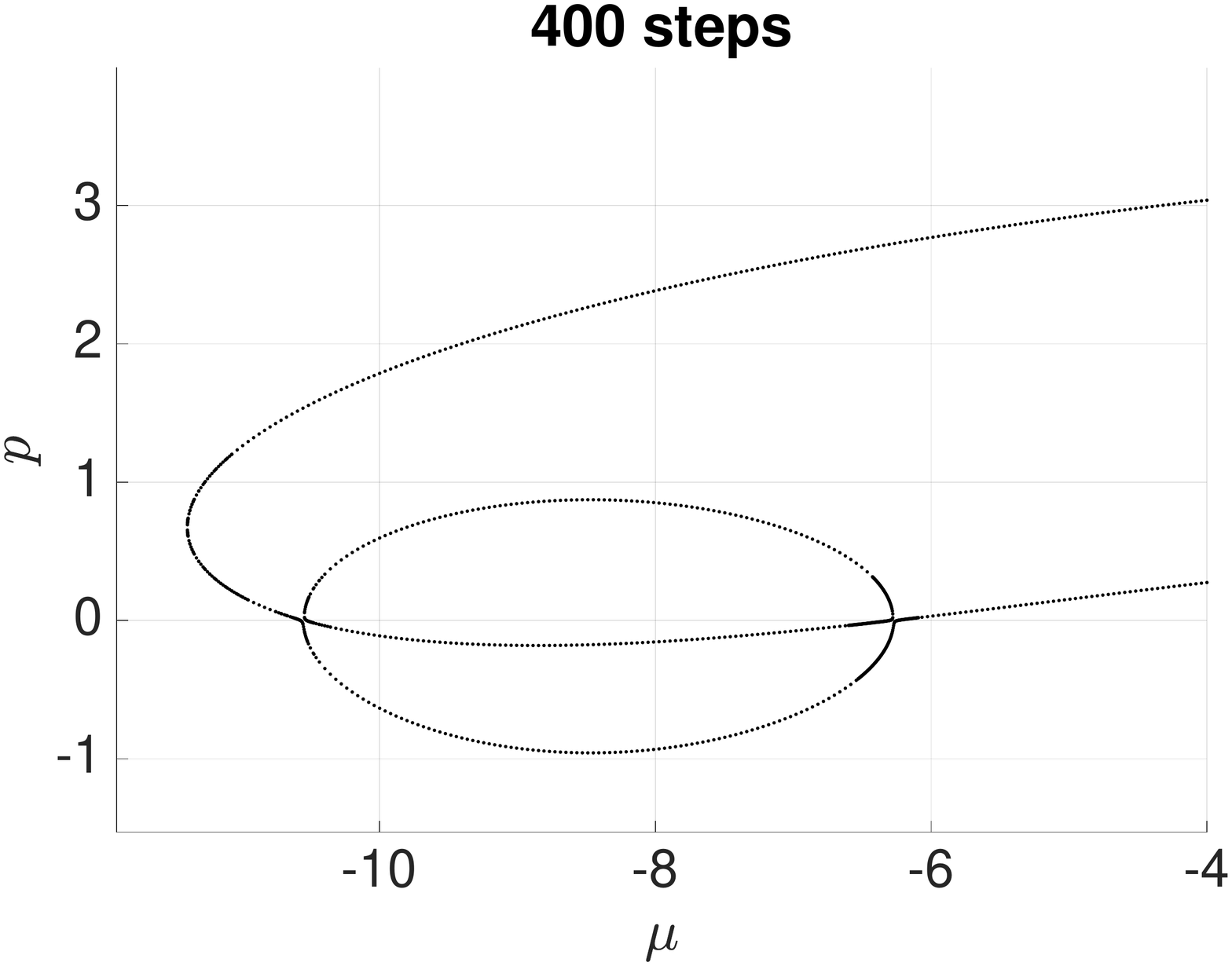}
\includegraphics[width=0.32\textwidth]{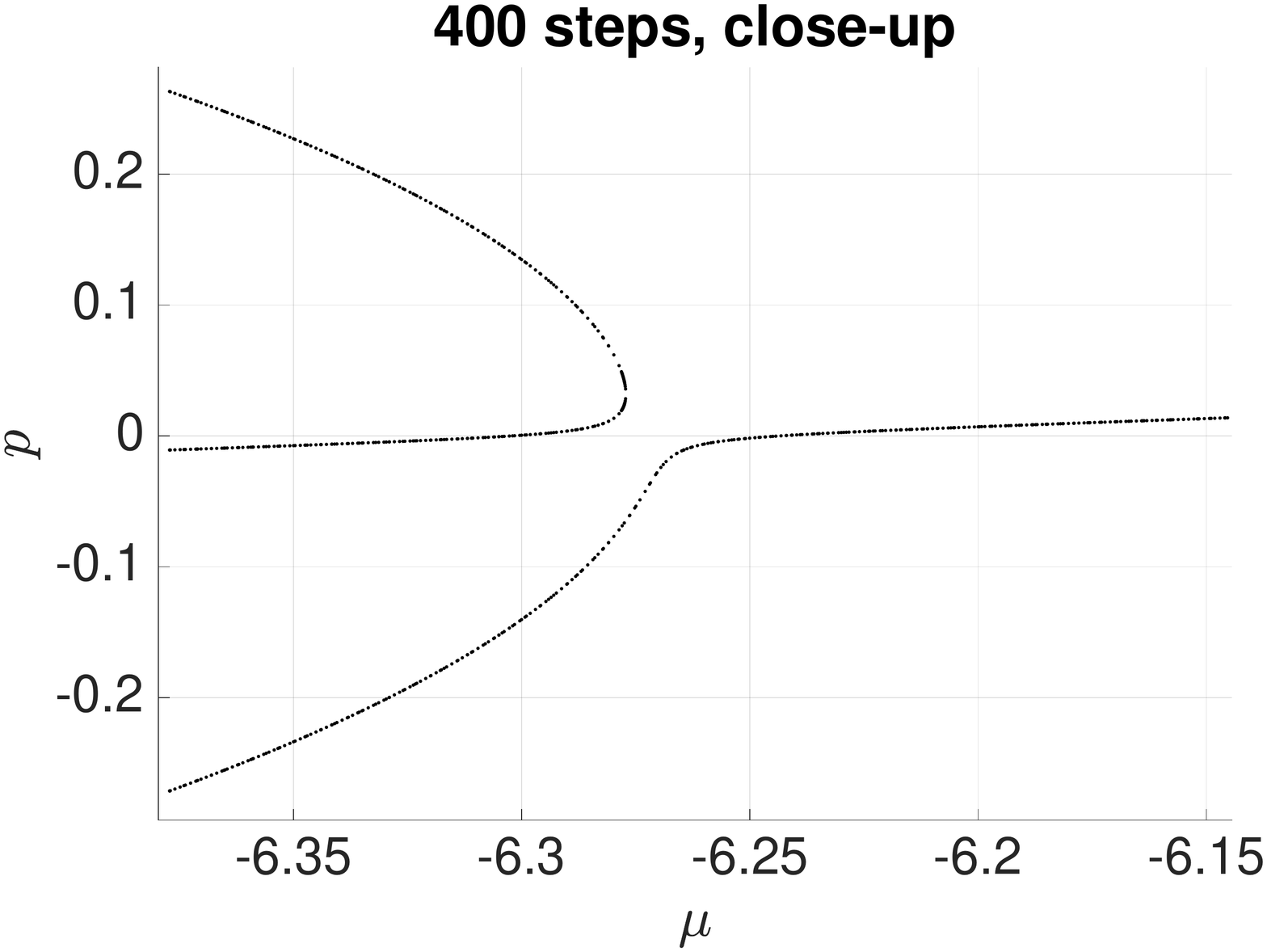}
\end{center}
\caption{Plot of bifurcation diagrams of Hamiltonian boundary value problem solved with RK2 and different number of integration steps.}\label{fig:pitchforkRK}
\end{figure}

\subsection{Construction of a periodic pitchfork in a non-trivial 4-dimensional Hamiltonian system}\label{subsec:modelnontrivialsym}
Let us construct a non-trivial, 4-dimensional completely integrable Hamiltonian system with a periodic pitchfork bifurcation that is not removable under small perturbations within the class of symmetrically separated Lagrangian boundary value problems for completely integrable Hamiltonian systems.

The circle $S^1$ can be viewed as the quotient space $\R / [0,2\pi]$. Local coordinates on $S^1$ can be obtained as the inverse of suitable restrictions of the projection map $\R \to \R / [0,2\pi] = S^1$. In this way, we obtain everywhere local coordinates $(q_1,q_2)$ for the torus $S^1 \times S^1$. Let us denote a copy of the torus $S^1 \times S^1$ by $\overline S^1 \times \overline S^1$ and its coordinates obtained as local inverses of the projection $\R \to \R / [0,2\pi]$ by $(\overline q_1, \overline q_2)$. Let $\e,\kappa$ be non-zero values in the real, open interval $(-1,1)$. The map
\begin{equation}\label{eq:changecoordsS1S1}
\begin{pmatrix}q_1,q_2\end{pmatrix} \mapsto \begin{pmatrix}
q_1+\e \cos(q_2),
q_2+\kappa \cos(q_1)
\end{pmatrix}
\end{equation}
gives rise to a diffeomorphism $h \colon S^1 \times S^1 \to \overline S^1 \times \overline S^1$.
Its cotangent lifted map $\Psi$ is given as
\begin{equation}\label{eq:cotlifth}
\Psi \colon T^\ast ( S^1 \times S^1) \to T^\ast (\overline S^1 \times \overline S^1),
\quad
\begin{pmatrix}
q,p
\end{pmatrix}
\mapsto
\begin{pmatrix}
h(q), D h(q)^{-T} p
\end{pmatrix}.
\end{equation}
Here $ D h(q)^{-T}$ denotes the inverse of the transpose of the Jacobin matrix of $h$ at the point $q$.
On the cotangent bundle of $\overline S^1 \times \overline S^1$ we consider the following family of Hamiltonians
\begin{equation}\label{eq:defoH}
\overline H_\mu (\overline q, \overline p) = \overline p_1^3 + \mu \overline p_1 + \overline p_2^2.
\end{equation}
Using the canonical symplectic form $-\d \overline \lambda$ on the cotangent bundle $T^\ast (\overline S^1 \times \overline S^1)$, we obtain a family of completely integrable Hamiltonian systems. The two independent integrals of motions are given as the coordinate functions $\overline p_1, \overline p_2$.
From Hamilton's equations we see that a motion is periodic at $\mu=0$ if ${\overline p_1(0)^2}/{\overline p_2(0)}$ is rational with initial condition $(\overline q,\overline p)=(\overline q_1(0),\overline q_2(0), \overline p_1(0),\overline p_2(0))$. In particular, motions on the Liouville torus
\[
\overline T = \left\{(\overline q_1,\overline q_2, \overline p_1,\overline p_2) = \left(\overline q_1,\overline q_2, 1,\frac 32\right) \, \Big| \, (\overline q_1,\overline q_2) \in \overline S^1 \times \overline S^1 \right\}
\]
are periodic for $\mu=0$ and their period is $\tau={2 \pi}/3$.
Now consider the family of Hamiltonian systems $(T^\ast ( S^1 \times S^1), -\d \lambda, H_\mu)$ with Hamiltonians $H_\mu = \overline H_\mu \circ \Psi$. The map $\Psi$ is symplectic such that a curve $\gamma$ is a motion in $(T^\ast ( S^1 \times S^1), -\d \lambda, H_\mu)$ if and only if $\Psi \circ \gamma$ is a motion in $(T^\ast ( \overline S^1 \times \overline  S^1), -\d \overline \lambda, \overline H_\mu)$. The preimage $T = \Psi^{-1}(\overline T)$ is a Liouville torus with orbits of the same period $\tau=2\pi/3$. The point $z = (\pi/2,0,1-3 \kappa/2,3/2)$ is an isolated point in the intersection of $T$ with the Lagrangian submanifold
\[
B = \left\{ (q_1,q_2,p_1,p_2)=\left( q_1,q_2,1-\frac 32 \kappa, \frac 32 \right) \, \Big | \, q_1,q_2 \in S^1 \right\}.
\]
Moreover, the tangent spaces of $T$ at $z$ and of $B$ at $z$ intersect in a 1-dimensional subspace which can be verified by a consideration of their images under the symplectomorphism $\Psi$. Since the parameter $\mu$ enters in a generic way, the lemma below follows by \cite[Thm. 3.2]{bifurHampaper}.

\begin{lemma}\label{lem:pitchforkpt} For $\e,\kappa \in (-1,1)\setminus\{0\}$ there exists a periodic pitchfork bifurcation at $\mu=0$ at the point $z = (\pi/2,0,1-3/2 \kappa,3/2)$ in the symmetrically separated Lagrangian boundary value problem
\[
P_\mu\left(q_1,q_2,1-\frac 32 \kappa, \frac 32\right) = \begin{pmatrix}
1-\frac 32 \kappa\\ \frac 32
\end{pmatrix}
\]
for the family of time-$2\pi/3$-maps $(Q_\mu,P_\mu)$ of the completely integrable Hamiltonian systems $(T^\ast ( S^1 \times S^1), -\d \lambda, H_\mu)$, where $-\d\lambda$ is the canonical symplectic form for cotangent bundles and $H_\mu = \overline H_\mu \circ \Psi$ is defined by \eqref{eq:changecoordsS1S1}, \eqref{eq:cotlifth}, \eqref{eq:defoH}.
\end{lemma}

Figure \ref{fig:captureintegrablepitchfork} shows the bifurcation diagrams of the numerical flow for the boundary value problem described in lemma \ref{lem:pitchforkpt} for $\e = \kappa = 0.1$. Hamilton's equations are solved over the time interval $[0, 2\pi/3]$ using the symplectic 2nd order St\"ormer-Verlet method with 20, 40 and 80 time-steps. The results indicate that the periodic pitchfork bifurcation of the exact flow is only captured up to the accuracy of the integrator.

\begin{figure}
\begin{center}
\includegraphics[width=0.325\textwidth]{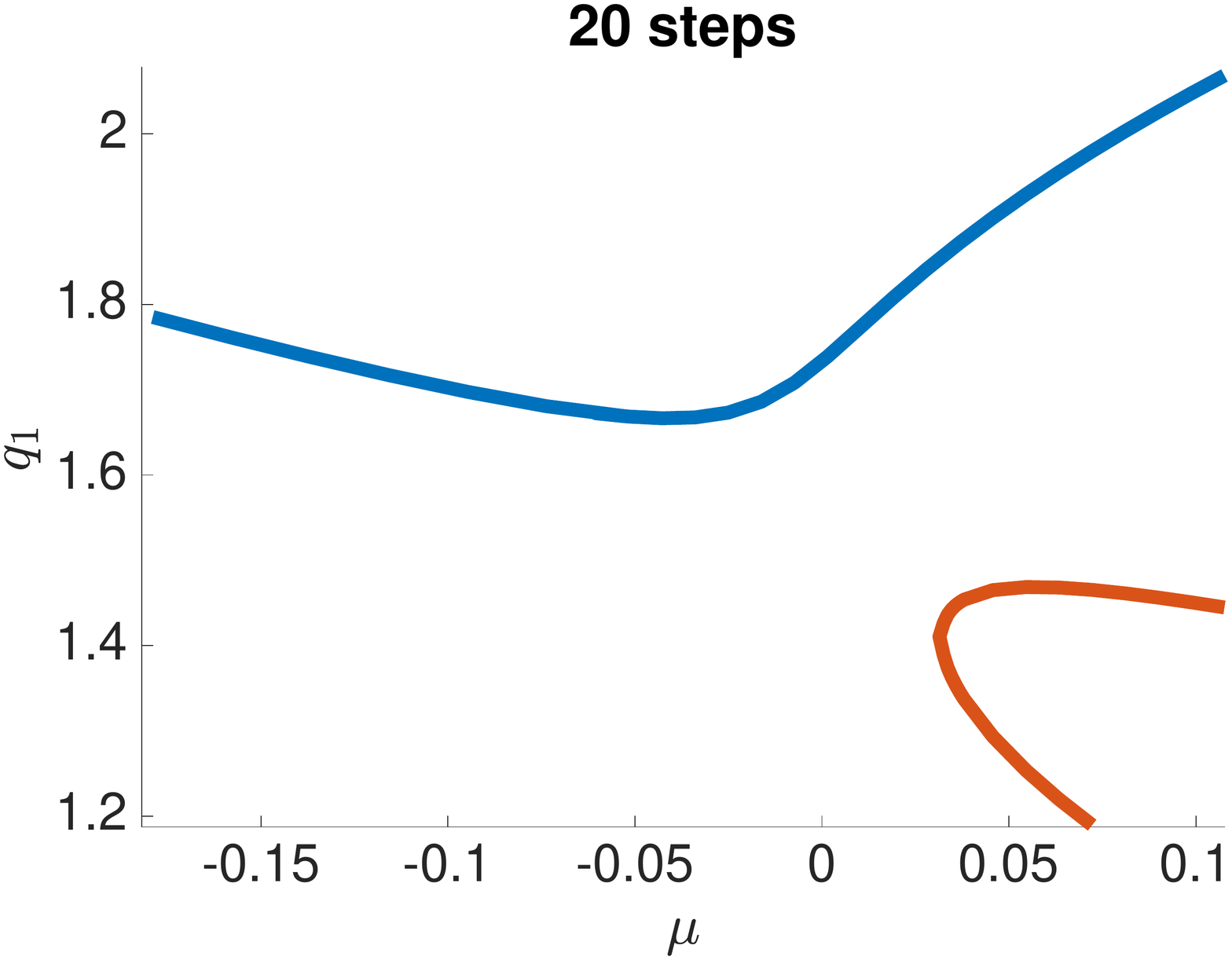}
\includegraphics[width=0.325\textwidth]{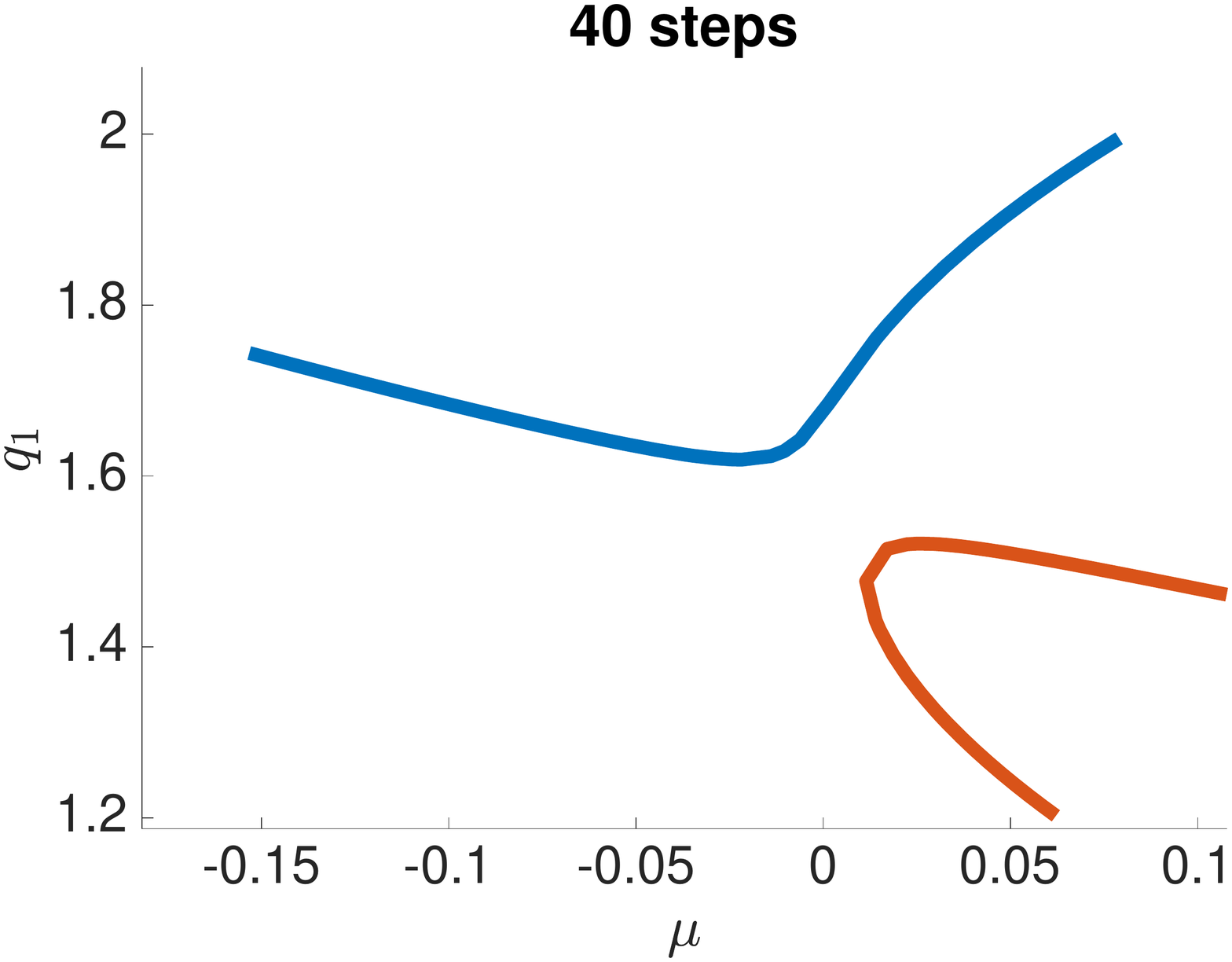}
\includegraphics[width=0.325\textwidth]{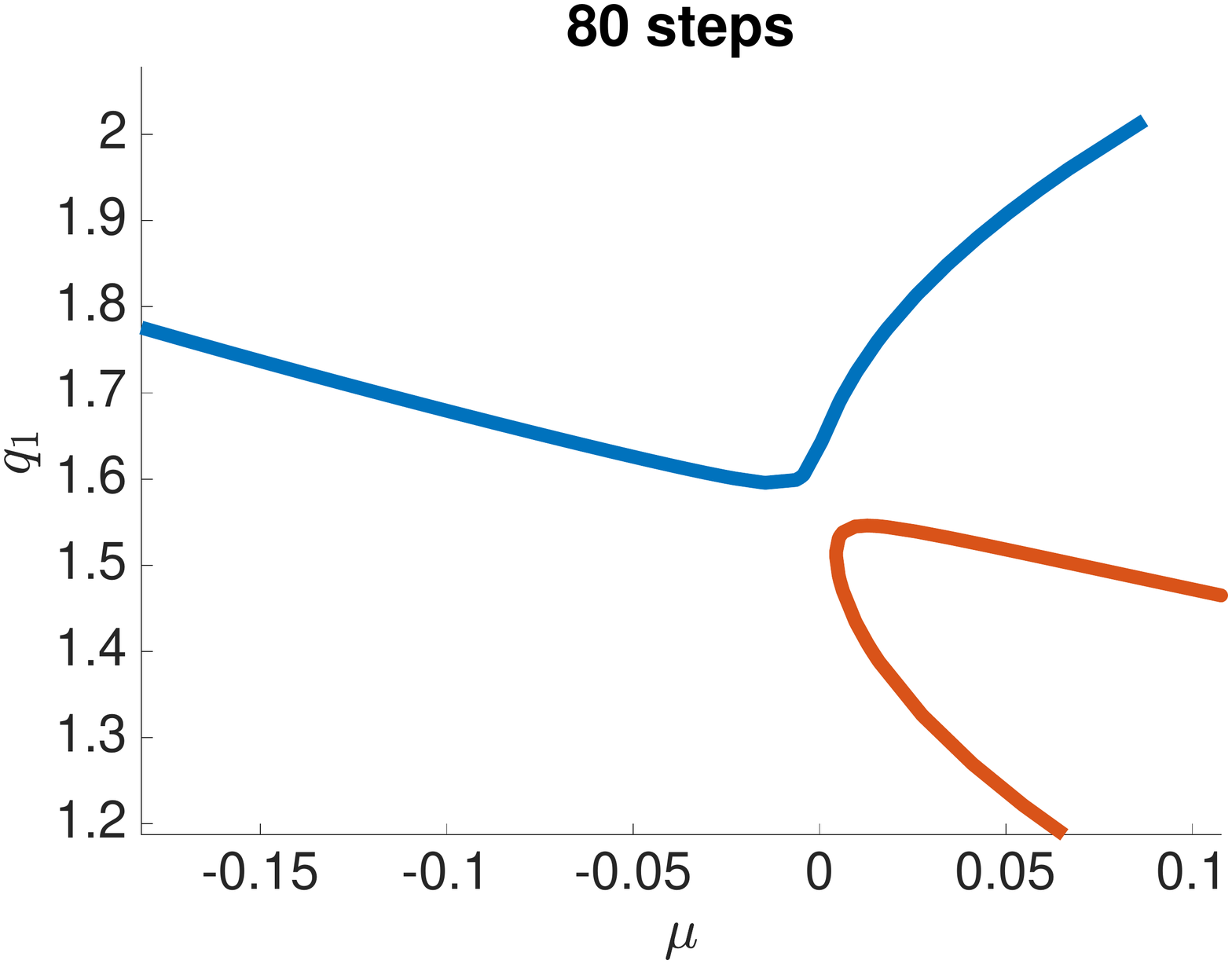}
\end{center}
\caption{Plot of bifurcation diagrams for boundary value problem of lemma \ref{lem:pitchforkpt} in numerical flow projected along the $q_2$-axis obtained using the symplectic 2nd order St\"ormer-Verlet method. The shape of the periodic pitchfork bifurcation is only captured as expected from the accuracy of the method.}\label{fig:captureintegrablepitchfork}
\end{figure}

\subsection{Affine linear symmetries and the St\"ormer-Verlet method}\label{subsec:SymmetricCapture}

The St\"ormer-Verlet method preserves linear invariants \cite[Thm. IV 1.5]{GeomIntegration} and quadratic invariants of the form $Q(q,p)=q^t A p$ for a fixed matrix $A$ \cite[Thm. IV 2.3]{GeomIntegration}. Let us see how a periodic pitchfork bifurcation is captured in two numerical examples of completely integrable Hamiltonian systems with simple symmetries / invariants.

\begin{example}[Cyclic variable]\label{ex:cylicvariable}
If a variable does not occur in the expression of a Hamiltonian then its conjugate momentum is a conserved quantity. The conserved quantity can be treated as a parameter for the system such that Hamilton's equations can be solved on a space whose dimension is reduced by two where the cyclic variable and its conjugate momentum do not appear as dynamical variables. The evolution in the cyclic variable can then be integrated separately. If the phase space dimension is $2n$ and the Hamiltonian has $n-1$ cyclic variables then we can obtain the same behaviour as for planar Hamiltonian systems by applying the symplectic integrator to the reduced planar system. However, even if we apply the St\"ormer-Verlet method to the non-reduced system, the conjugate momenta are preserved as these are linear integrals of motions. The integrals correspond to translation symmetries in the cyclic variables.

We consider the family of Hamiltonian systems defined by
\begin{equation}\label{eq:HamPP}
H_\mu(q_1,q_2,p_1,p_2)
= q_1^3+\mu q_1 + p_1 p_2 + p_1^2 + \frac 1 {10}(p_1^3+p_2^3)
\end{equation}
on the phase space $\R^4$ with the standard symplectic structure. The variable $q_2$ is cyclic. Let us consider the symmetric Dirichlet boundary value problem
\begin{equation}\label{eq:bvpforPP}
q(0)=\begin{pmatrix}0.2\\ 0.1\end{pmatrix} = q(5).
\end{equation}

Using the symplectic, 2nd order St\"ormer-Verlet method to calculate the numerical flow, we find a pitchfork bifurcation. Introducing the cyclic variable into the Hamiltonian by adding the term $0.01q_2$ breaks the pitchfork bifurcation. This confirms that the pitchfork bifurcation is due to the completely integrable structure (see figure \ref{fig:PPcyclic}).

\begin{figure}
\begin{center}
\includegraphics[width=0.49\textwidth]{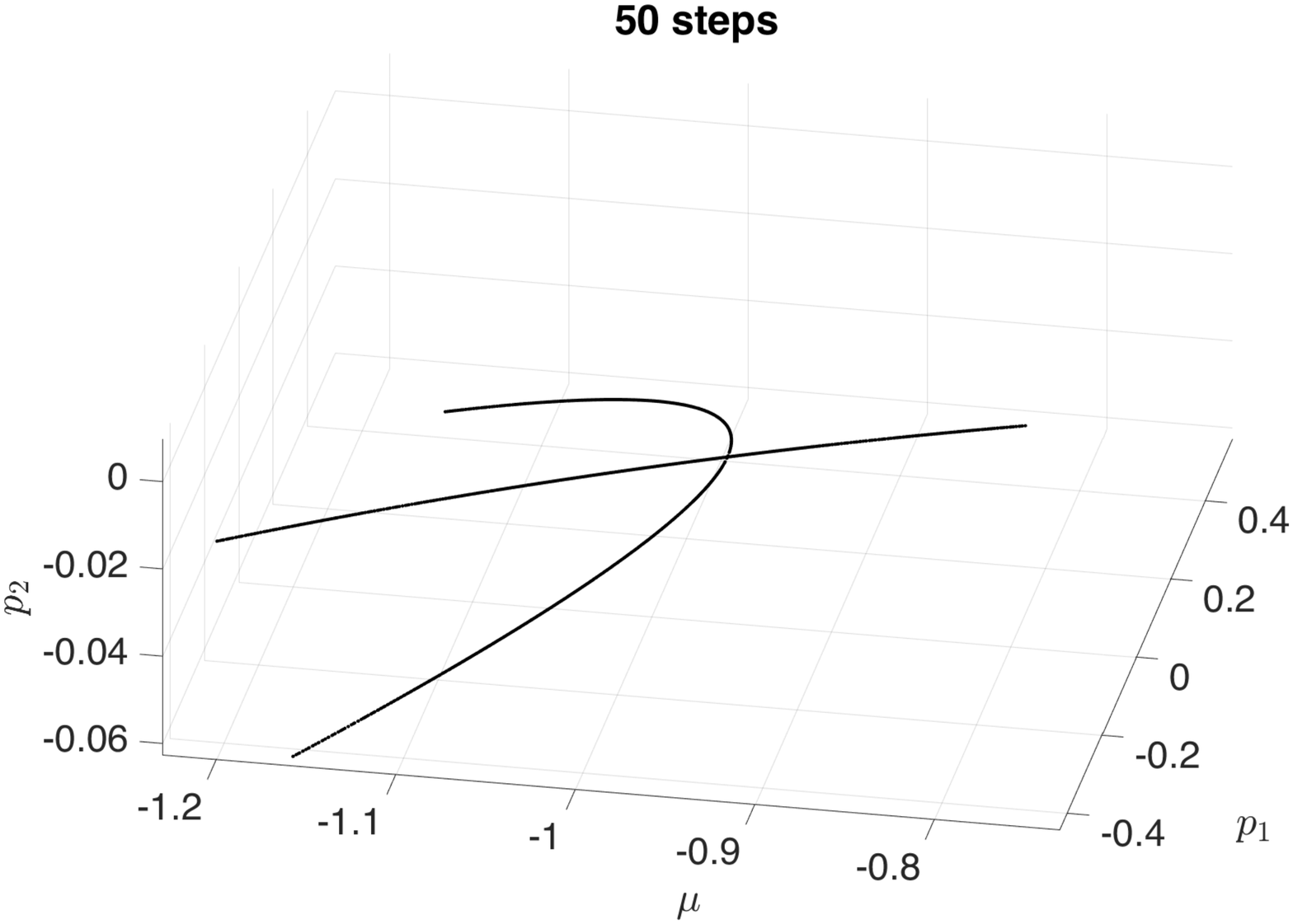}
\includegraphics[width=0.49\textwidth]{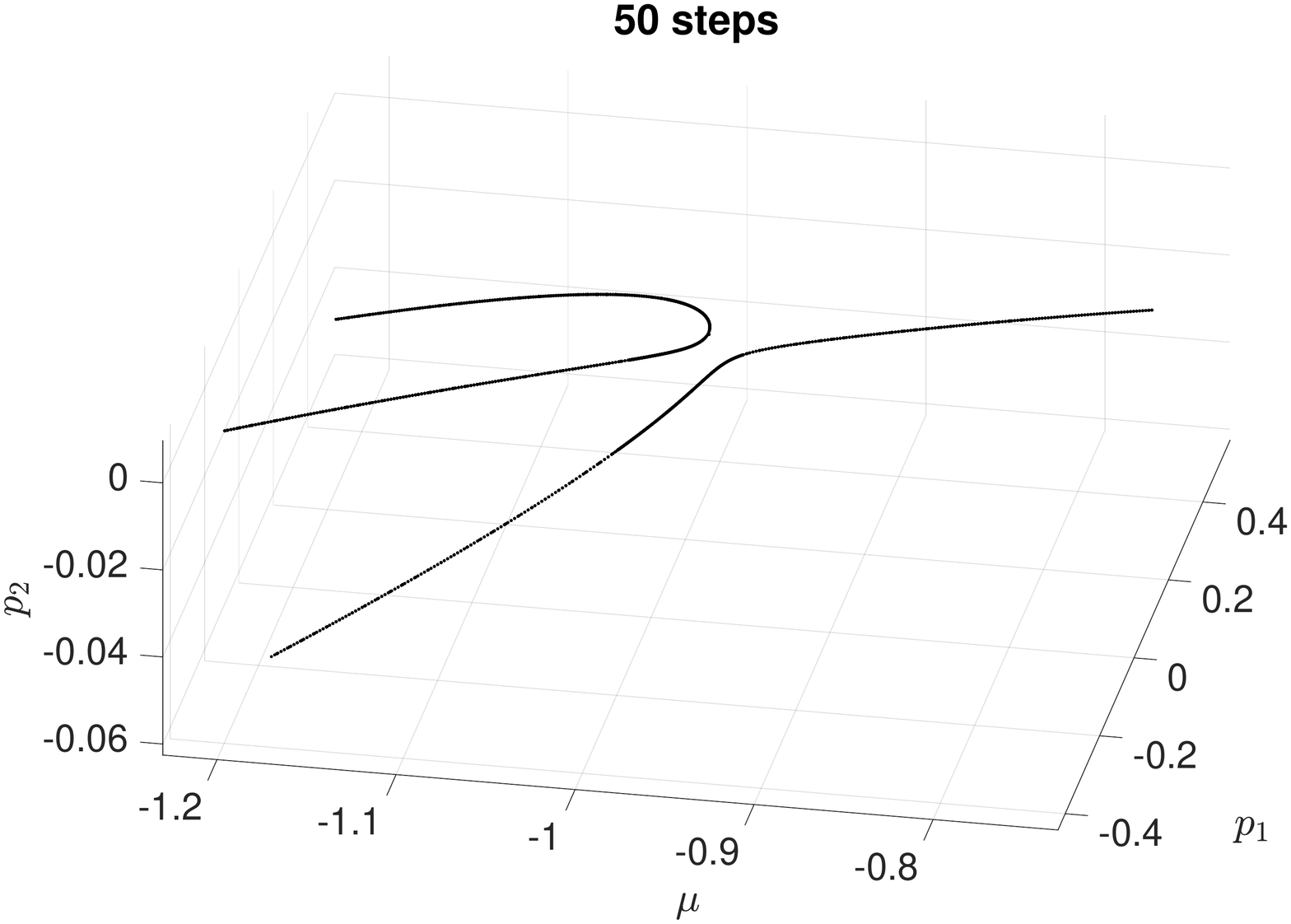}
\end{center}
\caption{The plot to the left shows a periodic pitchfork bifurcation in the 4 dimensional Hamiltonian system \eqref{eq:HamPP} with a cyclic variable for the boundary value problem \eqref{eq:bvpforPP}. The plot to the right shows how the pitchfork breaks if the cyclic variable is introduced in the Hamiltonian. This confirms that the appearance of the pitchfork is due to the complete integrable structure. For both plots the St\"ormer-Verlet method with 50 time-steps was used.}\label{fig:PPcyclic}
\end{figure}

\begin{figure}
\begin{center}
\includegraphics[width=0.325\textwidth]{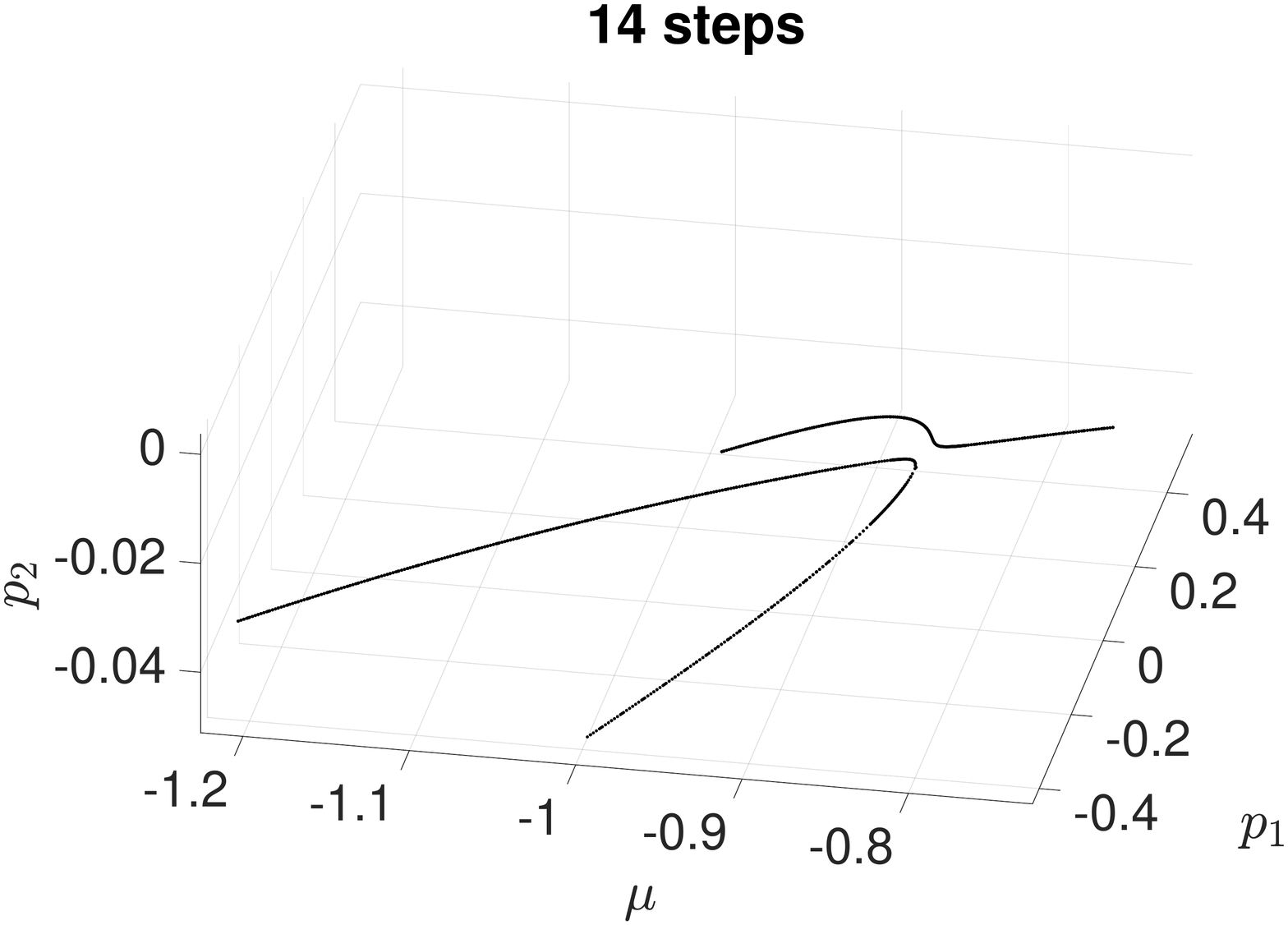}
\includegraphics[width=0.325\textwidth]{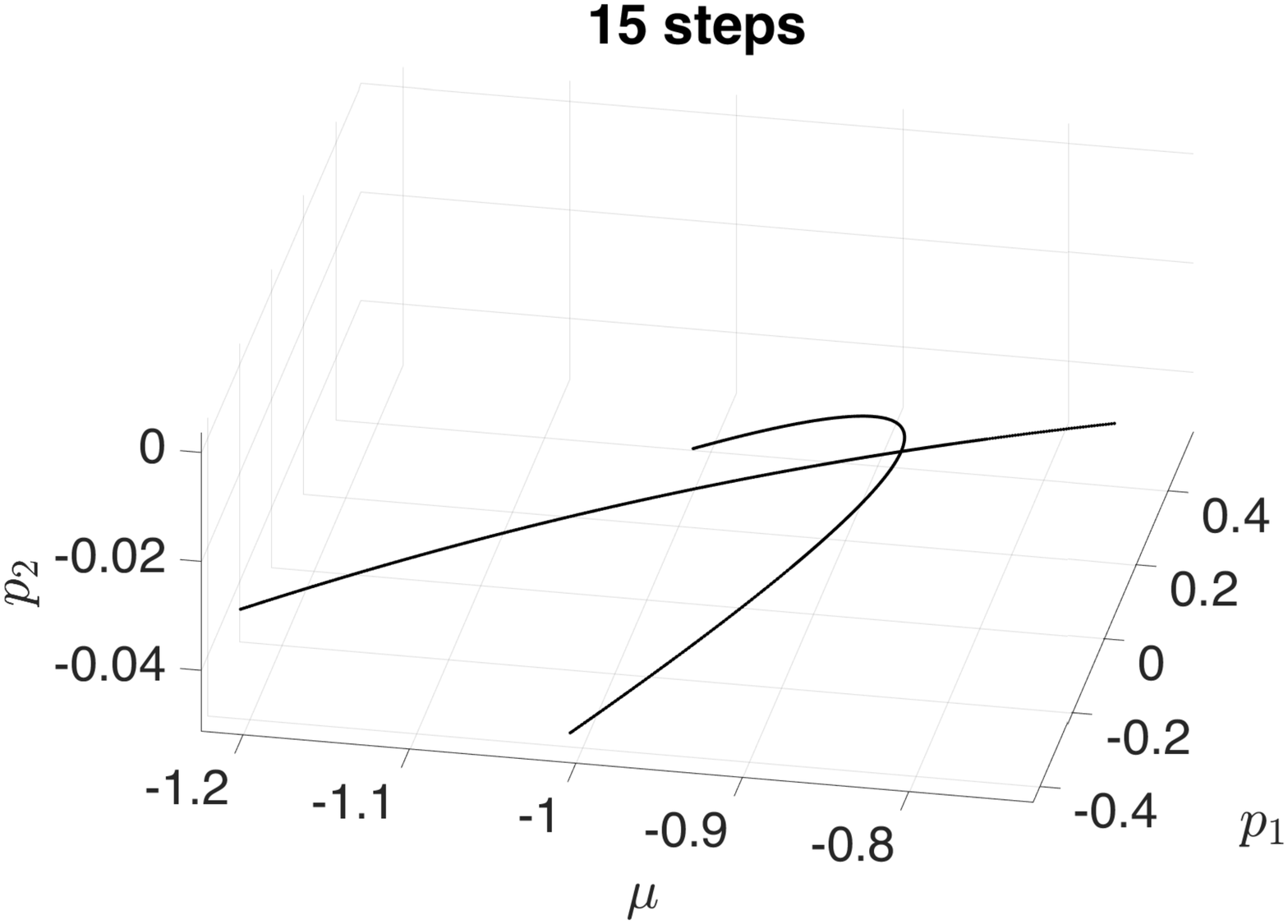}
\includegraphics[width=0.325\textwidth]{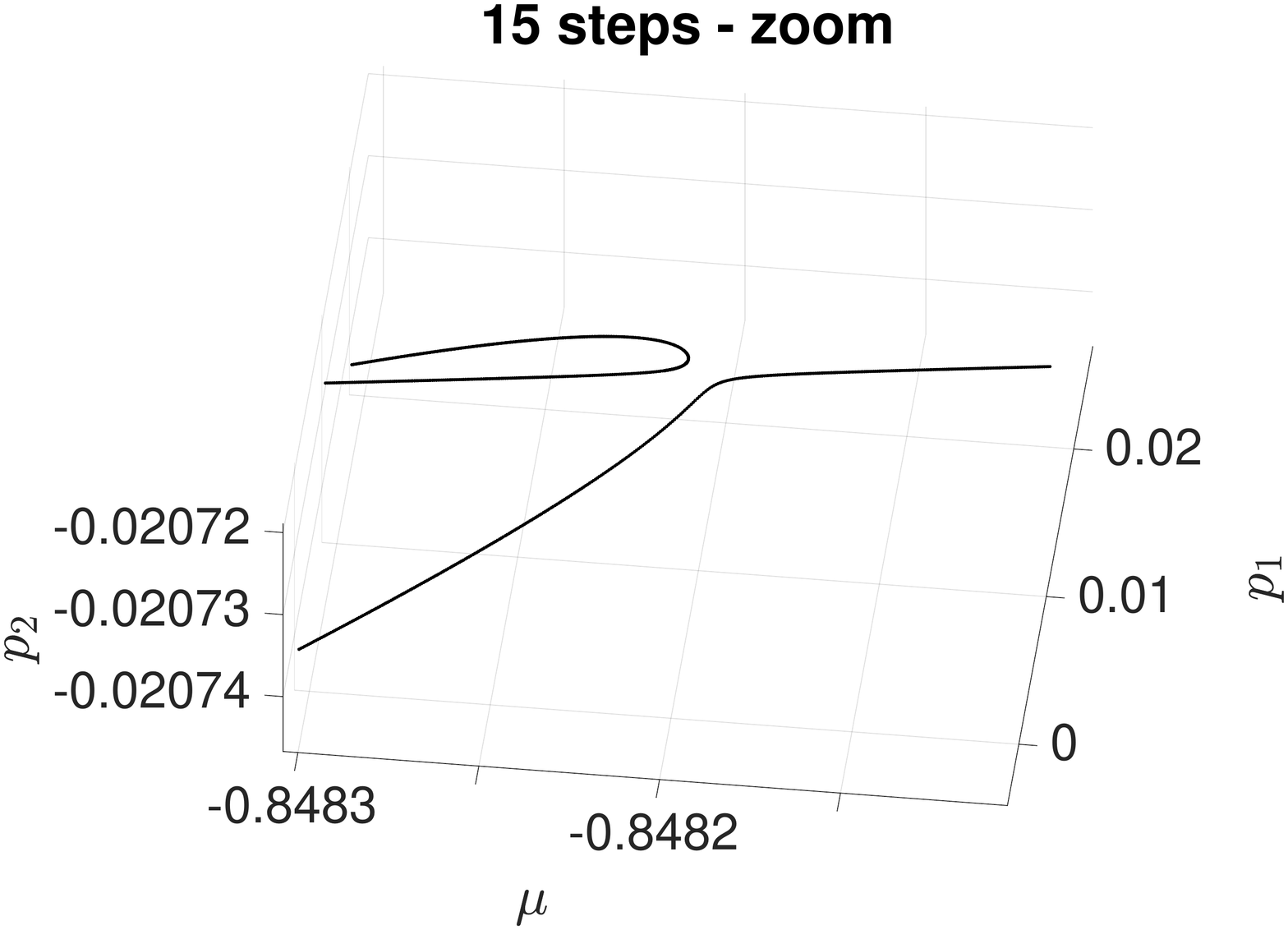}
\end{center}
\caption{The plots show how the periodic pitchfork bifurcation shown in figure \ref{fig:PPcyclic} breaks as we reduce the number of time-steps in the St\"ormer-Verlet scheme. While the break is clearly visible when 14 steps are used, it can only be spotted in a close-up when 15 steps are used.}\label{fig:PPcyclic1415}
\end{figure}

We reduce the number of steps in our calculation to analyse how the pitchfork bifurcation breaks. Figure \ref{fig:PPcyclic1415} shows that there is a clearly visible break when 14 steps are used. If 15 steps are used, however, the break can only be spotted in a close-up.
The $p_1$-axis is scaled approximately by a factor 10 and the $p_2$-axis and $\mu$-axis by 1000. Since the number of steps is only increased by 1, corresponding to an increase by approximately 7\%, this indicates a capturing to higher than polynomial order as we will justify in proposition \ref{prop:cappitch}. 
\end{example}

\begin{example}[Linear conservation law / linear symmetry]
Let us apply a linear, symplectic change of coordinates to the Hamiltonian boundary value problem (\ref{eq:HamPP},~\ref{eq:bvpforPP}) in the cyclic-variable example and test the behaviour of the St\"ormer-Verlet method.

If $A\in \mathrm{Gl}(n,\R)$ is a linear transformation of $\R^n$ then
\begin{equation}\label{eq:PhiDef}
\begin{pmatrix}\tilde q\\ \tilde p \end{pmatrix} 
= \Psi(q,p)
=\begin{pmatrix}A q\\ A^{-T} p \end{pmatrix}
\end{equation}
is a symplectic transformation on $T^\ast\R^{n}$.
We consider
\[
A = \begin{pmatrix} -1&2\\3&1\end{pmatrix}.
\]
and apply the transformation defined by $\Psi$ to the Hamiltonian boundary value problem (\ref{eq:HamPP},~\ref{eq:bvpforPP}) considered in example \ref{ex:cylicvariable}. In the transformed system $H\circ \Psi^{-1}$ and $q_2 \circ \Psi^{-1}$ are integrals of motions.
Figure \ref{fig:PPcyclic1415newcoords} shows the analogous situation to figure \ref{fig:PPcyclic1415} in the new coordinates. We see that a linear change of coordinates does not have any effect on how well the bifurcation is captured. This is to be contrasted to the integrable system presented in section \ref{subsec:modelnontrivialsym} whose quantities are not affine-linear in the variables used to integrate the Hamiltonian flow and the bifurcation is captured only up to the accuracy of the integrator (figure \ref{fig:captureintegrablepitchfork}).

\begin{figure}
\begin{center}
\includegraphics[width=0.325\textwidth]{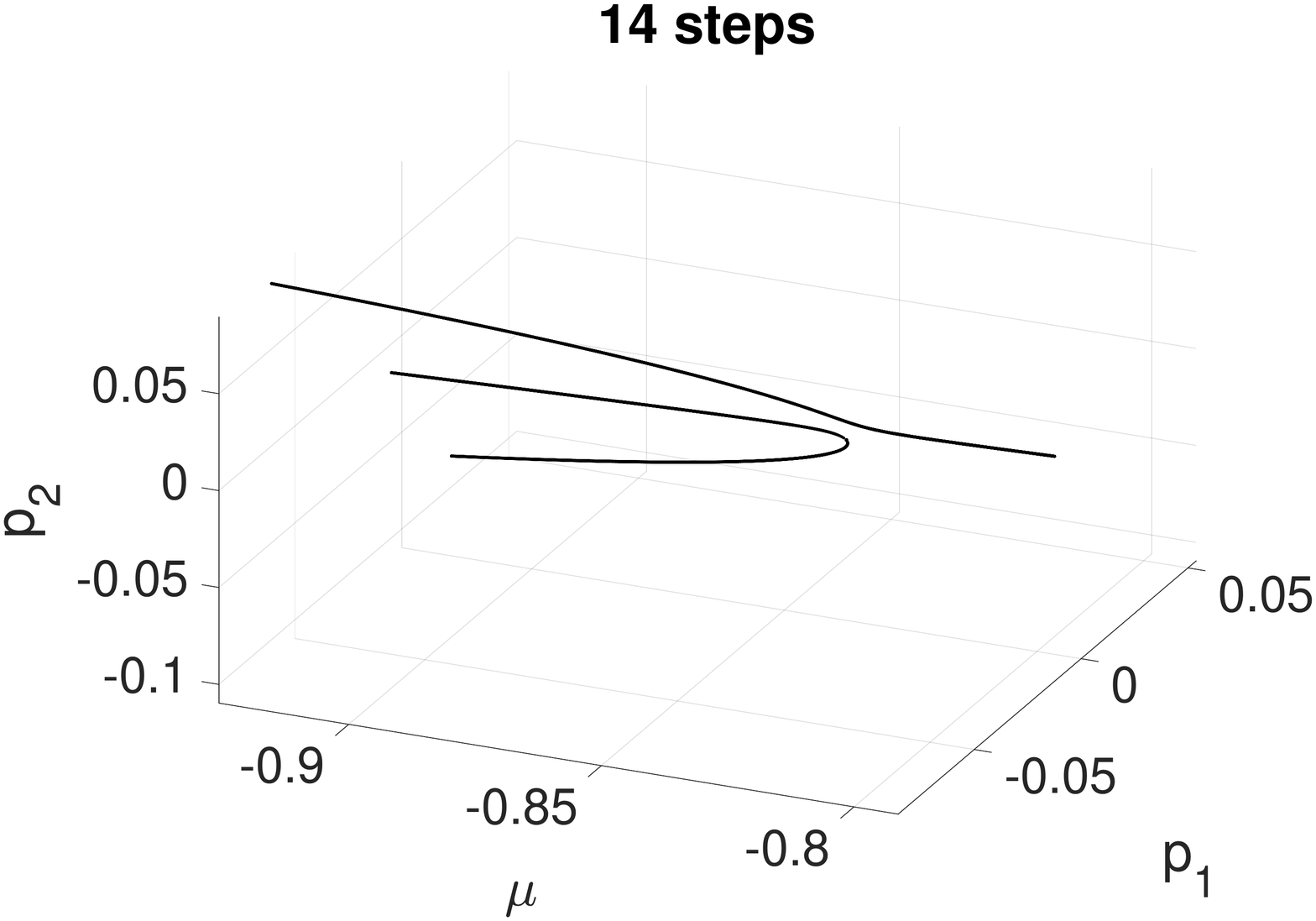}
\includegraphics[width=0.325\textwidth]{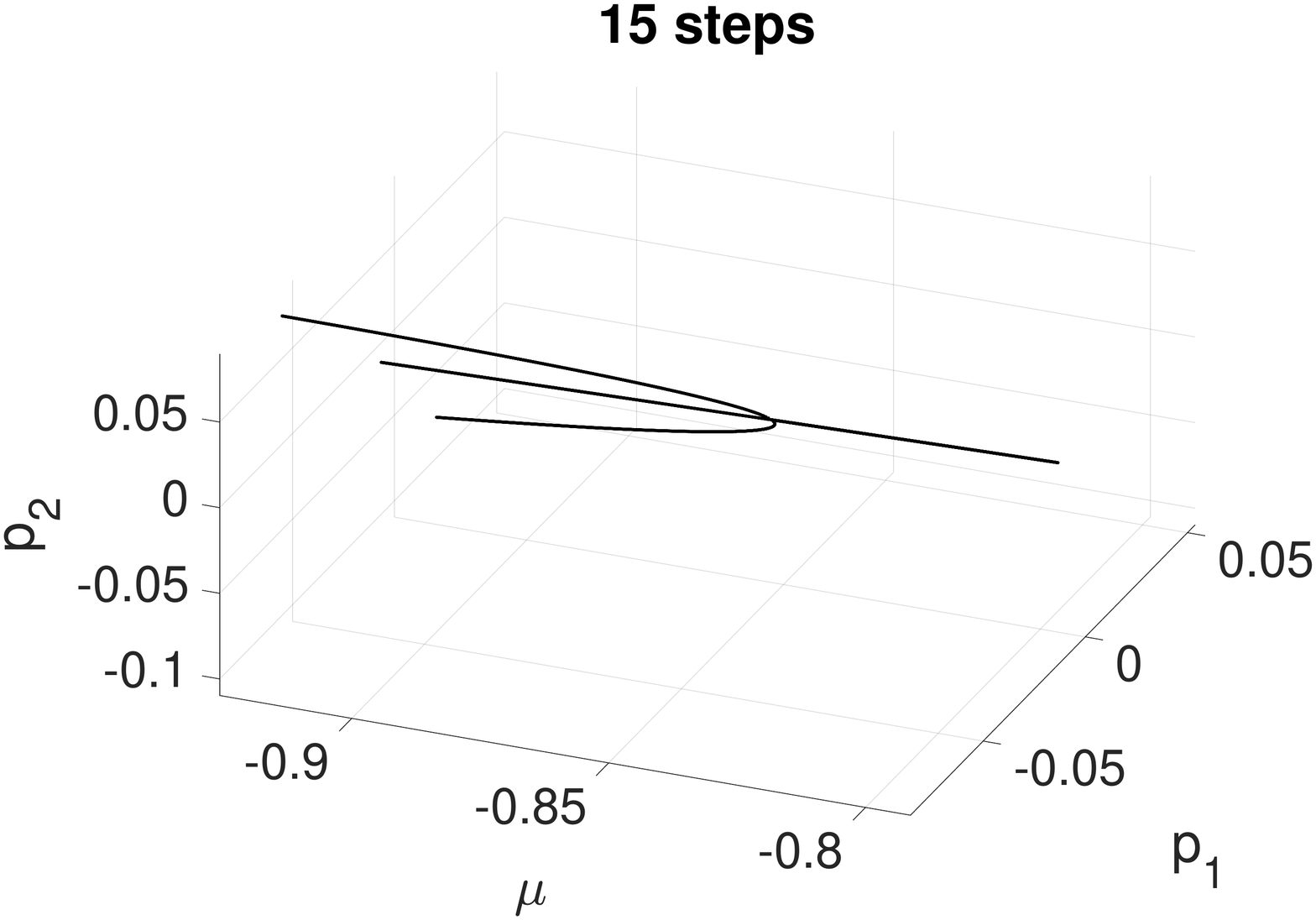}
\includegraphics[width=0.325\textwidth]{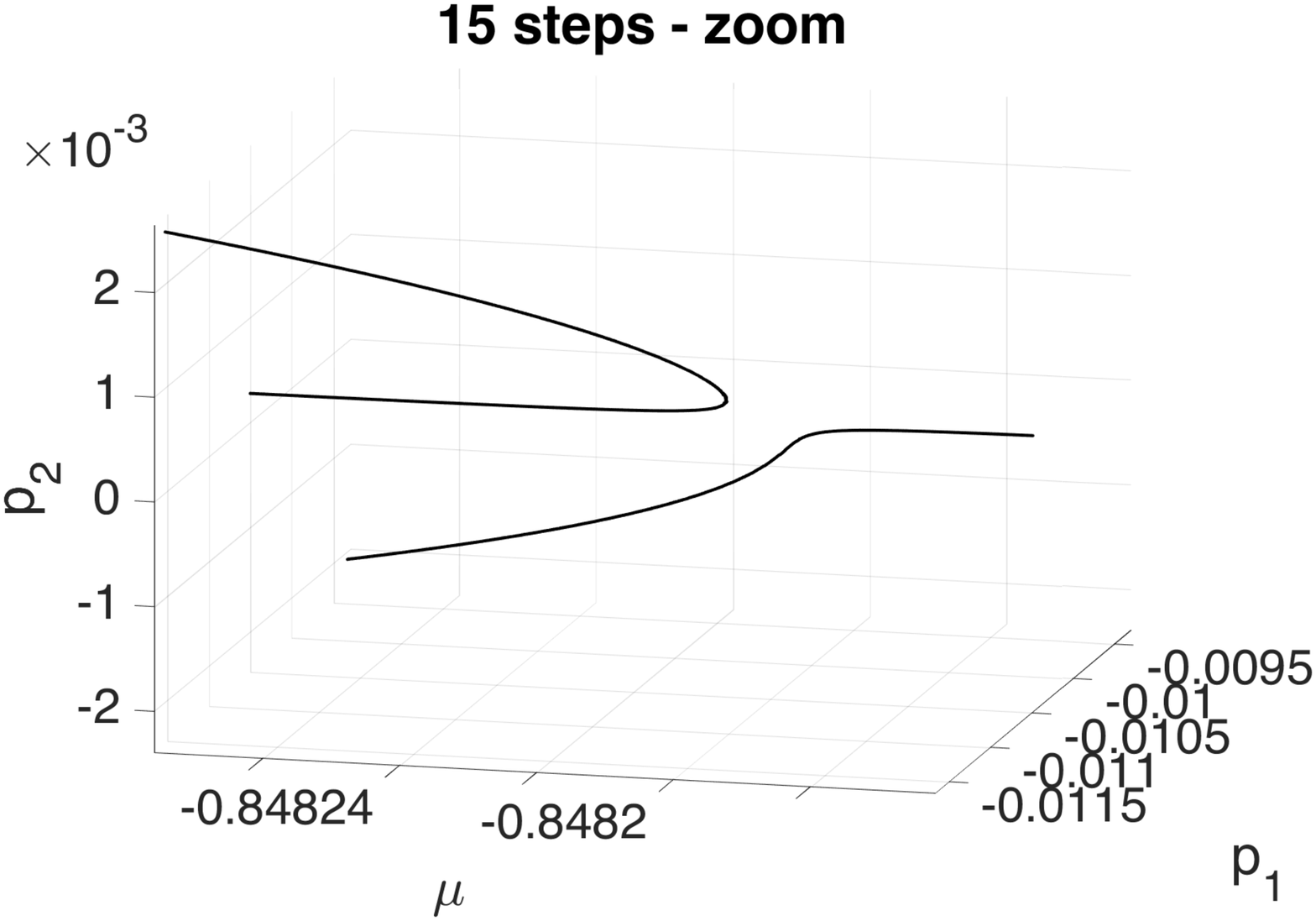}
\end{center}
\caption{The plots show how a periodic pitchfork bifurcation is captured if one of the integrals is linear. While the break is clearly visible when 14 steps are used, it can only be spotted in a close-up when 15 steps are used. This is in analogy to the cyclic-variable case (figure \ref{fig:PPcyclic1415}) but different to the case of complicated integrals (figure \ref{fig:captureintegrablepitchfork}).}\label{fig:PPcyclic1415newcoords}
\end{figure}
\end{example}

%% file: pitchfork_capture_discretisation.tex
\subsection{Theoretical consideration of the effects of symplectic structure preserving discretisation}\label{subsec:TheoCapture}

To which extent the completely integrable structure of a system is present in the numerical flow determines how well a pitchfork bifurcation is captured. This is made precise in the following

\begin{prop}[capturing of periodic pitchfork bifurcation]\label{prop:cappitch}
Consider a smooth 1-parameter family of Hamiltonian boundary value problems for $2n$-dimensional completely integrable Hamiltonian systems with symmetrically separated Lagrangian boundary conditions and a generic periodic pitchfork bifurcation.
Consider a discretisation of the Hamiltonian flows by a symplectic integrator with order of accuracy $k$ and constant step-size $h$. In a generic setting for sufficiently small $h$ the family of numerical flows has a bifurcation that is close to a pitchfork bifurcation 
\begin{itemize}
\item
to exponential order in $h^{-1}$ if all $n$ integrals are preserved exponentially well (e.g.\ in the planar case)
\item
to order $k$ otherwise.
\end{itemize}
\end{prop}

\begin{remark}
As shown in \cite[Sec.\ 3.3.2.]{bifurHampaper} and recalled in figure \ref{fig:periodicpitchforkmechanism}, 
a periodic pitchfork bifurcation takes place at an intersection point $z$ of the boundary condition $\Lambda$ with a Liouville torus $\mathcal T$ invariant under the Hamiltonian flow, such that the intersection of the tangent spaces $T_z\mathcal T$ and $T_z \Lambda$ is 1-dimensional.
The torus $\mathcal T$ is required to consist of periodic orbits of a given period $\tau$.
It is, therefore, highly resonant and immediately destroyed when the Hamiltonian flow is perturbed, even when the perturbation is symplectic.
Thus, results obtained by KAM-theory about the exponentially long persistence of invariant Liouville tori under symplectic discretisation (see \cite[X.5.2]{GeomIntegration}) do not apply in this setting as non-resonance conditions are not fulfilled. This is why symplectic integrators can break the structure significant for pitchfork bifurcations in a general setting as we saw in figure \ref{fig:captureintegrablepitchfork}.
\end{remark}

\begin{proof}[Proof of proposition \ref{prop:cappitch}]
In \cite{bifurHampaper} the authors reveal how the completely integrable structure and the structure of the boundary conditions induce a $\Z/2\Z$-symmetry in the generating function of the problem family.
The singular point of a pitchfork bifurcation is unfolded under the presence of a $\Z/2\Z$-symmetry to a pitchfork bifurcation. The corresponding critical-points-of-a-function problem is defined by the family $(x^4+\mu_2 x^2)_{\mu_2}$, i.e.\ the generating function of the problem family is stably right-left equivalent to the family $(x^4+\mu_2 x^2)_{\mu_2}$.
Unfolding of $x^4$ without the $\Z/2\Z$-symmetry leads, however, to the normal form of a cusp bifurcation which is defined by the family  $(x^4+\mu_2 x^2+\mu_1x)_{\mu_1,\mu_2}$. The effect of the symmetry breaking parameter $\mu_1$ is illustrated in figure \ref{fig:pitchforkcusp}: the bifurcation, which is present for $\mu_1=0$, breaks if $\mu_1 \not=0$.

Approximating the Hamiltonian flow with an integrator introduces the discretisation parameter $h$ as an additional parameter to the problem family. The discretisation does not respect the completely integrable structure which corresponds to a $\Z / 2\Z$-structure of the generating function.
\begin{figure}
\begin{center}
\includegraphics[width=0.6\textwidth]{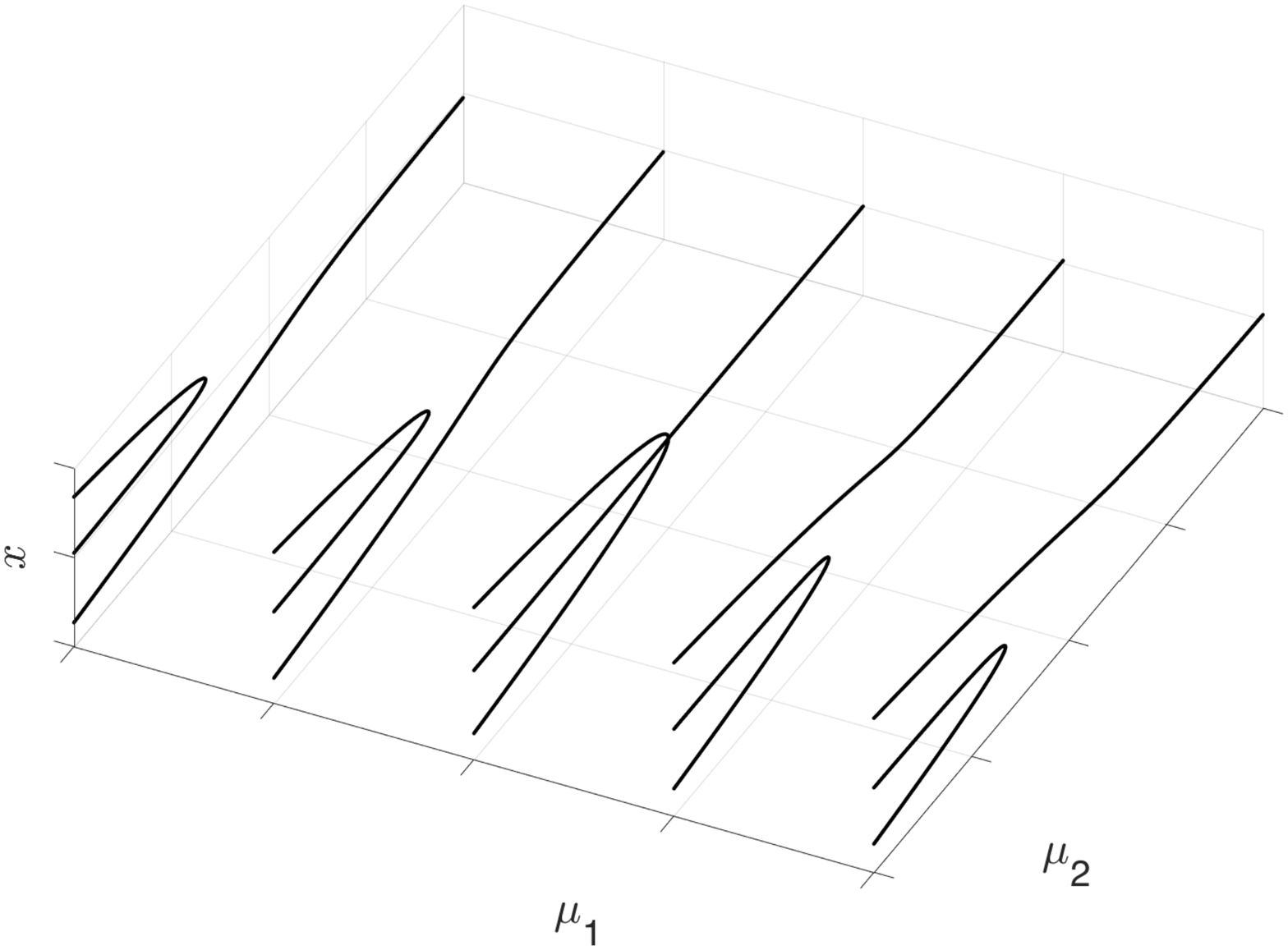}
\end{center}
\caption{The figure shows the critical point set of the model cusp $x^4+\mu_2 x^2+ \mu_1 x$ over the $\mu_1$/$\mu_2$-parameter space for selected values of $\mu_1$. For $\mu_1=0$ we see a pitchfork bifurcation.}\label{fig:pitchforkcusp}
\end{figure}
If the order of accuracy of the integrator is $k$ then, generically, the power of the step-size $h^k$ acts like the parameter $\mu_1$ in figure \ref{fig:pitchforkcusp}. We say the pitchfork is \textit{broken up to the order of accuracy of the integrator}.
This means in a generic setting symplecticity of an integrator cannot be expected to improve the numerical capturing of the periodic pitchfork bifurcation because the bifurcation is due to the integrable structure rather than to symplecticity. 
However, in many important cases, symplecticity does help because symplectic integrators preserve a modified Hamiltonian exponentially well \cite[IX]{GeomIntegration} and are, therefore, guaranteed to capture at least this part of the integrable structure very well. In the planar case, e.g.\, this means the whole integrable structure is captured exponentially well by symplectic integrators. Here, the discretisation parameter does not enter generically but unfolds the pitchfork bifurcation to a family of nearly perfect pitchforks. These pitchforks are broken only up to exponential order in $-h^{-1}$. The same is true in higher dimensions if the $n-1$ additional integrals/symmetries are captured at least exponentially well.
\end{proof}

%% file: summary_outlook.tex
\section{Summary and outlook}

The paper explains the role of symplecticity when calculating bifurcation diagrams for Hamiltonian boundary value problems. In particular, we show using algebraic considerations as well as numerical examples that hyperbolic and elliptic umbilic bifurcations are preserved if and only if the integration scheme is symplectic. 
Moreover, we show how to exploit the special structure of typical boundary conditions for computations and illustrate our findings in numerical examples including the calculation of conjugate loci.
Even when the stability of a bifurcation is not related to the symplectic structure itself, we show that symplectic integrators can be superior when calculating bifurcation diagrams: we present the periodic pitchfork bifurcation which is related to Liouville integrability. We explain the mechanism of the bifurcation in the discretised case and show that symplectic integrators perform significantly better at preserving the bifurcation in many cases. 
The findings open up new research directions. In \cite{PDEbifur} the authors interpret symplcticity as the existence of variational structure and use that viewpoint to extended the ideas to variational PDEs. 
Numerical aspects are studied further, by developing algorithms to locate high codimensional bifurcations automatically.
It would also be nice to extend existing highly developed bifurcation continuation software packages like AUTO \cite{AutoManual} or pde2path \cite{pde2path} to the high codimensional bifurcation analysis considered in this paper and to compute higher-dimensional manifolds of solutions using manifold continuation \cite{Henderson2002}. The symplectic structure of the steady states of initial-value problems like \eqref{BratuPDE}, that we have focused on in this paper, may also be relevant to their stability.

%% file: RATTLE_jet_method.tex
\section{Jet-RATTLE for calculation of geodesics on hypersurfaces}\label{app:geodesicsRATTLEformulas}

Let $(N,g)$ be a hypersurface of $\R^n$ defined by the equation $f(q)=0$ for $f\colon \R^n \to \R$ such that $\nabla f(q) \not=0$ for all $q \in M$. Here $g$ refers to the induced Riemannian metric on the hypersurface $N$. In order to compute geodesics on $(N,g)$ with respect to the Levi-Civita connection, we apply the RATTLE method to the Hamiltonian
\[
H(q,p)=\frac 12 \langle p,p\rangle
\]
defined on the cotangent bundle over $T^\ast \R^n$ with the standard symplectic structure for cotangent bundles and Darboux coordinates $q^1,\ldots,q^n,p_1,\ldots,p_n$. In the above formula $\langle .,.\rangle$ denotes the euclidean scalar product. For a fixed time-step $h>0$ the RATTLE method gives rise to a map on $T^\ast N$ which is symplectic with respect to the standard symplectic structure on cotangent bundles (assuming convergence of the implicit scheme) \cite{LEIMKUHLER1994117}.

The formulas for the time-$h$-map $\Psi_h$ calculating the two $n$-dimensional vectors $(q_{n+1},p_{n+1})$ from the initial values $(q_{n},p_{n})$ read:

\begin{align}\label{eq:implicitlambda}
0 &= f\left(q_n+h\left(p_n-\frac h2 \nabla f(q_n) \cdot \lambda\right)\right)\\ 
\label{eq:formulap12}p_{n+\frac 12} &= p_n-\frac h2 \nabla f (q_n)\cdot \lambda\\ 
\label{eq:formulaq1}q_{n+1} &= q_n+h p_{n+\frac 12}\\ 
\label{eq:formularn} n &= \frac{\nabla f(q_{n+1})}{\|\nabla f(q_{n+1}) \|}\\ 
\label{eq:formularpn1} p_{n+1} &= p_{n+\frac 12} - \left\langle n,p_{n+\frac 12}\right\rangle n
\end{align}

After the 1-dimensional equation \eqref{eq:implicitlambda} is solved for $\lambda \in \R$ the remaining equations can be evaluated explicitly.

\begin{remark}
The formulas \eqref{eq:formularn} and \eqref{eq:formularpn1} describe a projection of $p_{n+\frac 12}$ to the tangent space at $q_{n+1}$. The effect is wiped out by (\ref{eq:implicitlambda}, \ref{eq:formulap12}, \ref{eq:formulaq1}) of the following step, i.e.\ the value for $q_{n+2}$ does not depend on whether we set $p_{n+1}$ according to \eqref{eq:formularpn1} or simply $p_{n+1} = p_{n+\frac 12}$. If the formulas are iterated then the projection step (\ref{eq:formularn}, \ref{eq:formularpn1}) is only needed in the last step of the iteration (unless one is interested in the intermediate values for $p$ themselves). Indeed, in the examples presented in this paper not only the intermediate $p$-values but also the final momentum is irrelevant. This means for the calculation of conjugate loci one could simply use 
\begin{align*}
0 &= f\left(q_n+h\left(p_n-\frac h2 \nabla f(q_n) \cdot \lambda\right)\right)\\ 
p_{n+1} &= p_n-\frac h2 \nabla f (q_n)\cdot \lambda\\ 
q_{n+1} &= q_n+h p_{n+1}.
\end{align*}
\end{remark}

The derivative $D\Psi_h$ of the time-$h$-map (including the projection step) can be obtained by evaluating the following formulas. We interpret the vectors $q_n$, $p_n$ and the gradient vectors $\nabla f(q_n)$, $\nabla_q \lambda$, etc.\ as column vectors such that, for instance, $\nabla f(q_n) (\nabla_q \lambda)^T$ denotes a dyadic product. The symbol $I$ refers to an $n$-dimensional identity matrix.

\begin{align*}
\nabla_q \lambda &= \frac{-\lambda \Hess f(q_n) n + \frac 2 {h^2}n}{\langle n, \nabla f(q_n)\rangle}\\
\nabla_p \lambda &= \frac{2n}{h \langle n, \nabla f(q_n)\rangle}\\
D_q\left( p_{n+\frac 12} \right) &= -\frac h2 \left( \Hess f(q_n) \lambda + \nabla f(q_n) (\nabla_q \lambda)^T\right)\\
D_p\left( p_{n+\frac 12} \right)&= I-\frac h2 \nabla f(q_n) \nabla_p \lambda^T\\
D_q(q_{n+1}) &= I + h D_q\left( p_{n+\frac 12} \right)\\
m&= \frac{\nabla f(q_n)}{\|\nabla f(q_n)\|}\\
D_p(q_{n+1}) &= h D_p\left( p_{n+\frac 12} \right)\\
D_q (n) &= \frac 1 {\| \nabla f(q_{n+1})\|} \left( \Hess f(q_{n+1}) D_q (q_{n+1})
-n n^T \Hess f(q_{n+1}) D_q (q_{n+1})\right)\\
D_p (n) &= \frac 1 {\| \nabla f(q_{n+1})\|} \left( \Hess f(q_{n+1}) D_p (q_{n+1})
-n n^T \Hess f(q_{n+1}) D_p (q_{n+1})\right)\\
D_q(p_{n+1})&= D_q\left( p_{n+\frac 12} \right)
-\langle n,p_{n+\frac 12}\rangle D_q(n)
-n  p_{n+\frac 12}^T D_q(n)
-nn^T D_q\left( p_{n+\frac 12} \right)\\
D_p(p_{n+1})&= D_p\left( p_{n+\frac 12} \right)
-\langle n,p_{n+\frac 12}\rangle D_p(n)
-n  p_{n+\frac 12}^T D_p(n)
-nn^T D_p\left( p_{n+\frac 12} \right)
\end{align*}

When the time-$h$-map $\Psi_h$ is iterated $N$-times to obtain the numerical time-$Nh$-map $\Phi$, the derivatives can be updated as follows
\begin{align*}
V^{(0)}&= I\\
V^{(n)}&=D\Psi_h(q_{n-1},p_{n-1})V^{(n-1)}.
\end{align*}
We obtain the derivatives as $D\Phi (q_0,p_0)=V^{(N)}$.
We refer to this 1-jet version of the RATTLE method applied to a hypersurface as {\em jet-RATTLE}.

%% file: RK_break_numerical_experiment.tex
\section{Breaking of an elliptic umbilic using a non-symplectic integrator -- numerical example }\label{app:breakD4RK}

Let us compare the capturing of an elliptic umbilic bifurcation $D_4^-$ by the second order accurate symplectic St\"ormer-Verlet method to a non-symlectic method of the same order of accuracy.
For this we consider the Dirichlet problem for the H{\'e}non-Heiles-type Hamiltonian system described in section \ref{subsec:HenonHeilsD4}. In contrast to the numerical experiment described in \ref{subsec:HenonHeilsD4}, we reduce the number of time-steps from $N=10$ to $N=5$ and perturb the Hamiltonian from \eqref{eq:HenonHeilsHam} with the extra term $0.01 y_2 \sin(y_1)$ to
\[
H(x,y)
=\frac 12 \|y\|^2+\frac 12 \|x\|^2-10\left(x_1^2 x_2-\frac{x_2^3}3 \right)+0.01 y_2 \sin(y_1).
\]
In the considered boundary value problem
\[
(x^1,x^2)=(0,(x^\ast)^2), \quad (\phi^{X^1},\phi^{X^2})=((X^\ast)^1,(X^\ast)^2),
\]
where $\phi^X=(\phi^{X^1},\phi^{X^2})$ are the $x$-components of the Hamiltonian flow map at time 1 and $(x^\ast)^2$, $(X^\ast)^1$, $(X^\ast)^2$ the parameters of the problem, the level bifurcation set is locally given by
\[
\mathcal B = \{ (x^2,\phi^X(0,x^2,y_1,y_2)) \; | \; \det D_y \phi^X(0,x^2,y_1,y_2)=0, (0,x^2,y_1,y_2)\in U \},
\]
for a subset $U \subset \R^4$ of the phase space. The level bifurcation set $\mathcal B$ can be obtained from 
\[
B=\{ (x^2,y_1,y_2) \; | \; \det D_y \phi^X(0,x^2,y_1,y_2)=0, (0,x^2,y_1,y_2)\in U\}.
\]
using $\phi^X$. The figures \ref{fig:LPHenonHeilsBB} and \ref{fig:RKHenonHeilsBB} show plots of the sets $B$ (to the left) and $\mathcal B$ (to the right). For figure \ref{fig:LPHenonHeilsBB} the flow $\phi$ was approximated with the symplectic St\"ormer-Verlet method. We see an elliptic umbilic bifurcation, where three lines of cusps merge in one singular point marked by an asterisk. Its position in the phase portrait of the numerical flow can be calculated as a root of $(x^2,y_1,y_2) \mapsto D_y \phi^X(0,x^2,y_1,y_2)$. For figure \ref{fig:RKHenonHeilsBB} the flow $\phi$ was approximated with the explicit midpoint rule (RK2), which is a second-order non-symplectic Runge-Kutta method. While the sheets in the plot for $B$ still approach a singular point they cannot reach it and connect in a whole circle rather than a singular point.
In the level bifurcation set this has the effect that we do not obtain an elliptic umbilic bifurcation but three lines of cusp bifurcations which fail to merge in an umbilic point.


To know which parts of the bifurcation diagrams are to be compared, the calculations were first done to high accuracy such that the bifurcation diagrams obtained by the St\"ormer-Verlet method and by RK2 were close. Then the step sizes were increased gradually and the movement of the singular point where the matrix $D_y \phi^X(0,x^2,y_1,y_2)$ is near the zero matrix was tracked in both simulations.

\begin{figure}
\begin{center}
\includegraphics[width=0.48\textwidth]{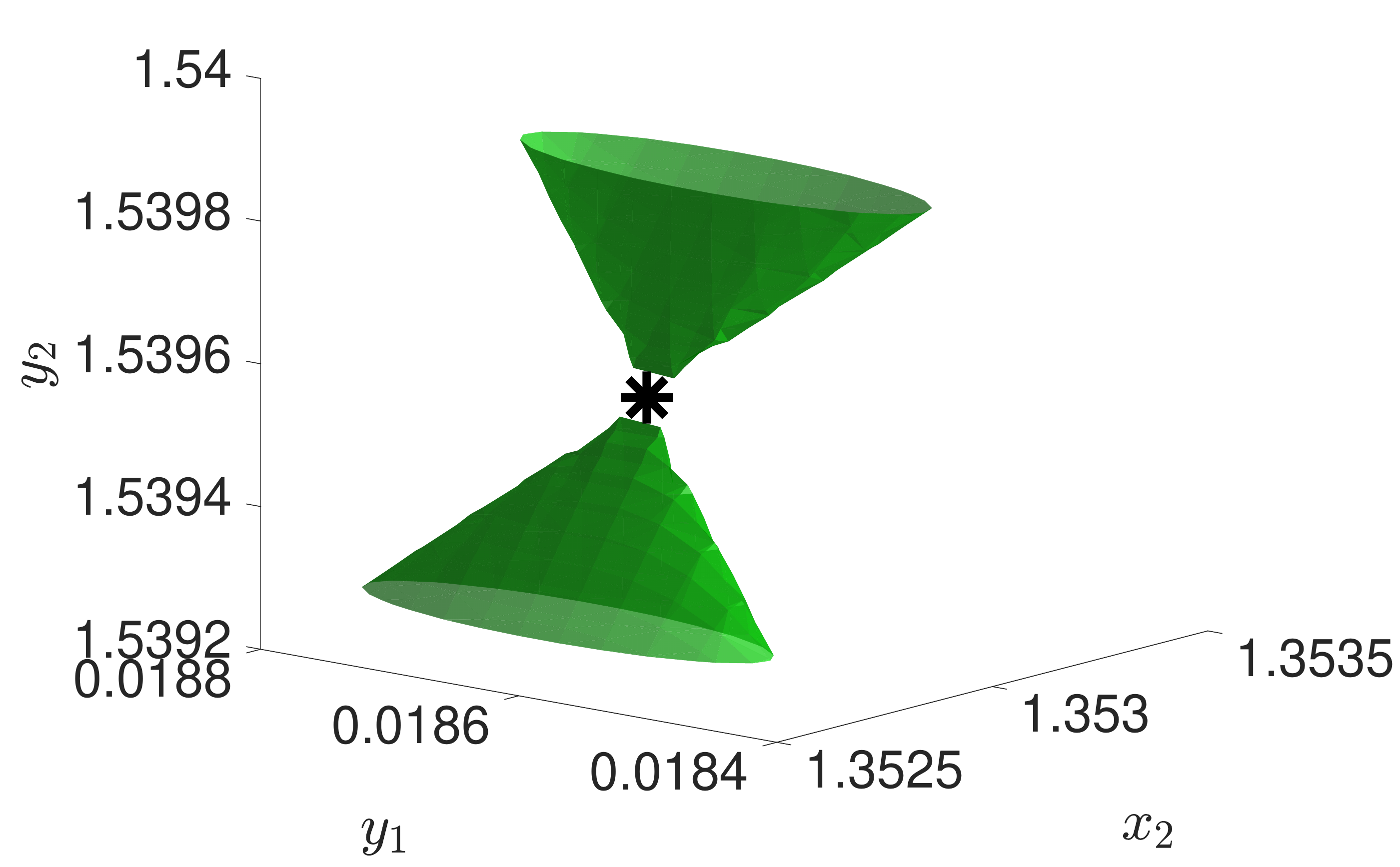}
\quad
\includegraphics[width=0.48\textwidth]{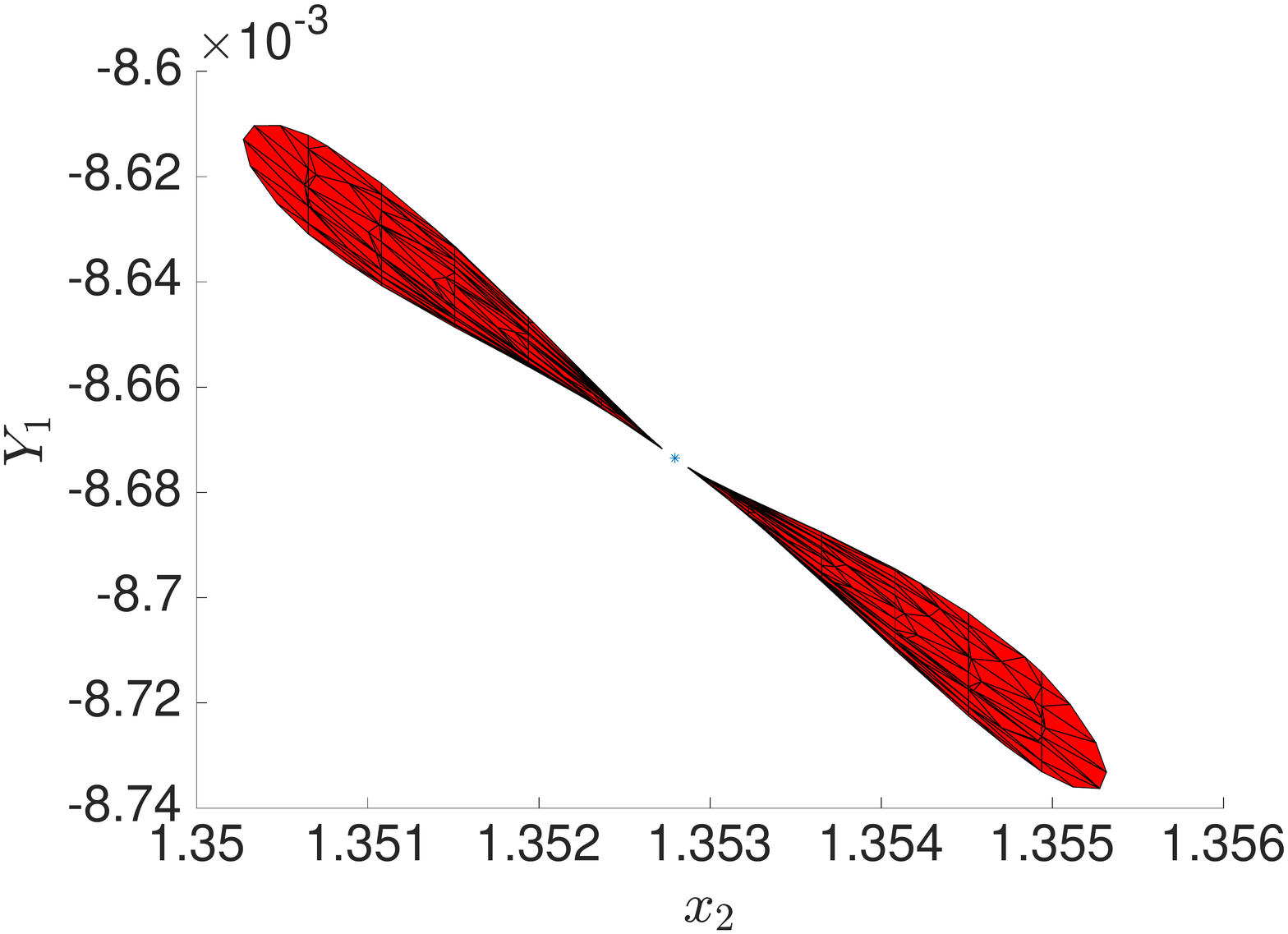}
\end{center}
\caption{Resolving an elliptic umbilic bifurcation $D^-_4$ with the symplectic St\"ormer-Verlet method. The plot to the left shows the set $B$ and the plot to the right shows the level bifurcation set $\mathcal B$.}\label{fig:LPHenonHeilsBB}
\end{figure}

\begin{figure}
\begin{center}
\includegraphics[width=0.48\textwidth]{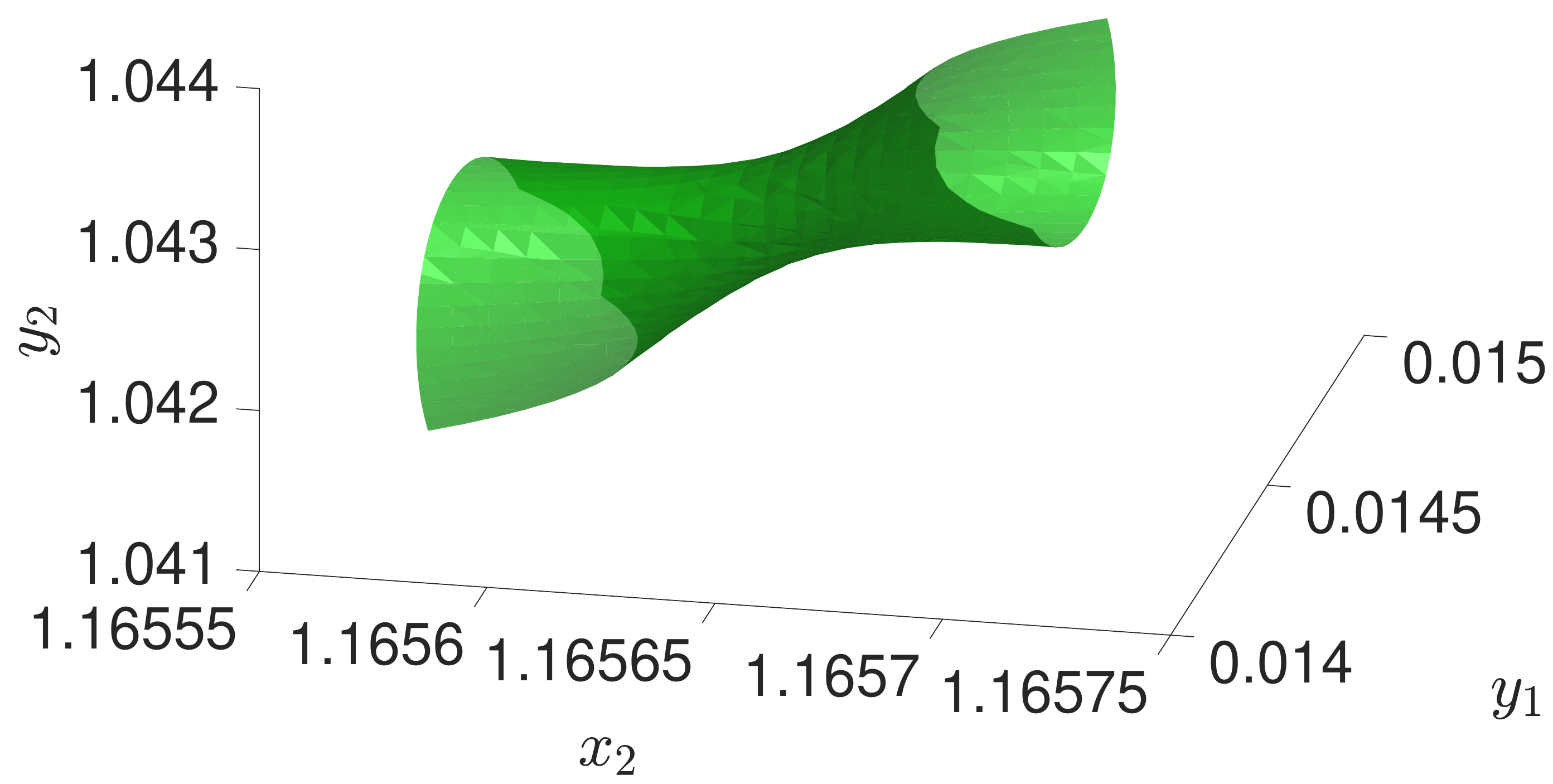}
\quad
\includegraphics[width=0.48\textwidth]{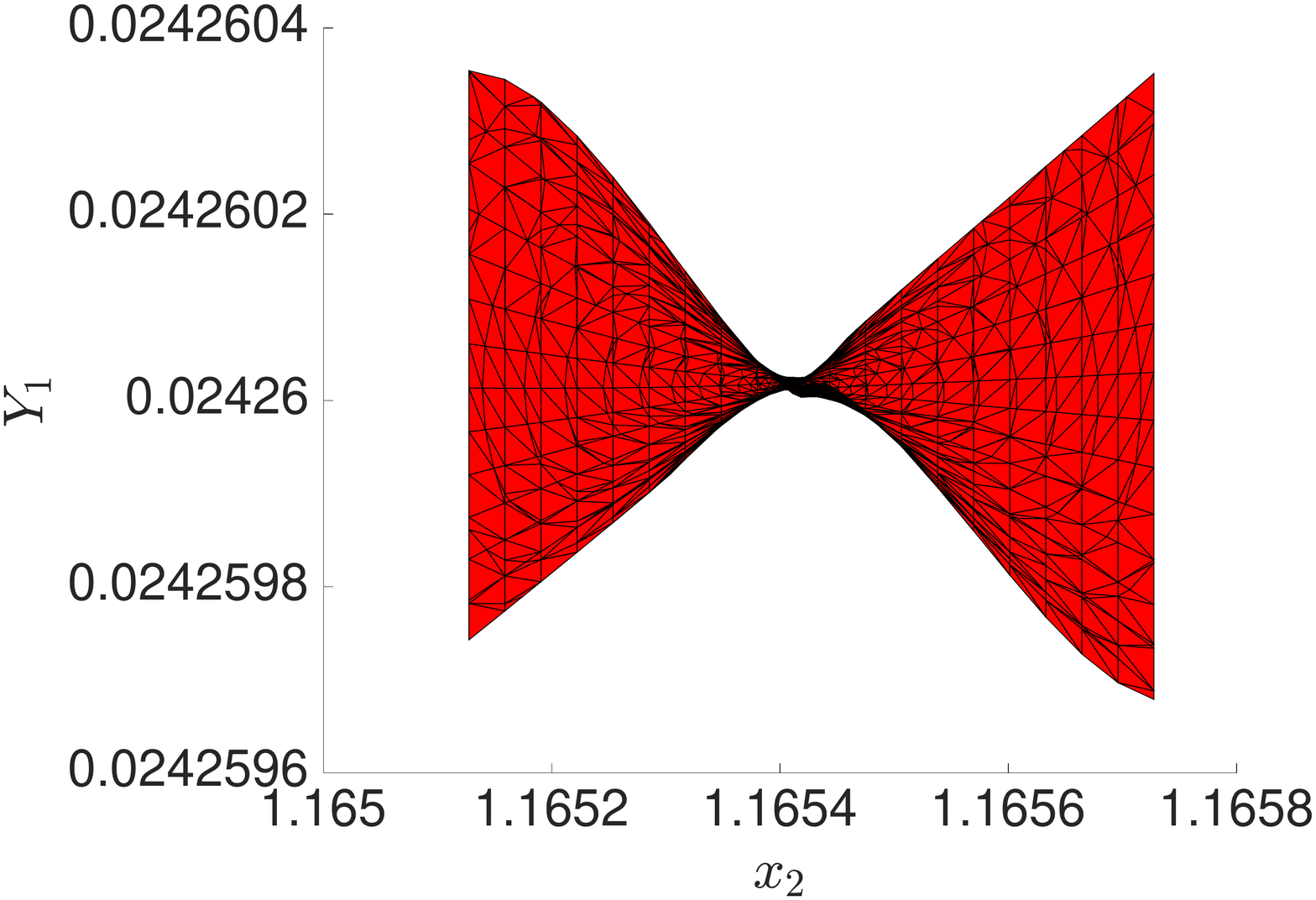}
\end{center}
\caption{Resolving an elliptic umbilic bifurcation $D^-_4$ with the non-symplectic second order Runge-Kutta method. The plot to the left shows the set $B$ and the plot to the right shows the level bifurcation set $\mathcal B$. The set was rotated around the $Y_2$ axis by $0.0271 \mathrm{rad}$ in order to allow for a convenient rescaling of the axes. Instead of an elliptic umbilic bifurcation there are three lines of cusp bifurcations which fail to merge. }\label{fig:RKHenonHeilsBB}
\end{figure}